\documentclass{amsart}
\usepackage{graphicx}%
\usepackage{multirow}%
\usepackage{amsmath,amssymb,amsfonts,stmaryrd}%
\usepackage{amsthm}%
\usepackage{mathrsfs}%
\usepackage[title]{appendix}%
\usepackage{xcolor}%
\usepackage{textcomp}%
\usepackage{manyfoot}%
\usepackage{booktabs}%
\usepackage{algorithm}%
\usepackage{algorithmicx}%
\usepackage{algpseudocode}%
\usepackage{listings}%

\usepackage{hyperref}
\usepackage[capitalize]{cleveref}

\theoremstyle{thmstyleone}%
\newtheorem{theorem}{Theorem}
\newtheorem{proposition}[theorem]{Proposition}%

\newtheorem{corollary}[theorem]{Corollary}

\theoremstyle{thmstyletwo}%
\newtheorem{remark}{Remark}%

\theoremstyle{thmstylethree}%

\newcommand\cexp{\mathop{\mbox{$\bcc$-$\mathrm{exp}$}}}
\newcommand{\sfT}{\mathsf{T}}
\newcommand{\GammaE}{\Gamma^{\cE}}

\newcommand{\ttX}{\mathtt{X}}
\newcommand{\ttY}{\mathtt{Y}}
\newcommand{\tth}{\mathtt{h}}
\newcommand{\ttg}{\mathtt{g}}
\newcommand{\Two}{\mathbb{II}}

\newcommand{\R}{\mathbb{R}}
\newcommand{\N}{\mathbb{N}}
\newcommand{\E}{\mathbb{E}}

\newcommand{\fD}{\mathfrak{D}}

\newcommand{\rD}{\mathrm{D}}
\newcommand{\rC}{\mathrm{C}}
\newcommand{\rH}{\mathrm{H}}
\newcommand{\rS}{\mathrm{S}}

\newcommand{\grad}{\mathrm{grad}}
\newcommand{\hess}{\mathrm{hess}}

\newcommand{\ttq}{\mathtt{q}}

\newcommand{\barx}{\bar{x}}

\newcommand{\Z}{\mathbb{Z}}

\newcommand{\cD}{\mathcal{D}}
\newcommand{\bcD}{\bar{\mathcal{D}}}
\newcommand{\rcD}{\mathring{\mathcal{D}}}
\newcommand{\cE}{\mathcal{E}}

\newcommand{\ft}{\mathfrak{t}}

\newcommand{\tphi}{\varphi}

\newcommand{\cM}{\mathcal{M}}
\newcommand{\bcM}{\bar{\mathcal{M}}}

\newcommand{\cN}{\mathcal{N}}
\newcommand{\hcN}{\hat{\mathcal{N}}}
\newcommand{\bcN}{\bar{\mathcal{N}}}

\newcommand{\bcc}{\mathsf{c}}
\newcommand{\barbcc}{\bar{\mathsf{c}}}

\newcommand{\bu}{\mathsf{u}}
\newcommand{\bs}{\mathsf{s}}
\newcommand{\bdomg}{\mathring{\omega}}
\newcommand{\bomg}{\bar{\omega}}
\newcommand{\bOmg}{\bar{\Omega}}
\newcommand{\opt}{\mathsf{opt}}
\newcommand{\bdeta}{\mathring{\eta}}
\newcommand{\baeta}{\bar{\eta}}

\newcommand{\bdxi}{\mathring{\xi}}
\newcommand{\bxi}{\bar{\xi}}
\newcommand{\brho}{\bar{\rho}}

\newcommand{\tq}{\mathrm{q}}

\newcommand{\cross}{\mathsf{cross}}
\newcommand{\brD}{\bar{\mathrm{D}}}
\newcommand{\bdD}{\mathring{\mathrm{D}}}

\newcommand{\bT}{\mathbf{T}}
\newcommand{\tT}{\mathtt{T}}

\newcommand{\tone}{\mathbf{1}}

\newcommand{\sfg}{\mathsf{g}}

\newcommand{\ptt}{\mathtt{p}}

\DeclareMathOperator{\sign}{sign}

\DeclareMathOperator{\Rc}{R}
\DeclareMathOperator{\diag}{diag}

\DeclareMathOperator{\arcsinh}{arcsinh}

\begin{document}

\title[Zero and nonnegative MTW tensor c-divergence]{Families of costs with zero and nonnegative MTW tensor in optimal transport and the c-divergences}


\author{Du Nguyen}
\address{Independent, Darient, CT, USA}
\email{nguyendu@post.harvard.edu}


\begin{abstract}We study the information geometry of $\bcc$-divergences from families of costs of the form $\mathsf{c}(x, \barx) =\mathsf{u}(x^{\mathfrak{t}}\barx)$ through the optimal transport point of view. Here, $\mathsf{u}$ is a scalar function with inverse $\mathsf{s}$, $x^{\ft}\barx$ is a nondegenerate bilinear pairing of vectors $x, \barx$ belonging to an open subset of $\mathbb{R}^n$. We compute explicitly the MTW tensor (or cross curvature) for the optimal transport problem on $\mathbb{R}^n$ with this cost. The condition that the MTW-tensor vanishes on null vectors under the Kim-McCann metric is a fourth-order nonlinear ODE, which could be reduced to a linear ODE of the form $\mathsf{s}^{(2)} - S\mathsf{s}^{(1)} + P\mathsf{s} = 0$ with constant coefficients $P$ and $S$. The resulting inverse functions include {\it Lambert} and {\it generalized inverse hyperbolic\slash trigonometric} functions. The square Euclidean metric and $\log$-type costs are equivalent to instances of these solutions. The optimal map may be written explicitly in terms of the potential function. For cost functions of a similar form on a hyperboloid model of the hyperbolic space and unit sphere, we also express this tensor in terms of algebraic expressions in derivatives of $\mathsf{s}$ using the Gauss-Codazzi equation, obtaining new families of strictly regular costs for these manifolds, including new families of {\it power function costs}. We express the divergence geometry of the $\mathsf{c}$-divergence in terms of the Kim-McCann metric, including a $\mathsf{c}$-Crouzeix identity and a formula for the primal connection. We analyze the $\sinh$-type hyperbolic cost, providing examples of $\mathsf{c}$-convex functions, which are used to construct a new \emph{local form} of the $\alpha$-divergences on probability simplices. We apply the optimal maps to the problem of sampling fat-tailed distributions, in particular, to sample the multivariate $t$-distribution.
\end{abstract}



\maketitle

\section{Introduction}
In recent years, it has emerged that there is a close connection between information geometry \cite{Amari} and the geometry of optimal transport \cite{KimMcCann,KhanZhang,AyOtto}. An important theme in information geometry is the geometry of divergences, where important statistical concepts are shown to have geometric interpretations. Two important concepts in optimal transport are $\bcc$-convexity and the optimal map; both are important in applications. Both are used to define the $\bcc$-divergences \cite{WongYang,WongInfo,WongZhangIEEE}, which generalize the classical Bregman divergence and the $\alpha$-divergence. Naturally, this suggests a study of $\bcc$-divergences for other tractable costs could potentially lead to new insights into this important concept and would lead to new applications.

With that in mind, in this paper, we obtain simple \emph{regularity criteria} for a general family of costs, regularity means smoothness of the optimal map, which appears in the $\bcc$-divergence. Focusing on the $\bcc$-divergence, we obtain general results for the dualistic geometry, in particular for a subfamily of \emph{hyperbolic costs}. Our analysis leads to a new divergence inequality. We propose a \emph{mirror sampling algorithm} and suggest potential applications in \emph{hyperbolic representation} in machine learning\cite{Alvarez,Idrobo}. Preliminarily, this confirms the earlier suggestion that investigating alternative costs would lead to further insights and applications in both optimal transport and information geometry.

  Let $\rho$ and $\brho$ be Borel probability measures on manifolds $\cM$ and $\bcM$, and let $\bcc: \cM\times \bcM\to \R\cup \{\infty\}$ be a cost function, the optimal transport problem finds the Borel map $\bT$ minimizing the total cost $\int_{\cM}\bcc(x, \tT(x))d\rho(x)$ among all Borel maps $\tT$ pushing $\rho$ to $\brho$.

  For the square Euclidean distance cost, a general existence result was proved in \cite{Brenier}, then extended to more general costs, including on manifolds, see \cite{KimMcCann} for a review. In general, the optimal map $\bT$ is not even continuous. Through the pioneering work of \cite{MTW} and subsequent results (\cite{Loeper,KimMcCann}, and others), the regularity of $\bT$ is tied to the Ma-Trudinger-Wang tensor (cost sectional curvature or cross-curvature), which is best expressed as a type of sectional curvature of a semi-Riemannian metric on a subdomain $\cN\subset \cM\times \bcM$, called the Kim-McCann metric \cite{KimMcCann}, (reviewed in \cref{sec:review}). Another geometric approach to optimal transport is studied in \cite{KhanZhangPAA}, where the authors study K{\"a}hler geometry of optimal transport problems. In addition to allowing the extension of the problem to manifolds, the expression of the MTW tensor in terms of a semi-Riemannian metric allows us to use computational tools from differential geometry. A cost is called weakly regular, or satisfying A3w, if the MTW tensor is nonnegative on null vectors (zero length in the Kim-McCann metric), and strictly regular, or satisfying A3s, if it is positive on nonzero null vectors. Under additional conditions, in \cite{LoeperActa}, the author shows A3w is necessary and sufficient for the regularity of $\bT$, while A3s allows stronger estimates. See \cite{LiuHolder,GuillenMcCann,FigalliKimMcCann} for H{\"o}lder continuity results.

  Costs satisfying the MTW conditions A3w and A3s are important because of the regularity of the optimal map. Here, we produce several new examples, including those with zero MTW tensors, a result we believe to be significant both theoretically and practically. The work \cite{WongYang} shows that we can construct a divergence from convex potentials of a transportation problem. Among our zero MTW-tensor costs, the hyperbolic family discussed in \cref{sec:optimalexample} possesses a collection of $\bcc$-convex potentials that could be useful in various problems. The associated divergence provides a new notion of distance on the $p$-ball $|x|_p < R$ for $R> 0$, and also allows us to construct divergences on probability simplices, converging to the well-known Kullback-Leibler divergence in statistics and machine learning. A similar consideration could apply to other manifolds.
  
In general, the MTW tensor is quite lengthy to compute, but there are several examples of simple functions with zero MTW tensor on $\R^n$ \cite{LeeMcCann,LeeLi}, the square distance cost, and the $\log$-type cost. Motivated by these examples, we consider cost functions of the form $\bcc(x, \barx) = \bu(x^{\ft}\barx)$ on a subset of $\R^n\times\R^n$. Here, we fix a nondegenerate symmetric matrix $A_{\ft}$ and define $x^{\ft}\barx:=x^{\sfT}A_{\ft}\barx$, $\sfT$ is the transpose operator, $\bu$ is a scalar, monotonic, $C^4$-function from a real interval to $\R$ with inverse $\bs$. We write $\bs_i$ for the $i$-th derivative $\bs^{(i)}$ of $\bs$. With the Kim-McCann metric $\langle , \rangle_{KM}$
defined in \cref{eq:KimMcCann} on a submanifold $\cN\subset \cM\times\cM\subset\R^n\times\R^n$, where $\bs_1\circ \bu(x^{\ft}\barx)$ and $(\bs_1^2-\bs\bs_2)\circ \bu (x^{\ft}\barx)$ do not vanish, the first main result of the paper is a formula for its MTW tensor
\begin{theorem}\label{theo:crossMain}For $\tq=(x, \barx)\in \cN$ and for a vector $\bdomg=(\omega, \bomg)\in \R^n\times\R^n$, set $u = \bu(x^{\ft}\barx)$ and $s_i = \bs_i(u)$. Then the cross-curvature (MTW-tensor) of the cost $\bcc: (x, \barx) \mapsto \bu(x^{\ft}\barx)$ is given by
    \begin{equation}\begin{gathered}
      \cross(\bdomg) = \frac{(s_1^2 - s_0s_2)s_4 + s_0s_3^2 - 2s_1s_2s_3 + s_2^3}{2(s_1^2-s_0s_2)s_1^5}(\omega^{\ft}\barx\bomg^{\ft}x)^2 \\
      - \frac{2(s_1s_3 - s_2^2)}{s_1^4}(\omega^{\ft}\barx\bomg^{\ft}x)\langle \bdomg, \bdomg\rangle_{KM} \
      + \frac{s_2}{s_1}\langle \bdomg, \bdomg\rangle_{KM}^2.\label{eq:crossRn}
      \end{gathered}
    \end{equation}
Thus, if $\bdomg$ is a null vector ($\langle \bdomg, \bdomg\rangle_{KM}=0$), then the cross curvature is nonnegative if and only if $\frac{(s_1^2 - s_0s_2)s_4 + s_0s_3^2 - 2s_1s_2s_3 + s_2^3}{2(s_1^2-s_0s_2)s_1^5}\geq 0$. In particular, if $\bs=\bs_0$ satisfies the ODE
    \begin{equation}
      (\bs_1^2 - \bs_0\bs_2)\bs_4 + \bs_0\bs_3^2 - 2\bs_1\bs_2\bs_3 + \bs_2^3 =0\label{eq:zeroODE}
    \end{equation}
then the cross-curvature is identically zero on null-vectors.
\end{theorem}

The scalar function in \cref{eq:zeroODE} determines the regularity condition of the cost, (compared with the three functions for costs of the form $\Phi(|x-\barx|)$ \cite{LeeLi} for a scalar function $\Phi$). Here, $\bs$ is evaluated at all $u$ in the range of $\bu(x^{\ft}\barx)$ for $(x, \barx)$ in $\cN$, assuming the denominator is nonzero (required for the metric to be well-defined). The vanishing of the MTW tensor is thus a fourth-order nonlinear ODE. It turns out this ODE could be reduced to a second-order linear ODE with closed-form solutions
  \begin{theorem}\label{theo:sol}For the fourth-order ODE in \cref{eq:zeroODE} with initial conditions $\bs_i(u_0) = s_i$ for $0\leq i \leq 3$ for $u_0\in \R$  with $s_1\neq 0, s_1^2-s_0s_2\neq 0$, define
    \begin{equation}S: = \frac{\bs_1\bs_2 - \bs\bs_3}{\bs_1^2-\bs\bs_2}, \quad P := \frac{\bs_2^2-\bs_3\bs_1}{\bs_1^2-\bs\bs_2}.\label{eq:SPDelta}
    \end{equation}
If $\bs$ is a solution of \cref{eq:zeroODE}, then $S$ and $P$ are constants, and $\bs$ also satisfies
\begin{equation}\bs_2 - S\bs_1 + P\bs = 0.\label{eq:secondord}
\end{equation}
Thus, \cref{eq:zeroODE} has a unique solution of one of the following forms in an interval containing $u_0$ depending on $\Delta =\Delta(u_0) := S^2 - 4P$
 \begin{gather}\bs(u) = p_0e^{p_1u} + p_2e^{p_3u} \text{ with }p_0p_2(p_1-p_3)\neq 0, p_3 > p_1\text{ if }\Delta > 0\label{eq:Sol1},\\
         \bs(u) = (a_0+a_1u)e^{a_2u} \text{ with }a_1\neq 0 \text{ if }\Delta = 0\label{eq:Sol3},\\
      \bs(u) = b_0e^{b_1u}\sin(b_2u + b_3) \text{ with }b_0>0,b_2 > 0 \text{ if }\Delta < 0 \label{eq:Sol2}.
    \end{gather}
For \cref{eq:Sol1}, $p_1$ and $p_3$ are roots of $z^2 -S(u_0)z + P(u_0)$, for \cref{eq:Sol2}, $b_1 \pm \sqrt{-1}b_2$ are roots of the same equation. If $\Delta=0$ then $a_2$ in \cref{eq:Sol3} is $\frac{S}{2}$. The remaining coefficients could be determined once these roots are found.
  \end{theorem}
  In these three families of solutions, the classical example \cite{Brenier} $\bcc(x, \barx) = - x^{\ft}\barx$, corresponds to $\bs(u) = - u$, the well-known $\log$\slash reflector antenna cost $\bcc(x, \barx) = -\log(1\pm x^{\ft}\barx)$ corresponds to $\bs(u) = \mp(1 - e^{- u})$. The corresponding families of $\bu$ contain Lambert and inverse hyperbolic functions, and other transcendental functions. Even when the coefficient of $(\omega^{\ft}\barx\bomg^{\ft}x)^2$ in (\ref{eq:crossRn}) is positive, $\omega^{\ft}\barx$ or $\bomg^{\ft}x$ could still be zero, so this cost is A3w but not A3s on $\R^n$ for $n\geq 2$.

As we will introduce many symbols and notations in this paper, we summarize them in \cref{tbl:notations} for the reader's convenience.
\begin{table}
  \resizebox{\textwidth}{!}{
    \begin{tabular}{ll}
\toprule      
Symbol & Meaning \\
\midrule
$\bs, \bu$, $\bs_i, \bu_i, s_i, u_i$ & A scalar monotonic function, its inverse, its derivatives
$\bs^{(i)}$ and $\bu^{(i)}$, and their values\\
$\bcc$ & Transport cost, with $\bcc(x, \barx) = \bu(s) = \bu(x^{\ft}\barx)$ for $s=x^{\ft}\barx$\\
$\ft$& $x^{\ft}y = x^{\sfT}A_{\ft}y$ for $x, y\in \R^n$ for a symmetric matrix $A_{\ft}\in\R^{n\times n}$\\
$x, \barx, \ttq$ & Typical points in source $x\in\cM$, target $\barx\in\bcM$ and total manifold $\ttq=(x, \barx)\in\cN\subset\cM\times \bcM$\\
$\omega, \bomg, \bdomg$ & Source and target components of a tangent vector $\bdomg=(\omega, \bomg)\in T_{\ttq}\cN$\\
$\rD, \brD, \bdD$ & Directional derivatives in source ($x$), target ($\barx$) and total manifold variable $(\ttq)$\\
$\cD\bcc, \bcD\bcc, \rcD\bcc$ & Partial and total gradients of $\bcc$ in source ($x$), target ($\barx$) and total manifold variable $(\ttq)$\\
$\nabla, \Gamma$ & A connection and the corresponding Christoffel functions\\
$\Pi, \Pi'$ & The metric projection to the tangent bundle\slash tangent space, and its diffential (Jacobian)\\
$\Two$ & The second fundamental form as an operator\\
$\bT, \bT_{\phi}, x_{\bT}$ & The optimal map corresponding to a potential $\phi$, $x_{\bT} = \bT (x)$\\
$\nabla^{1}, \nabla^{-1}, \Gamma^{1}, \Gamma^{-1}$ & The primal and dual connections and Christoffel functions of a divergence $\fD$
\\
\bottomrule
\end{tabular}
  }
\caption{Summary of notations.}
\label{tbl:notations}  
\end{table}


Beyond the Euclidean space, we consider the case when $\cM$ is a sphere $\rS^n$ or a hyperboloid model of hyperbolic geometry $\rH^n$, considered as subsets of $\R^{n+1}$. In these cases, the induced pairing by $\ft$ is Riemannian, and the MTW condition reduces (\cref{theo:R1234}) to the nonnegativity of {\it three scalar expressions} in $\bs$. In these cases, the Riemannian distance on $\cM$ is a function of $x^{\ft}\barx$, thus, $\bcc$ could be expressed in terms of the Riemannian distance, the point of view considered in \cite{LeeLi}, using Jacobi fields \cite{LeeMcCann}. Their criteria should be equivalent to ours when restricted to these manifolds. It is easy to construct costs that satisfy A3w(s) for a range of $x^{\ft}\barx$, it is harder to do so for a comprehensive range. With the help of symbolic differentiation tools, we have the following two theorems
\begin{theorem}\label{theo:Hn1}For $\rH^n$, if the parameters satisfy the specified conditions below then the cost function satisfies A3s for all $(x, \barx)\in \rH^n\times\rH^n$:\hfill\break
\underline {1. Generalized hyperbolic, sinh-like:} $\bs(u) = p_0e^{p_1u} + p_2e^{p_3 u}$ with $p_0 < 0, p_1 < 0, p_2 > 0, -p_1 \geq p_3 \geq 0$.\hfill\break
\underline {2. Generalized hyperbolic, one-side range:} $\bs(u) = p_0e^{p_1u} + p_2e^{p_3 u}$ with
    $ p_0 < 0, p_1 < 0, p_2 > 0, p_1 < p_3 < 0$, with range $u \leq u_c = \frac{1}{p_1-p_3}\log\frac{-p_2p_3}{p_0p_1}$. \hfill\break
\underline {3. Lambert:} $\bs(u) = (a_0+a_1u)e^{a_2u}, a_1 > 0, a_2 <0$ with range $u\leq u_c = -\frac{a_0a_2+a_1}{a_1a_2}$ and $\bcc(x, \barx) = \frac{1}{a_2}W(\frac{a_2\exp{\frac{a_0a_2}{a_1}}}{a_1}x^{\ft}\barx) - \frac{a_0}{a_1}$. Here, $W$ is $W_0$ or $W_1$, a Lambert functions.\hfill\break
\underline {4. Affine:} $\bs(u) = a_0+a_1u, a_1 > 0$, $\bcc(x, \barx) = \frac{1}{a_1}(x^{\ft}\barx - a_0)$.\hfill\break

Additionally, we have the log cases. For $p_1\neq 0$,  $\bcc(x, \barx) = -\frac{1}{|p_1|}\log(\frac{1-x^{\ft}\barx}{|p_0|})$ satisfies A3s and $\bcc(x, \barx) =  \frac{1}{p_1}\log(\frac{-x^{\ft}\barx}{|p_0|})$ satisfies A3w.
\end{theorem}
For cases 1 and 2, $\bu$ could be inverted explicitly in following subcases (note that $x^{\ft}\barx\leq -1$ on the hyperboloid), but in general, we need a numerical solver \hfill\break
\underline {1a. $p_3=-p_1 >0$ in case 1:} $p_0 <0, p_2,p_3 >0$, $\bcc(x, \barx)= \frac{1}{|p_3|}\log\frac{x^{\ft}\barx + \sqrt{|4p_0p_2| + (x^{\ft}\barx)^2}}{2|p_2|}$.\hfill\break
\underline {2a. $p_1=2p_3$ in case 2:} $ p_0 <0, p_2 >0, p_3<0$, $\bcc(x, \barx)=-\frac{1}{|p_3|} \log\frac{|p_2| + \sqrt{-4|p_0|x^{\ft}\barx + p_2^2}}{2|p_0|}$.
We can normalize the parameters by noting that optimality of $|a|\bcc(x, \barx)+b$ is the same as that of $\bcc(x, \barx)$ for $a>0$ and $b\in \R$, thus, we can remove the additive constants. 
\begin{theorem}\label{prop:power}For $\rS^n$, the cost function $\bcc(x, \barx) = \bu(s)$ for  $s= x^{\sfT}\barx$, $\bu: [-1, 1]\to[-2^{\frac{1}{\alpha}}, 0],  \bu(s)  = - (1 + s)^{\frac{1}{\alpha}}$, $\bs(u) = (- u)^{\alpha}-1$ ($u\in [-2^{\frac{1}{\alpha}}, 0]$), for a real number $\alpha \geq 2$, 
  satisfies A3s in $S_* := \rS^n\times \rS^n\setminus \{(x, -x)\;|x\in\rS^n\}$. 
 
  For $\rH^n$, with $\beta\in [1,2]$, the cost $\bcc(x, \barx) = - (-x^{\ft}\barx)^{\beta}$ corresponding to $\bu: (-\infty, -1]\to (-\infty, -1]$ with $\bu(s)  = - (- s)^{\beta}, \bs(u) = - (-u)^{\frac{1}{\beta}}$ is strictly regular.
\end{theorem}

Compare with \cite[Theorem 1.1]{LeeLi}, the $\log$ and affine costs are already known on the hyperboloid. For the sphere, the well-known reflector antenna case \cite{Loeper,WangXJ} could be considered a limit of the (scaled) cost in \cref{prop:power} when $\alpha\to\infty$. This is one new family of costs for the sphere, where the condition of \cite[Theorem 4.1]{LoeperActa} is satisfied, assuring the optimal path avoids the cut locus, implying regularity of the optimal map.

\emph{Previous examples.} In \cite[Section 8]{TrudingerWang09}, the authors have several regularity results for costs $\bcc(x, \barx)=-(1\pm(x-\barx)^2)^{\frac{1}{2}}$, the power cost, and $\bcc(x, \barx) = x^{\sfT}\barx+f(x)g(\barx)$ for two convex functions $f,g$.

In \cite{LeeMcCann}, the authors study costs arising from mechanical systems, and give the MTW tensor in terms of Jacobi fields, the example $\bcc(x, \barx)=x^{\sfT}A_{\ft}\barx$ was found in that paper, which inspired us to look at $\bu(x^{\ft}\barx)$. In \cite{LeeLi},  the authors give regularity criteria for functions of the Riemannian distance on space forms, which, when specialized to the hyperboloid and sphere, should be the same as ours\footnote{We understand (private communication) this was done before the Kim-McCann metric was defined.}. They found examples of $\log$ and linear types. In our approach, we use the fact that the contribution from the second fundamental form is nonnegative in some cases, and we use symbolic tools to experiment with the formulas in terms of $x^{\ft}\barx$. Deformation results were studied in \cite{DelanoeGe}, and also in the cited works. Submersions of products of positively curved manifolds \cite{KimMcCann2012,FiKimMc} give another collection of examples. In \cite{KimKitagawa}, the authors showed an example which is A3w but not A3s. In \cite{vonNessi} the author found new costs as functions of geodesic distance on the sphere.

\emph{Hyperbolic $\sinh$-type cost.} The case $\bs(u) = p_0e^{-ru}+p_2e^{ru}$ with $r>0, p_0 > 0 > p_2$ is probably the most tractable of the new costs. We show an absolutely-homogeneous convex function $\tphi$ restricts to a $\bcc$-convex function in the domain $|x^{\sfT}\grad_{\tphi}(x)|<\frac{1}{r}$ ($\grad_{\tphi}$ is the gradient of $\tphi$). We compute the duals explicitly under certain conditions.

\emph{Divergence.} The relationship between optimal transport and divergence, an important concept in information geometry and learning \cite{Amari,Banerjee}, is clarified in \cite{WongPal,WongYang}, where a divergence is constructed from the optimal map of a convex potential. \emph{A divergence could be considered a normalized cost}, where constant potentials are convex, with the optimal map being the identity map.

Several classical results on convex functions generalize to $\bcc$-convex functions by studying the \emph{optimal map}. If $\bT=\bT_{\phi}$ is the optimal map for the potential $\phi$, and write $x_{\bT}$ for $\bT(x)$, in \cref{prop:cdivmetric}, we show the $\bcc$-hessian of a $\bcc$-convex function, defined as $\hess^{\bcc}_{\phi}(x): \xi\mapsto  \rD_{\xi}\cD\bcc(x, x_{\bT}) +\hess_{\phi}(x)\xi$ in \cref{eq:hesscgeneral} plays the role of the divergence metric tensor. The same proposition also relates $\hess^{\bcc}_{\phi}(x)$ to the optimal map by
$$d\bT(x)= - (\brD_{}\cD\bcc(x, x_{\bT}))^{-1}\hess^{\bcc}_{\phi}(x).$$
This leads to the $\bcc$-Crouzeix formula \cref{eq:crouzeix} relating the differential of the optimal maps of $\bcc$-conjugate functions $    \left((\brD\cD\bcc)^{-1}\hess^{\bcc}_{\phi}(\rD\bcD\bcc)^{-1}\hess^{\barbcc}_{\phi^{\bar{\bcc}}}\right)_{x, x_{\bT}} = I_{\bar{\Omega}}$, the primal connection $\Gamma^1$ in \cref{eq:GammaD1} and the determinant formula \cref{eq:detdT}. For cost functions of the form $\bu(x^{\sfT}\barx)$ in this paper, the determinant of $(-\brD\cD\bcc(x, \bT(x)))^{-1}$ has a simple form $\frac{(-s_1)^{n+2}}{s_1^2-s_2s_0}$. This leads to a simple formula of the determinant term in the Monge-Amp\`{e}re equation, which suggests the \emph{Mirror Langevin} method discussed shortly could be generalized to this context, and it provides the computational framework for our sampling application. While we study the hyperbolic cost more carefully, the treatment of hyperbolic cost also extends to the log-type cost $\frac{\pm 1}{\alpha}\log (1+\alpha x^{\sfT}\barx)$ in \cite{WongInfo,WongZhangIEEE}, for which there is a simple characterization of $\bcc$-convex functions. For sampling applications, we suggest many classical algorithms could be extended if we replace classically convex functions with an appropriate $\bcc$-convex function satisfying MTW regularity conditions, in particular, those satisfying the nonnegativity condition in \cref{theo:crossMain}.

In addition to the background in \cite{Nielsen}, Legendre-type potentials defined in \cref{sec:dualistic} are (informally) potentials with the associated optimal maps sufficiently smooth and bijective in the next result. We now describe the dualistic geometry of the hyperbolic cost $-\frac{\arcsinh (rx^{\sfT}\barx)}{r}$ as the main example of a new $\bcc$-divergence
\begin{proposition}\label{eq:dualistichyper}Let $\phi^{(3)}$ denote the third-order derivative tensor of a $\bcc$-convex potential $\phi$ of Legendre-type. For $x\in\Omega$, let $\ttg = \grad_{\phi}(x)$. Let $\mathtt{h} = \hess_{\phi}(x)$ considered as a matrix. For $\bcc(x, \barx) = -\frac{\arcsinh (rx^{\sfT}\barx)}{r}$, the divergence metric of the $\bcc$-divergence
$$\phi(x) - \phi(x')- \frac{1}{r}\log
  \frac{rx^{\sfT}\grad_{\phi}(x') + \left(r^2(x^{\sfT}\grad_{\phi}(x'))^2
    - r^2((x')^{\sfT}\grad_{\phi}(x'))^2+1\right)^{\frac{1}{2}}}{
    r(x')^{\sfT}\grad_{\phi}(x')+ 1 }
  $$
 in \cref{eq:divergence} is given by the operator $\hess^{\bcc}_{\phi}(x)\xi = \tth\xi + r^2(x^{\sfT}\ttg \ttg^{\sfT}\xi)\ttg$ in \cref{eq:hesschyper}. The dualistic pair of connections are
  \begin{equation}
\begin{gathered}    
  \Gamma^1(x; \xi_1, \xi_2) = -r^2((\ttg ^{\sfT}\xi_1\ttg ^{\sfT}\xi_2)x \
                      + (x^{\sfT}\ttg \ttg ^{\sfT}\xi_2)\xi_1 \
                      + (x^{\sfT}\ttg \ttg ^{\sfT}\xi_1)\xi_2), \\
\Gamma^{-1}(x; \xi_1, \xi_2) =  (\hess^{\bcc}_{\phi}(x))^{-1}\big( \phi^{(3)}(\xi_1, \xi_2)+
                                  r^2(x^{\sfT}\ttg \xi_1^{\sfT}\ttg )\tth\xi_2 \
                                  + r^2(\xi_2^{\sfT}\ttg x^{\sfT}\ttg )\tth\xi_1 \\
                                  + r^2\left(
                                      (1+3r^2(x^{\sfT}\ttg )^2)\xi_1^{\sfT}\ttg \xi_2^{\sfT}\ttg  \
                                      + 2(\xi_1^{\sfT}\tth\xi_2)x^{\sfT}\ttg
                                      + \xi_2^{\sfT}\ttg (x^{\sfT}\tth\xi_1)
                                      + \xi_1^{\sfT}\ttg(x^{\sfT}\tth\xi_2) 
                                      \right)\ttg \big).\label{eq:dualistic}
\end{gathered}
  \end{equation}
  The cubic Amari-Chentsov tensor $\xi_3^{\sfT}\hess^{\bcc}_{\phi}(\Gamma^{-1}(\xi_1, \xi_2)-\Gamma^1(\xi_1, \xi_2))$ is given by
  \begin{equation}\begin{gathered}
      \rC(\xi_1, \xi_2, \xi_3) = \phi^{(3)}(\xi_1,  \xi_2, \xi_3) \
        + r^2\big(
        2x^{\sfT}\ttg(\xi_1^{\sfT}\ttg \xi_2^{\sfT}\tth  \xi_3                
         +\xi_2^{\sfT}\ttg\xi_1^{\sfT}\tth  \xi_3 
         + \xi_3^{\sfT}\ttg\xi_2^{\sfT}\tth  \xi_1)\\
        + (1+6r^2(x^{\sfT}\ttg)^2)\xi_3^{\sfT}\ttg\xi_1^{\sfT}\ttg\xi_2^{\sfT}\ttg
        + \xi_3^{\sfT}\ttg\xi_2^{\sfT}\ttg x^{\sfT}\tth  \xi_1
                + \xi_1^{\sfT}\ttg\xi_2^{\sfT}\ttg x^{\sfT}\tth \xi_3 \
                + \xi_3^{\sfT}\ttg\xi_1^{\sfT}\ttg x^{\sfT}\tth \xi_2\big)
\end{gathered}                
\end{equation}  
  for three vectors $\xi_1, \xi_2, \xi_3\in\R^n$, and it is a symmetric tensor.
\end{proposition}
The curvature of $\Gamma^1$ is also easy to compute, see a discussion in \cref{sec:dualistic}.

\subsection{Applications and perspectives}
We provide an application to sample high-dimensional multivariate distributions, in particular of the multivariate $t$-distribution. However, we will start by discussing several further potential applications. They are beyond the scope of the present article, but we believe following up works would be fruitful.

First, there are important injectivity results associated with regular costs \cite{FigalliKimMcCann}. If we have a bijective optimal map, measures from one space could be pushed forward or pulled back to another. Thus, statistics and sampling from one space could be \emph{mirrored} in the other. Recently, this is applied in the \emph{Mirror Langevin algorithm} \cite{NEURIPS2018_Mirror,pmlrv125_zhang20a,AhnChewi}. The pioneering paper \cite{NEURIPS2018_Mirror} looks at pushforward SDEs from a constrained space to its optimal transport target, called the \emph{mirror}, which is $\R^n$, and provides a state-of-the-art sampling algorithm of the Dirichlet distribution on a probability simplex. The subsequent papers analyze the Riemannian Langevin process with the driving Brownian motion associated with the Hessian metric of the potential of the transport problem. It is well-known that the Hessian metric is the divergence metric of the Bregman divergence. As pointed out above and in \cref{sec:dualistic}, the divergence metric of a $\bcc$-divergence (called the $\bcc$-hessian metric here) is of a similar format and satisfies a \emph{Crouzeix} identity. Thus, the \emph{Mirror Langevin algorithm} could be extended to other $\bcc$-divergences, if we can construct good examples of $\bcc$-convex potentials. This is what we hope to investigate in subsequent work.

Since the $\bcc$-Mirror Langevin method would likely result in a much longer article, to provide an application of the pullback\slash push forward measures of optimal transport, we study an example of mirror sampling of the multivariate $t$-distribution (MVT). We observe, at least empirically, the MVT pullbacks to a measure close to the uniform distribution in an ellipsoid domain of a quadratic potential of a hyperbolic cost. Thus, reversing the approach of Mirror Langevin sampling, we can run Monte Carlo sampling in the ellipsoid domain for fat-tailed distributions. We show this is successful for MVT, using only the expression of the density. This suggests the approach could work for more general distributions, and could be improved if combined with other sampling techniques and exploiting the structure of these distributions.

\emph{Hyperbolic representation and matching.} In machine learning, one method to study hierarchy data sets is to embed the hierarchy data tree in an Euclidean space to introduce a notion of distance between hierarchy trees. Trees and graphs also grow exponentially with the number of nodes, while Euclidean volume only grows polynomially with radius. Several authors (see references in \cite{Alvarez,Idrobo}) suggest that embedding in the hyperbolic space $\rH^n$ would be a natural approach, as the hyperbolic volume also increases exponentially with radius. This approach, called hyperbolic representation, has attracted attention in recent years. Matching or aligning two hierarchies could be done by optimal transport. A naive transport-based approach using the hyperbolic distance could lead to a non-continuous optimal map, as the MTW tensor is proportional to the sectional curvature \cite{LoeperActa} and is thus negative on a hyperbolic space. In \cite{Alvarez,Idrobo}, the authors use the previously discovered regular costs for the hyperbolic space in \cite{LeeLi} to study the matching problem. Thus, it is natural to extend these results to our new regular costs on the hyperbolic space $\rH^n$ in \cref{theo:Hn1,prop:power}. This is a project that we are keen to attempt, but it is beyond the scope of this paper.

We expect the new divergence to be useful not only as a measure of distance for latent states, as mentioned in \cref{sec:latent}, but also as a regularization measure for learning problems \cite{NIPS2013_Cuturi}.

In the next section, we review basic facts about semi-Riemannian geometry, optimal transport, and the Kim-McCann metric. We then prove our main results, first for $\R^n$ then for the sphere and hyperboloid model. We show the form of the optimal map, the first and second-order criteria for optimality, applied to our cost functions. We give families of $\bcc$-convex functions for hyperbolic costs. We then study $\bcc$-divergences, and their dualistic geometry, with examples from the hyperbolic costs. The divergence metric of the dualistic geometry provides a framework to extend the Mirror Langevin method and also allows us to attempt mirror sampling for the multivariate $t$-distribution, our main application. We also discuss numerical aspects of a local $\alpha$-divergence and end the main article with a discussion. In the appendices, we provide a proof of \cref{prop:power}, we list the monotonic ranges for the costs in \cref{theo:sol}, and include another example of a convex potential for hyperbolic costs. In \cite{NguyenMTWGitHub}, we implement the cost functions, potential and optimal map programmatically, and verify some statements and proofs in the paper symbolically and numerically.

\emph{Potential limitation and challenges} The construction of $\bcc$-convex potentials on $\R^n$ and their $\bcc$-convex dual for is presumably the main challenge in applications. Our construction of $\bcc$-convex potential is applicable to absolutely-homogeneous functions. From \cite{WongInfo}, there may be other applicable constructions. Similar to the hyperbolic cost, the cost $\frac{1}{p_3}\log\frac{-p_2 \pm  \sqrt{4p_0x^{\sfT}\barx + p_2^2}}{2p_0}$ (corresponding to $\bs(u) = p_0e^{2p_3u} + p_2e^{p_3u}$) could also give additional tractable convex potentials and optimal maps.

\section{Review: MTW tensor, Kim-McCann metric, and semi-Riemannian geometry}\label{sec:review}
The basic setting of the Kim-McCann metric is an open domain $\cN\subset\cM\times \bcM$ of two manifolds $\cM\times\bcM$ with a $C^4$-cost function $\bcc$. We typically denote the first variable of $\bcc$ by $x$ and second by $\barx$. Denote by $\rD$ the directional derivative (of a scalar or vector-valued function) in $x$, $\brD$ the directional derivative in $\barx$. We denote by $\bdD$ the directional derivative with respect to $\tq = (x, \barx)$. For $x\in\cM$, denote by $\bcN(x)$ the collections of $\barx$ such that $(x, \barx)\in \cN$, and defined $\cN(\barx)$ similarly.

At a point $\tq = (x, \barx)\in\cN$, a tangent vector to $\cN$ will have the form $\bdomg = (\omega, \bomg)$ with $\omega\in T_x\cM, \bomg\in T_{\barx}\bcM$. We use $\mathring{}$ to denote the pair and $\bar{}$ to denote the second coordinate. Assuming the cost $\bcc$ satisfies the conditions
\begin{itemize}
\item[] (A1). The mapping $\barx\mapsto -\rD\bcc(x, \barx)$ from $\bcN(x)\subset\bcM$ to $T_x^*\cM$ is injective. Here, $\rD\bcc(x, \barx)$ is the linear functional on $T_x\cM$ sending $\xi\in T_x\cM$ to $\rD_{\xi}\bcc(x, \barx)$.
\item[] (A2). For $\tq=(x, \barx)\in \cN$, the bilinear pairing $(\xi, \bxi)\mapsto \brD_{\bxi}\rD_{\xi}\bcc$ for $\xi\in T_x\cM, \bxi\in T_{\barx}\bcM$ is nondegenerate.
\end{itemize}  

The Kim-McCann pairing for a cost function $\bcc(\tq) = \bcc(x, \barx)$ is defined for two tangent vectors $\bdomg_1, \bdomg_2$ as
\begin{equation}\langle\bdomg_1, \bdomg_2 \rangle_{KM} = -\frac{1}{2}\left((\rD_{\omega_1}\brD_{\bomg_2}\bcc)(\tq) + (\brD_{\bomg_1}\rD_{\omega_2}\bcc)(\tq)\right)\label{eq:KimMcCann}
\end{equation}
or $\langle \bdomg, \bdomg\rangle_{KM} = - \brD_{\bomg}\rD_{\omega}c$ for $\bdomg\in T\cN$. By (A2), this defines a semi-Riemannian metric on $\cN$. By the fundamental theorem of Riemannian geometry \cite{ONeil1983}, there exists a metric compatible, torsion-free connection $\nabla$ on $\cN$, the Levi-Civita connection.

The connection and curvature are typically given in Einstein summation notation in differential geometry. We will use the notation of operator\slash vector\slash matrix calculus, hoping it is more intuitive for readers using optimal transport in applications. In this setting, consider an open subset of $\R^n$ (a local chart of a manifold is identified with a subset of $\R^n$). The Christoffel symbols $\Gamma^k_{ij}$ ($i, j, k$ runs over a basis) are given in operator form, a {\it Christoffel function $\Gamma: (x, \xi, \eta)\mapsto \Gamma(x; \xi, \eta)$}, a $\R^n$-valued function from $(\R^n)^3$, bilinear in tangent vector variables $\xi, \eta$. Thus, for two vector fields $\ttX, \ttY$ on the manifold $\cN$, considered as vector-valued functions locally, the covariant derivative $\nabla_{\ttX}\ttY$ is given in the form $\bdD_{\ttX}\ttY + \Gamma(\ttX, \ttY)$, where $\bdD_{\ttX}\ttY$ is the coordinate-dependent directional derivative on $\cN$. The connection is torsion free means $\Gamma(\ttX, \ttY)= \Gamma(\ttY, \ttX)$ and metric compatibility means
\begin{equation}\bdD_{\ttX}\langle \ttY, \ttY\rangle_{\sfg} = 2\langle \ttY, \nabla_{\ttX}\ttY\rangle_{\sfg} = 2\langle \ttY, \bdD_{\ttX}\ttY + \Gamma(\ttX, \ttY)\rangle_{\sfg}.\label{eq:metricComp}
\end{equation}
Here, $\sfg$ is a semi-Riemannian pairing, mainly the  Kim-McCann metric in this paper.

When $\cN$ is a (non open) submanifold of lower dimension of a vector space $\cE=\R^n$, we also use the same notation. Here, $\ttX$ and $\ttY$ are identified with $\cE$-valued functions from $\cN$. At $\tq\in\cN$, $\Gamma$, originally defined for two vectors in $T_{\tq}\cN$, is extended linearly to a $\cE\times\cE$-bilinear map to $\cE$, and the covariant derivative $\nabla_{\ttX}\ttY = \bdD_{\ttX}\ttY + \Gamma(\ttX, \ttY)$ evaluates to a vector field of $\cN$ for two vector fields $\ttX, \ttY$.

In operator notation, the curvature in terms of Christoffel functions is given as
\begin{equation}\Rc_{\bdomg_1, \bdomg_2}\bdomg_3 = (\bdD_{\bdomg_1}\Gamma)(\bdomg_2, \bdomg_3)
-(\bdD_{\bdomg_2}\Gamma)(\bdomg_1, \bdomg_3) + \Gamma(\bdomg_1, \Gamma(\bdomg_2, \bdomg_3))- \Gamma(\bdomg_2, \Gamma(\bdomg_1, \bdomg_3)).\label{eq:curveGeneral}
\end{equation}
The derivation is with respect to the manifold variable $\tq$. We mention (but do not use in this paper) that it holds \cite{NguyenGeoglobal} even for non open submanifolds of $\R^n$. We note that this sign convention is opposite of \cite{KimMcCann,ONeil1983}, this is consistent with the cross-curvature formulation below.

To describe the Levi-Civita connection of the Kim-McCann metric and its curvature in this notation, in the case of a subset of $\R^n$, or a local chart, the standard basis of $\R^n$ allows us to define a gradient. It also allows us to define the transpose of any matrix. Gradient and transpose can be defined using any basis of $\R^n$, or relative to any nondegenerate bilinear pairing of $\R^n$, thus, we will fix such a bilinear pairing on $\R^n$, and will use the gradient (called {\it local gradient}) as the main mechanism for index raising. We will soon use this to construct the gradient $\cD$ and $\bcD$ of the cost function. The local pairing\slash gradient is a notational replacement of index notation, and may not be related to $\bcc$. Just as intrinsic expressions are  independent of coordinates, they are independent of local gradients. Let $\cD$ be the gradient with respect to $x$ and $\bcD$ be the gradient with respect to $\barx$ (corresponding to the directional derivatives $\rD$ and $\brD$, respectively). Then $\rD\bcD \bcc: \omega\mapsto \rD_{\omega}\bcD \bcc$ and $\brD\cD \bcc: \bomg\mapsto \brD_{\bomg}\cD \bcc$ are adjoint maps with respect to the local pairing, and the condition (A2) requires they are invertible. The Levi-Civita connection of the Kim-McCann metric is given by \cite[equation 4.1]{KimMcCann}, which is written in operator form as follows
\begin{equation}\Gamma(\bdomg_1, \bdomg_2) =
  ((\rD\bcD \bcc)^{-1}\rD_{\omega_1}\rD_{\omega_2}\bcD \bcc,
  (\brD\cD \bcc)^{-1}\brD_{\bomg_1}\brD_{\bomg_2}\cD \bcc).\label{eq:GammaKM}
\end{equation}
To see this, the first component is written in local coordinates in \cite[equation 4.1]{KimMcCann}
as $\Gamma_{ij}^m = \sum_{\bar{k}}\bcc^{m\bar{k}}\bcc_{\bar{k}ij}$, where in a coordinate system $\{e_i|i=1,\cdots, n\}$ of $\R^n$, write $\partial_i :=\rD_{e_i} =\frac{\partial}{\partial x_i}$ and
$\bar{\partial}_{\bar{j}} :=\brD_{e_j}$, then $\bcc^{mk}$ is the $mk$-entry of the inverse matrix of $\rD\bcD \bcc=(\partial_i\bar{\partial}_{\bar{j}}\bcc)_{i,j=1}^n$, while $\bcc_{\bar{k}ij}$ is defined as $\partial_i\partial_j\bar{\partial}_{\bar{k}}\bcc$, thus, $\Gamma_{ij}^m$ is just 
$(\rD\bcD \bcc)^{-1}\rD_{\omega_1}\rD_{\omega_2}\bcD \bcc$ written in coordinate form. We can check $\Gamma_{\bar{i}\bar{j}}^{\bar{m}} = \sum_{k}\bcc^{\bar{m}k}\bcc_{k\bar{i}\bar{j}}$ represents the second component similarly.

For a tangent vector $\omega\in T\cM$, we write $\omega_{\ast 0}$ for $(\omega, 0)\in T\cN$ and for $\bomg\in T\bcM$, we write $\bomg_{0\ast}$ for $(0, \bomg)$. The cross curvature $\cross(\bdomg)=\langle \Rc_{\omega_{\ast 0},\bomg_{0 \ast}}\bomg_{0 \ast},\omega_{\ast 0} \rangle_{KM}$ is the numerator of the sectional curvature (i.e. without the normalizing denominator). Note that in the the curvature convention of \cite{KimMcCann,ONeil1983}, the cross-curvature would be $
-\langle \Rc_{\omega_{\ast 0},\bomg_{0 \ast}}\bomg_{0 \ast},\omega_{\ast 0} \rangle_{KM}$. This numerator is given by the tensor
\begin{equation}\cross(\bdomg)=\frac{1}{2}\left( \brD_{(\brD\cD \bcc)^{-1}\brD_{\bomg}\brD_{\bomg}\cD \bcc} \rD_{\omega}\rD_{\omega} \bcc
-\rD_{\omega}\brD_{\bomg}\brD_{\bomg}\rD_{\omega} \bcc\right).\label{eq:crosssec}
\end{equation}
In fact, set $\omega = \omega^i \partial_i, \bomg = \bomg^{\bar{j}}\bar{\partial}_{\bar{j}}$, expand $(\brD\cD \bcc)^{-1}\brD_{\bomg}\brD_{\bomg}\cD \bcc=(\bomg^{\bar{j}}\bomg^{\bar{k}}\bcc^{a\bar{f}}\bcc_{a\bar{j}\bar{k}})e_{\bar{f}}$ using the summation convention then \cref{eq:crosssec} expands to
$$\frac{1}{2}\big(\omega^i\omega^l\bomg^{\bar{j}}\bomg^{\bar{k}}\bcc^{a\bar{f}}\bcc_{a\bar{j}\bar{k}}\bcc_{il\bar{f}}\bcc^{a\bar{f}}
-\omega_i\omega_l\bomg_{\bar{j}}\bomg_{\bar{k}} \bcc_{i\bar{j}\bar{k}l}  \big)=\omega^i\omega^l\bomg^{\bar{j}}\bomg_{\bar{k}}\Rc_{i\bar{j}\bar{k}l}
$$
where $\Rc_{i\bar{j}\bar{k}l} = \frac{1}{2}(\sum_{a\bar{f}} \bcc_{il\bar{f}}\bcc^{a\bar{f}}\bcc_{a\bar{j}\bar{k}} -\bcc_{i\bar{j}\bar{k}l})$ is of the opposite sign of \cite[Lemma 4.1]{KimMcCann}, matching its sign convention. This also agrees with the MTW tensor \cite[equation 1.2]{KimMcCann}. The extra terms $\bcc^{\bar{k}e}\bcc^{\bar{l}f}$ there are due to a formulation using the cotangent bundle $T^*\bcM$ instead of $T\bcM$, but the regularity results are unchanged.

We recall the following conditions
\begin{itemize}
\item[] (A3w) $\cross(\bdomg)\geq 0$ for a null vector $\bdomg = (\omega, \bomg)$ (ie $\langle \bdomg, \bdomg\rangle_{KM}=0$),
\item[] (A3s) A3w and equality implies $\omega=0$ or $\bomg=0$.
\end{itemize}  
The results of \cite{MTW,Loeper,KimMcCann}, in particular \cite[Theorem 3.1 and 3.4]{Loeper}, show the regularity of the optimal map is governed by these conditions on the cross curvature. For specific regularity results based on these conditions, please consult the cited papers.
\subsection{The Gauss-Codazzi equation}\label{sec:GaussCodazzi}
The main reference is \cite[Chapter 4]{ONeil1983}. Let $\cN$ be a manifold embedded in a vector space $\cE$, where $\cE$ has a semi-Riemannian pairing $\sfg$, with the Levi-Civita connection $\nabla^{\cE}$  has the Christoffel function $\GammaE$. We will assume the semi-Riemannian pairing restricted to $T\cN$ is not degenerate. In that case, there is a metric-compatible projection $\Pi$ from $\cE$ to $T\cN$ (that means if $\tq\in \cN$ and $\omega\in \cE$ then $\omega-\Pi(\tq)\omega$ is normal to $T_{\tq}\cN$ in the metric $\sfg$). We can consider $\Pi$ a map from $\cN$ to the vector space of operators on $\cE$, satisfying the idempotent condition $\Pi(\tq)^2=\Pi(\tq)$, so the directional derivative
\begin{equation}\Pi'(\tq, \xi):=(\rD_{\xi}\Pi)(\tq)=\lim_{t\to 0}\frac{1}{t}\left(\Pi(\tq+t\xi) - \Pi(\tq) \right)\end{equation}
makes sense for a tangent vector $\xi$, where we extend $\Pi$ to a smooth operator-valued function (but not required to be idempotent outside of $\cN$) on $\cE$ near $\cN$. The directional derivative is independent of the extension and computable by operator calculus.

If $\eta$ is another tangent vector then $\Pi'(\tq;\xi)\eta$ is normal, as $\Pi(\rD_{\xi}\Pi)\eta = (\rD_{\xi}\Pi^2)\eta - (\rD_{\xi}\Pi)\Pi\eta =0$ since $\Pi$ is idempotent and $\eta$ is tangent. The second fundamental form is the tensor
\begin{equation}\label{eq:secform}
  \Two(\tq; \xi, \eta) = \Pi'(\tq, \xi)\eta + (I_{\cE} - \Pi)(\tq)\GammaE(\xi, \eta).
\end{equation}
This follows from the usual definition in \cite[Lemma 4.4]{ONeil1983}, $\Two(\ttX, \ttY) = (I-\Pi)\nabla_{\ttX}^{\cE}\ttY$, where $\ttX, \ttY$ are two vector fields extending $\xi$ and $\eta$. In fact, using $\Pi$, we can extend the tangent vectors $\xi, \eta$ to vector fields $\ttX_{\xi}:\tq_1 \mapsto \Pi(\tq_1)\xi, \ttX_{\eta}:\tq_1\mapsto \Pi(\tq_1)\eta$, $\tq_1\in\cN$
$$\Two(\tq; \xi, \eta) =(I_{\cE} - \Pi)(\tq)\nabla_{\ttX_{\xi}}\ttX_{\eta}= (I_{\cE} - \Pi)(\tq)\rD_{\xi}\ttX_{\eta}
+ (I_{\cE} - \Pi)(\tq)\GammaE(\xi, \eta)$$
then evaluate $\rD_{\xi}\ttX_{\eta} = (\rD_{\xi}\Pi)\eta$, which is normal, $\Pi(\tq)\Pi'(\tq, \xi)\eta=0$. The Gauss-Codazzi equation \cite[Theorem 4.5]{ONeil1983} computes the curvature of $\cN$ from the curvature of $\cE$. A close examination shows the proof of that theorem still works in a more general setting, where the pairing on $\cE$ may be degenerate, but its restriction to $\cN$ is nondegenerate (thus $\cN$ is semi-Riemannian while $\cE$ is not), provided we have a (non-unique) pairing-compatible projection and connection on $\cE$.  We will discuss this in detail when considering the cost $\pm\log(-x^{\ft}\barx)$ on the hyperboloid model.

The Gauss-Codazzi equation implies, says for $\cN\subset \cM\times \bcM\subset \cE^2$  
\begin{equation}\cross_{\cM}(\bdxi) = \cross_{\cE}(\bdxi) + \langle\Two(\xi_{\ast 0},\xi_{\ast 0}), \Two(\bxi_{0\ast},\bxi_{0\ast})\rangle_{KM}-\langle\Two(\xi_{\ast 0},\bxi_{0 \ast}), \Two(\xi_{\ast 0},\bxi_{0\ast})\rangle_{KM}.\label{eq:GaussCod0}
\end{equation}
As we will see, the term $\langle\Two(\xi_{\ast 0},\bxi_{0 \ast}), \Two(\xi_{\ast 0},\bxi_{0\ast})\rangle_{KM}$ vanishes in our examples.

\section{The MTW tensor: the vector space case}We will assume $n\geq 2$ is an integer. As before, consider a fixed symmetric, nonsingular matrix $A_{\ft}$ and define $x^{\ft} = x^{\sfT}A_{\ft}$, which is a row vector for a column vector $x$. Define the $\ft$-dot product $x^{\ft} \barx = x^{\sfT}A_{\ft} \barx$ for $x, \barx\in \R^n$. Let $\bu$ be a $C^4$-scalar function on an interval $I_{\bu}$. Consider the cost function
\begin{equation}  \bcc(x, \barx) = \bu(x^{\ft}\barx)\label{eq:bcc}\end{equation}
defined on the subset of $\R^{2n}=\R^n\times\R^n$ of pairs $(x, \barx)$ such that $x^{\ft}\barx\in I_{\bu}$. To keep the formulas compact, we will write $\bu_i$ for the $i$-th derivative $\bu^{(i)}$. If $\bu$ has an inverse function $\bs$ of class $C^4$ (which will be the case below), then we write $\bs_i$ for the $i$-th derivative $\bs^{(i)}$, for $i\in \N$, up to $i=4$. We write $s_i$ for the value $\bs_i(u)$ at a given $u$, with $s = s_0$, and $u_i$ for the value of $\bu_i$ at a point $s = x^{\ft}\barx$.

For a transportation problem on $\R^n$, we look at an open subset $\cN\subset \R^n\times \R^n$. A point in $\cN$ is often denoted $\tq = (x, \barx)$. For the cost function \cref{eq:bcc} and $\bdomg_1=(\omega_1, \bomg_1)\in \R^n\times\R^n, \bdomg_2=(\omega_2, \bomg_2)\in\R^n\times\R^n$, the pairing in \cref{eq:KimMcCann} becomes
  \begin{equation}
   \langle\bdomg_1, \bdomg_2 \rangle_{KM} = -\frac{1}{2}\left((\omega_1^{\ft} \barx\bomg_2^{\ft} x + \bomg_1^{\ft} x\omega_2^{\ft} \barx)\bu_2(x^{\ft}\barx)
                     + (\omega_1^{\ft}\bomg_2
                     + \bomg_1^{\ft}\omega_2)\bu_1(x^{\ft}\barx)\right).\label{eq:KMC}
  \end{equation}

Going forward, besides assuming $\bu$ is of class $C^4$, we will assume $\bu_1$ and $\bu_1+\bu_2s$ do not vanish as in the following proposition, so that condition A1 is satisfied in the paper. In particular, this implies the inverse $\bs$ of $\bu$ exists and is of class $C^4$.
\begin{proposition}\label{prop:metricKM} Let $s = s_0 = x^{\ft}\barx$. The pairing in \cref{eq:KMC} is nondegenerate if and only if $u_1\neq 0$, $u_1+u_2 s\neq 0$, where $u_i = \bu^{(i)}(x^{\ft}\barx)=\bu_i(x^{\ft}\barx)$. If $\bu_1$ is nonzero on its domain $I_{\bu}$, $\bu$ is monotonic and has an inverse function $\bs$ also of class $C^4$. In terms of $\bs$, the Kim-McCann pairing is given by
\begin{equation}
      \begin{gathered}
      \langle\bdomg_1, \bdomg_2 \rangle_{KM} = -\frac{1}{2}\left(-\frac{s_2}{s_1^3}(\omega_1^{\ft}\barx\bomg_2^{\ft}x + \bomg_1^{\ft} x\omega_2^{\ft}\barx)
                     + \frac{1}{s_1}(\omega_1^{\ft}\bomg_2
                     + \bomg_1^{\ft}\omega_2)\right)                   \label{eq:KMCs}
    \end{gathered}\end{equation}
where $s_i = \bs^{(i)}(u)$, $u = u_0=\bu(x^{\ft}\barx)$. It is nondegenerate if and only if $s_1\neq 0$ and $s_1^2 - s_0s_2\neq 0$. This condition also guarantees condition (A1).

Assuming these nondegenerate conditions, then the Christoffel function of the Kim-McCann metric is given by $\Gamma(\bdomg_1, \bdomg_2) = (\Gamma_x, \Gamma_{\barx})$ where
    \begin{align}
        \Gamma_x & = \frac{u_1u_3-2u_2^2}{(u_1+u_2 s)u_1}(\barx^{\ft}\omega_1\omega_2^{\ft}\barx)x \
            + \frac{u_2}{u_1}(\omega_2^{\ft}\barx)\omega_1 
            + \frac{u_2}{u_1}(\omega_1^{\ft}\barx)\omega_2\\
            & =\frac{s_1s_3- s_2^2}{s_1^2(s_0s_2-s_1^2)}(\barx^{\ft}\omega_1\omega_2^{\ft}\barx)x \
            - \frac{s_2}{s_1^2}(\omega_2^{\ft}\barx)\omega_1 
            - \frac{s_2}{s_1^2}(\omega_1^{\ft}\barx)\omega_2,\label{eq:sGammax}\\
       \Gamma_{\barx} & = \frac{u_1u_3-2u_2^2}{(u_1+u_2 s)u_1} (x^{\ft}\bomg_1\bomg_2^{\ft}x)\barx \
            + \frac{u_2}{u_1} (\bomg_2^{\ft}x)\bomg_1 
            + \frac{u_2}{u_1} (\bomg_1^{\ft}x)\bomg_2\\
            & = \frac{s_1s_3-s_2^2}{s_1^2(s_0s_2-s_1^2)} (x^{\ft}\bomg_1\bomg_2^{\ft}x)\barx \
            - \frac{s_2}{s_1^2} (\bomg_2^{\ft}x)\bomg_1 
            - \frac{s_2}{s_1^2} (\bomg_1^{\ft}x)\bomg_2.
    \end{align}
\end{proposition}
\begin{proof}Equation (\ref{eq:KMC}) for $\langle,\rangle_{KM}$ follows from the definition $\langle \bdomg, \bdomg\rangle_{KM} = - \brD_{\bomg}\rD_{\omega}\bcc$, the chain rule, and the identity $\langle \bdomg_1, \bdomg_2\rangle_{KM} = \frac{1}{2}(Q(\bdomg_1+\bdomg_2) - Q(\bdomg_1) - Q(\bdomg_2))$ where we defined $Q:\bdomg\mapsto \langle \bdomg, \bdomg\rangle_{KM}$ for $\bdomg\in\R^n\times\R^n$. Equation (\ref{eq:KMCs}) follows from the following relations between $u_1, u_2$ and $s_1, s_2, s_3$, using the inverse function theorem
    \begin{equation}\begin{aligned}
        u_1 & = \frac{1}{s_1},& u_2 & = -\frac{s_2}{s_1^3}.
      \end{aligned}\label{eq:invfunc}
      \end{equation}      
Written in block matrix form, (recall we define $x^{\ft} = x^{\sfT}A_{\ft}$), the pairing has the form
 \begin{equation} \sfg(x, \barx) =  \frac{-1}{2}\begin{bmatrix} 0 & \brD\cD \bcc\\
   \rD\bcD\bcc & 0
      \end{bmatrix}
=    \frac{-1}{2}\begin{bmatrix} 0 & u_1I_n + u_2\barx x^{\ft}\\
   u_1I_n + u_2x\barx^{\ft} & 0
      \end{bmatrix}
 \end{equation}
 which is nondegenerate if and only if $B = u_1I_n + u_2\barx x^{\ft}$ is invertible. Since $n \geq 2$, it is not invertible if $u_1=0$ since the kernel of $x^{\ft}$ is non-trivial. Thus, $B$ is invertible implies $u_1\neq 0$. Then, the equation $u_1\bomg + u_2\barx x^{\ft}\bomg = 0$ has a nonzero solution only if $x^{\ft}\bomg\neq 0$. Assuming this, $B\bomg=0$ implies
 $$0 = x^{\ft} (u_1\bomg + u_2\barx x^{\ft}\bomg)= u_1x^{\ft}\bomg + u_2x^{\ft}\barx x^{\ft}\bomg = x^{\ft}\bomg(u_1 + u_2x^{\ft}\barx)
 $$
 thus, $(u_1 + u_2x^{\ft}\barx)=0$. Hence, $B$ is invertible implies $u_1\neq 0$ and $(u_1 + u_2x^{\ft}\barx)\neq 0$. The converse follows from the same relation, and we can solve explicitly
\begin{equation} \sfg(x, \barx)^{-1} = -2\begin{bmatrix} 0 & \frac{1}{u_1}I_n - \frac{u_2}{(u_1 + u_2 x^{\ft} \barx)u_1}x\barx^{\ft}\\
   \frac{1}{u_1}I_n -\frac{u_2}{(u_1 + u_2 x^{\ft} \barx)u_1} \barx x^{\ft} & 0
      \end{bmatrix}.
\end{equation}
From $u_1 + u_2s = \frac{s_1^2-s_2s}{s_1^3}$, we get the nondegenerate condition in terms of $s$. We have the gradient $\cD \bcc(x, \barx)$ of $\bcc$ in $x$ is $\bu_1(x^{\ft}\barx)\barx$. To check condition (A1), if $\bu_1(x^{\ft}\barx_1)\barx_1 = \bu_1(x^{\ft}\barx_2)\barx_2$, then $\barx_1 = t\barx_2$ for a scalar $t$, thus, $\bu_1(ts)t = \bu_1(s)$ for $s = x^{\ft}\barx_2$. Deriving  the function $t\mapsto \bu_1(ts)t$ in $t$, we get $\bu_2(ts)st + \bu_1(ts)$. This quantity is not zero by assumption, so that function in $t$ is monotonic, thus $t=1$, giving us (A1). 

 In \cref{eq:GammaKM},  $\bcD \bcc(x, \barx) = \bu_1(x^{\ft}\barx)x$ and
$$\begin{aligned}\rD_{\omega_2}\bcD \bcc(x, \barx) & = \bu_2(x^{\ft}\barx)(\omega_2^{\ft}\barx) x + \bu_1(x^{\ft}\barx)\omega_2,
  \\
  \rD_{\omega_1}\rD_{\omega_2}\bcD\bcc(x, \barx)& = (u_3\omega_1^{\ft}\barx\omega_2^{\ft}\barx) x + (u_2\omega_2^{\ft}\barx) \omega_1
  + (u_2\omega_1^{\ft}\barx)\omega_2.
\end{aligned}  $$
From here, we compute $\Gamma_x$ below
$$\begin{gathered} 
  (\frac{1}{u_1}I_n -\frac{u_2}{(u_1 + u_2 x^{\ft} \barx)u_1} x\barx^{\ft})\rD_{\omega_1}\rD_{\omega_2}\bcD\bcc(x, \barx)
  = \frac{1}{u_1}((u_3\omega_1^{\ft}\barx\omega_2^{\ft}\barx) x + (u_2\omega_2^{\ft}\barx)\omega_1
  + (u_2\omega_1^{\ft}\barx)\omega_2)\\
  -\frac{u_2}{(u_1 + u_2 s_0)u_1}\left( (u_3\omega_1^{\ft}\barx\omega_2^{\ft}\barx)x (\barx^{\ft}x) 
  + (u_2\omega_2^{\ft}\barx) x(\barx^{\ft}\omega_1)
  +  (u_2\omega_1^{\ft}\barx)x(\barx^{\ft}\omega_2)\right)\\
  = \frac{1}{u_1}(u_2(\omega_2^{\ft}\barx)\omega_1   + u_2(\omega_1^{\ft}\barx)\omega_2) +
\frac{(\omega_1^{\ft}\barx)(\barx^{\ft}\omega_2)}{u_1(u_1 + u_2 s_0)}\left((u_1 + u_2 s_0)u_3 -u_2( u_3s_0  + u_2   +  u_2)\right) x,  
\end{gathered}  $$
and the coefficient of $x$ is reduced to $\frac{u_1u_3 -2u_2^2}{u_1(u_1 + u_2 s_0)} (\omega_1^{\ft}\barx)(\barx^{\ft}\omega_2).$

We express $\Gamma_x$ in terms of $s_i$'s using $u_1, u_2$ in \cref{eq:invfunc} and $u_3$ below. We also compute $u_4$, which will be used in the cross curvature.
\begin{equation}\begin{aligned}
    u_3 & = - \frac{s_1s_3-3s_2^2}{s_1^5}, &\quad  u_4 & = - \frac{s_1^2s_4 - 10s_1s_2s_3 + 15s_2^3}{s_1^7}.
\end{aligned}    \label{eq:invfunc2}
\end{equation}

We derive $\Gamma_{\barx}$ similarly.
\end{proof}
The cross curvature formula is given in \cref{theo:crossMain}. We can now prove the theorem.
\begin{proof}(of \cref{theo:crossMain}.) We compute
$$\begin{aligned}  \brD_{\bomg}\brD_{\bomg}\cD\bcc(x, \barx)& = u_3(\bomg^{\ft}x)^2 \barx + 2(u_2\bomg^{\ft}x) \bomg,\\
  \rD_{\omega}\brD_{\bomg}\brD_{\bomg}\cD\bcc(x, \barx)& = (u_4\omega^{\ft}\barx(\bomg^{\ft}x)^2) \barx
  + (2u_3\bomg^{\ft}\omega\bomg^{\ft}x) \barx 
  + (2u_3\omega^{\ft}\barx\bomg^{\ft}x) \bomg + (2u_2\bomg^{\ft}\omega) \bomg,\\
  \rD_{\omega}\brD_{\bomg}\brD_{\bomg}\rD_{\omega}\bcc(x, \barx)& = u_4(\omega^{\ft}\barx)^2(\bomg^{\ft}x)^2
  + 4u_3(\bomg^{\ft}\omega)(\bomg^{\ft}x) (\omega^{\ft}\barx)  + 2u_2(\bomg^{\ft}\omega)^2.\\  
  \end{aligned}$$
  From here, $  \bomg_{\diamond}  := (\brD\cD \bcc)^{-1}\brD_{\bomg}\brD_{\bomg}\cD \bcc=
\frac{u_1u_3-2u_2^2}{(u_1+u_2 s)u_1} (x^{\ft}\bomg)^2\barx \
+ 2\frac{u_2}{u_1} (\bomg^{\ft}x)\bomg$ and
$$\begin{aligned}  
            \rD_{\omega}\rD_{\omega}\bcc(x, \barx) & = u_2(\omega^{\ft}\barx)^2,\\
            \brD_{\bomg_{\diamond}}\rD_{\omega}\rD_{\omega}\bcc(x, \barx) & = u_3 (x^{\ft}\bomg_{\diamond})(\omega^{\ft}\barx)^2 + 2u_2(\omega^{\ft}\barx)(\omega^{\ft}\bomg_{\diamond}).
\end{aligned}  $$
Setting $S_{\omega} = \barx^{\ft}\omega, S_{\bomg} =x^{\ft}\bomg, S_{\omega\bomg}=\omega^{\ft}\bomg$, we simplify $2\cross(\bdomg)$ from \cref{eq:crosssec}
$$\begin{gathered}u_3 (x^{\ft}\bomg_{\diamond})(\omega^{\ft}\barx)^2 + 2u_2(\omega^{\ft}\barx)(\omega^{\ft}\bomg_{\diamond})
  - u_4(\omega^{\ft}\barx)^2(\bomg^{\ft}x)^2 - 4u_3(\bomg^{\ft}\omega)(\bomg^{\ft}x)(\omega^{\ft}\barx) - 2u_2(\bomg^{\ft}\omega)^2\\
= u_3(\frac{u_1u_3-2u_2^2}{(u_1+u_2s)u_1}sS_{\bomg}^2 + 2\frac{u_2}{u_1}S_{\bomg}^2)S_{\omega}^2 + 2u_2S_{\omega}(\frac{u_1u_3-2u_2^2}{(u_1+u_2s)u_1}S_{\bomg}^2S_{\omega} + 2\frac{u_2}{u_1}S_{\omega\bomg}S_{\bomg}) \\
- u_4S_{\bomg}^2S_{\omega}^2 - 4u_3S_{\omega}S_{\bomg}S_{\omega\bomg} - 2u_2S_{\omega\bomg}^2  \\
=(\frac{u_1u_3^2s+4u_2u_1u_3-4u_2^3}{(u_1+u_2s)u_1}-u_4)S_{\bomg}^2S_{\omega}^2 + (\frac{4u^2_2}{u_1} - 4u_3)S_{\omega}S_{\bomg}S_{\omega\bomg} - 2u_2S_{\omega\bomg}^2.
\end{gathered}$$
Set $J = \langle\bdomg, \bdomg \rangle_{KM} = - u_2S_{\omega}S_{\bomg} - u_1S_{\omega\bomg}$, then replacing $S_{\omega\bomg}$ by $- \frac{1}{u_1}(u_2S_{\omega}S_{\bomg}+ J)$ in the above, $2\cross(\bdomg)$ reduces to
$$\begin{gathered}(\frac{u_1u_3^2s+4u_2u_1u_3-4u_2^3}{(u_1+u_2s)u_1}-u_4)S_{\bomg}^2S_{\omega}^2 - 4\frac{u_2^2-u_1u_3}{u_1^2}S_{\omega}S_{\bomg}(u_2S_{\omega}S_{\bomg}+J)  - 2\frac{u_2}{u_1^2}(u_2S_{\omega}S_{\bomg}+J)^2.
\end{gathered}$$
Using $u_4$ in \cref{eq:invfunc2}, the coefficient of $J^2$ is $-2\frac{u_2}{u_1^2} = 2\frac{s_2}{s_1}$, the coefficient of $S_{\omega}S_{\bomg}J$ is
$$\frac{- 8u_2^2+ 4u_1u_3}{u_1^2} = s_1^2(\frac{-8s_2^2}{s_1^6}  - 4\frac{s_1s_3-3s_2^2}{s_1^6})=\frac{4s_2^2-4s_1s_3}{s_1^4}.$$
The coefficient of $S_{\omega}^2S_{\bomg}^2$ factored out $((u_1+u_2s)u_1^2)^{-1}=(\frac{s_1^2-ss_2}{s_1^5})^{-1}$ is
$$\begin{gathered}u_1(u_1u_3^2s+4u_2u_1u_3-4u_2^3)-u_4(u_1+u_2s)u_1^2 - 4u_2(u_1+u_2s)(u_2^2-u_1u_3) - 2u_2^3(u_1+u_2s)\\
   = u_1(u_1u_3^2s+4u_2u_1u_3-4u_2^3) - (u_1+u_2s)(u_4u_1^2 + 6u_2^3-4u_1u_2u_3)\\
    = \frac{1}{s_1^{12}}(s(s_3s_1-3s_2^2)^2
    + 4s_1^2s_2(s_3s_1-3s_2^2)
    + 4s_1^2s_2^3) \\
    - \frac{1}{s_1^{12}}(s_1^2-s_2s)(- (s_4s_1^2 - 10s_3s_2s_1 + 15s_2^3)
    - 6s_2^3-4s_2(s_3s_1-3s_2^2))\\
 = \frac{1}{s_1^{12}}\big(ss_1^2s_3^2- 6ss_1s_2^2s_3 + 9ss_2^4 + 4s_1^3s_2s_3-12s_1^2s_2^3 \
            + 4s_2^3s_1^2 \\
            + (s_1^2-s_2s)(s_4s_1^2 - 6s_1s_2s_3 + 9s_2^3 )         \big)\\
= \frac{1}{s_1^{12}}\big(ss_1^2s_3^2- 6ss_1s_2^2s_3 + 9ss_2^4 + 4s_1^3s_2s_3-12s_1^2s_2^3 \
            + 4s_2^3s_1^2 \\
            + (s_1^2-s_2s)s_4s_1^2 -6s_1^3s_2s_3+6ss_1s_2^2s_3 + 9s_1^2s_2^3-9ss_2^4 \big)
\end{gathered}
$$
which is $\frac{s_1^2(s_1^2 - s_0s_2)s_4 + ss_1^2s_3^2 - 2s_1^3s_2s_3 + s_1^2s_2^3}{s_1^{12}} $, the coefficient for $S_{\omega}^2S_{\bomg}^2$ in \cref{eq:crossRn}.
\end{proof}
The solution to the ODE (\ref{eq:zeroODE}) is in \cref{theo:sol}. We will now prove it.
  \begin{proof}(of \cref{theo:sol}) Direct calculation shows 
    $$\begin{gathered}
      \frac{dS}{du} = -\frac{((\bs_1^2-\bs \bs_2)\bs_4 + \bs\bs_3^2  - 2\bs_1\bs_2\bs_3 +\bs_2^3)\bs}{(\bs\bs_2 - \bs_1^2)^2},\\    
      \frac{dP}{du} = -\frac{((\bs_1^2-\bs \bs_2)\bs_4 + \bs\bs_3^2  - 2\bs_1\bs_2\bs_3 +\bs_2^3)\bs_1}{(\bs\bs_2 - \bs_1^2)^2}.
    \end{gathered}$$
    Thus, $S$ and $P$ are constants for a solution of \cref{eq:zeroODE}. We have
    $\bs_3= \frac{\bs_2^2}{\bs_1}-(\frac{\bs_1^2-\bs\bs_2}{\bs_1})P$, while the equation for $S$ implies
    $$\begin{gathered}
     (\bs_1^2-\bs\bs_2)S - \bs_1\bs_2 + \bs(\frac{\bs_2^2}{\bs_1}-(\frac{\bs_1^2-\bs\bs_2}{\bs_1})P)=0\\
\Rightarrow      0 = \bs\bs_2^2 + (\bs^2P-\bs_1\bs S-\bs_1^2)\bs_2  +  \bs_1^3S- \bs_1^2\bs P = (\bs\bs_2 - \bs_1^2)(\bs_2 - S\bs_1 + P\bs).
\end{gathered}$$
Since $\bs\bs_2 -\bs_1^2 \neq 0$, this implies \cref{eq:secondord}. Conversely, if this holds then $\bs_2 - S\bs_1 + P\bs = 0$, and with $S =\frac{\bs_2 +P\bs}{\bs_1}$, we extract the expression for $P$, and $P'=0$ implies \cref{eq:zeroODE}. This gives us the solutions in \cref{eq:Sol1,eq:Sol3,eq:Sol2}. The below is for later convenience.\hfill\break
    For \cref{eq:Sol1}, $s_1^2-ss_2 = -p_0p_2(p_1-p_3)^2e^{(p_1+p_3)u}, S=p_1+p_3, P=p_1p_3$, $\Delta > 0$.\hfill\break
    For \cref{eq:Sol3}, $s_1^2-ss_2 = a_1^2e^{2a_2u}, S=2a_2, P=a_2^2$, $\Delta=0$.\hfill\break
    For \cref{eq:Sol2}, $s_1^2-ss_2 = b_0^2b_2^2e^{2b_1u}, S=2b_1, P=b_1^2+b_2^2$, $\Delta < 0$.
  \end{proof}
\section{The sphere and the hyperboloid model}
We derive a number of common formulas for both manifolds by treating the general case of the semi-Riemannian sphere. However, for a tangent vector $\xi$, since $\xi^{\ft}\xi$ could be negative when the induced model is only semi-Riemannian, the analysis leads to A3w only for the sphere and the hyperboloid model.

For $\epsilon\in \{\pm 1\}$, assume $\epsilon A_{\ft}$ is nondegenerate and diagonal. Let $P_+=P_{+,\epsilon}$ and $P_-=P_{-,\epsilon}$ be the sets of indices corresponding to positive and negative eigenvalues of $\epsilon A_{\ft}$, respectively, and assume the cardinality $|P_+|$ of $P_+$ is not less than $1$. Set
\begin{equation}\rS^n_{\ft, \epsilon} = \{x^{\ft}x = \epsilon\; |\;x\in \R^{n+1}\}.
\end{equation}
The equation for $\rS^n_{\ft, \epsilon}$ could be normalized to $\sum_{i\in P_+} \lambda_i x_i^2 = 1+\sum_{j\in P_-} \lambda_j x_j^2$, with $\lambda_l > 0, l\in P_+\cup P_-$ are absolute values of eigenvalues of $\epsilon A_{\ft}$. This manifold is considered in \cite[Sections 4.22-4.30]{ONeil1983}. The metric on $\rS^n_{\ft, \epsilon}$ is semi-Riemannian. The (induced) Riemannian cases correspond to $|P_-|=0$ and $|P_+|=1$, normalized to

\begin{enumerate}\item The round sphere $\rS^n$, $A_{\ft} = Id, \epsilon = 1$. Here,
  $|x^{\ft}\barx|\leq 1$ and $|x^{\ft}\barx|=1$ implies $\barx= \pm x$.
\item The hyperboloid model of the hyperbolic space $\rH^n$, $A_{\ft} =\diag(-1, 1,\cdots,1)$, and $\epsilon=-1$, on the component $x_0 >0$ for $x = (x_0,x_1,\cdots,x_n)\in \R^{n+1}$. The last restriction is because we have two branches and we focus on one connected branch. The manifold $\rS^n_{\ft, \epsilon}$ is not connected in this case, but we abuse the notation to refer to $\rS^n_{\ft, \epsilon}$ as one connected component, the \emph{right half} $\rH^n$ (the alternative choice $A_{\ft} =\diag(1,\cdots,1, -1)$ with $x_n > 0$ corresponds to the {upper half}). Here, $x^{\ft}\barx\leq -1$ for $x, \barx\in\rH^n$ and $|x^{\ft}\barx|=1$ implies $\barx= x$ (by the additional the condition $x_0>0$).
\end{enumerate}
With this normalization, we note $\epsilon(1-(x^{\ft}\barx)^2) \geq 0$ for both $\rS^n$ and $\rH^n$. We will consider transport problems on $\rS_{\ft, \epsilon}$, with cost $\bu(x^{\ft}\barx)$, and will focus on the two Riemannian cases above.
\begin{proposition}\label{prop:semisphere} Assume the function $\bu$ and its inverse $\bs$ are such that
\begin{equation} (-\bs_0\bs_1^2+\bs_2(\bs_0^2-1))(\bs_1^2-\bs_2\bs_0)\bs_1\neq 0\end{equation}
  for $u\in I_u$. Set $s_i=\bs^{(i)}(u)$. With the cost function $\bcc(x, \barx) = \bu(x^{\ft}\barx)$, the MTW tensor for $\tq=(x, \barx)\in \cN\subset\cM\times\cM$ where $\cM:=\rS_{\ft, \epsilon}^n$ and $\cN$ is an open subset of $\cM\times\cM$ at a tangent vector $\bdxi=(\xi,\bxi)\in T_{\tq}\cN$ is given by the Gauss-Codazzi equation (\ref{eq:GaussCod0}),
  $\cross_{\rS^{n}_{\ft, \epsilon}} = \cross_{\R^{n+1}} + \mathcal{II}$, where $\cross_{\R^{n+1}}$ is given by \cref{eq:crossRn}, and the contribution of the second fundamental form $\mathcal{II}:=\langle \Two(\xi_{\ast 0}, \xi_{\ast 0}), \Two(\bxi_{0\ast }, \bxi_{0\ast })\rangle_{KM}$ is
\begin{gather}
\mathcal{II}=  \frac{(- (s_1^2-ss_2)s_1^2\xi^{\ft}\xi + \epsilon(s_2^2 - s_3s_1)(\barx^{\ft}\xi)^2
  )(- (s_1^2-ss_2)s_1^2\bxi^{\ft}\bxi + \epsilon(s_2^2 - s_3s_1)(x^{\ft}\bxi)^2)             
}{2s_1^5(s_1^2-ss_2)(-ss_1^2+(s^2-1)s_2)}.\label{eq:TwoII}
\end{gather}
There is no contribution from $\langle \Two(\bxi_{0\ast }, \bxi_{\ast 0}),\Two(\bxi_{0\ast }, \bxi_{\ast 0})\rangle_{KM}$ since  $\Two(\xi_{\ast 0}, \xi_{0 \ast})=0$.
\end{proposition}
\begin{proof}From calculus, the tangent space $T_x\rS^n_{\ft, \epsilon}$ is the null space of the Jacobian of the constraint $x^{\ft}x=\epsilon$. Thus, a vector $\xi$ is tangent to $\rS^n_{\ft, \epsilon}$ if and only if $x^{\ft}\xi=0$.

  For an element $K\in \R^{n+1}\times\R^{n+1}$, we write $K=(K_x, K_{\barx})$. First, we verify for $\bdomg = (\omega, \bomg)\in \R^{n+1}\times\R^{n+1}$, the metric-compatible projection from $\R^{n+1}\times\R^{n+1}$ to the product $T\rS^n_{\ft, \epsilon}\times T\rS^n_{\ft, \epsilon}$ is given by $\Pi(\tq)\bdomg$ with components
  \begin{equation}
  \begin{split}
    (\Pi(\tq)\bdomg)_x &= \omega-  \frac{x^{\ft}\omega}{ss_1^2 - (s^2-1)s_2}((s_1^2-ss_2)\barx + \epsilon s_2x), \\
    (\Pi(\tq)\bdomg)_{\barx} &= \bomg-\frac{\barx^{\ft}\bomg}{ss_1^2- (s^2-1)s_2}((s_1^2-ss_2)x + \epsilon s_2\barx).
  \end{split}
  \end{equation}
  By assumption, $ss_1^2 -s_2(s^2-1)\neq 0$, and the vector $\Pi(\tq)\bdomg)$ given above is tangent to $\cN$, as, for example $x^{\ft}(\Pi(\tq)\bdomg)_x=0$ from
  $$x^{\ft}((s_1^2-ss_2)\barx + \epsilon s_2x)=s(s_1^2-ss_2) + \epsilon^2 s_2=ss_1^2 - (s^2-1)s_2.
  $$
  Assume $\bdxi=(\xi, \bxi)$ is tangent to $\cN$ then
  $$\begin{gathered}\langle \bdxi, \bdomg - \Pi(\tq)\bdomg\rangle_{KM}=-\frac{1}{2s_1^3(ss_1^2 + (s^2-1)s_2)}\big(
    -s_2\barx^{\ft}\xi \barx^{\ft}\bomg x^{\ft}((s_1^2-ss_2)x + \epsilon s_2\barx) \\
    - s_2 x^{\ft}\bxi x^{\ft}\omega \barx^{\ft}((s_1^2-ss_2)\barx +\epsilon s_2x) 
    + s_1^2 \barx^{\ft}\bomg\xi^{\ft}((s_1^2-ss_2)x + \epsilon s_2\barx) \\
    + s_1^2 x^{\ft}\omega\bxi^{\ft}((s_1^2-ss_2)\barx + \epsilon s_2x)\big).
  \end{gathered}$$
Since $x^{\ft}((s_1^2-ss_2)x + \epsilon s_2\barx) = \epsilon s_1^2= \barx^{\ft}((s_1^2-ss_2)\barx +\epsilon s_2x)$ and $\xi^{\ft}x = \bxi^{\ft}\barx=0$, the items in the bracket reduces to
$$\begin{gathered}
  - s_2\barx^{\ft}\xi \barx^{\ft}\bomg \epsilon s_1^2 
    - s_2 x^{\ft}\bxi x^{\ft}\omega \epsilon s_1^2 
    + s_1^2 \barx^{\ft}\bomg\xi^{\ft}(\epsilon s_2\barx) 
    + s_1^2 x^{\ft}\omega\bxi^{\ft}(\epsilon s_2x)=0.  
\end{gathered}$$
confirming $\Pi$ is metric compatible. Differentiating $\Pi(\tq)\bdeta$ in the variable $\tq$ in direction $\bdxi$ for $\bdxi, \bdeta\in T\cN$ and noting $x^{\ft}\eta = 0 = \barx^{\ft}\baeta$ to remove these terms after taking derivative
$$\Pi'(\tq;\bdxi)\bdeta = -(
    \frac{\xi^{\ft}\eta}{ss_1^2 - (s^2-1)s_2}((s_1^2-ss_2)\barx + \epsilon s_2x),
    \frac{\bxi^{\ft}\baeta}{ss_1^2- (s^2-1)s_2}((s_1^2-ss_2)x + \epsilon s_2\barx)
    ).
$$
Recall $\xi_{\ast 0} = (\xi, 0)$ and $\bxi_{0\ast} = (0, \bxi)$ for $\bdxi=(\xi, \bxi)\in T\cN$. From \cref{eq:sGammax}
$$\begin{gathered}
  (I - \Pi(\tq))\Gamma(\xi_{\ast 0}, \xi_{\ast 0}) = (
\frac{s_2^2 - s_3s_1}{(s_1^2-ss_2)s_1^2}\frac{(\barx^{\ft}\xi)^2\epsilon}{
            ss_1^2-(s^2-1)s_2}((s_1^2-ss_2)\barx+\epsilon s_2x)
  , 0)
\end{gathered}
$$
since in  $\Gamma_x$, we only need to project the term with $x$, the tangent $\xi$ has zero vertical component. From \cref{eq:secform}, the second fundamental form $\Two(\xi_{\ast 0}, \xi_{\ast 0})$ is $$\begin{gathered}\Two(\xi_{\ast 0}, \xi_{\ast 0}) = 
  (\frac{\epsilon(s_2^2 - s_3s_1)(\barx^{\ft}\xi)^2 - (s_1^2-ss_2)s_1^2\xi^{\ft}\xi}{
(s_1^2-ss_2)s_1^2(ss_1^2-(s^2-1)s_2)}((s_1^2-ss_2)\barx+\epsilon s_2x)
, 0)
\end{gathered}
$$
and $\Two(\bxi_{0\ast }, \bxi_{0\ast })$ is evaluated similarly by permuting the components. Note,
$$\begin{gathered}
\langle ((s_1^2-ss_2)\barx+\epsilon s_2x, 0), (0, (s_1^2-ss_2)x+\epsilon s_2\barx))\rangle_{KM}\\
= \frac{1}{2}\frac{s_2}{s_1^3}\epsilon^2 s_1^4 - \frac{1}{2s_1} ((s_1^2-ss_2)^2s + 2s_1^2s_2 - s_2^2s)\\
=\frac{1}{2s_1}(-s_2(s_1^2-ss_2) - s(s_1^2-ss_2)^2)=\frac{(s_1^2-ss_2)(- ss_1^2 + (s^2-1)s_2)}{2s_1}.
\end{gathered}$$
Canceling similar terms, $\langle \Two(\xi_{\ast 0}, \xi_{\ast 0}), \Two(\bxi_{0\ast }, \bxi_{0\ast })\rangle_{KM}$ is given by \cref{eq:TwoII}.

Finally, both $\Gamma(\xi_{\ast 0}, \xi_{0 \ast})$ and $\Pi'(\tq, \xi_{\ast 0}) \xi_{0 \ast}$ vanish, thus $\Two(\xi_{\ast 0}, \xi_{0 \ast})=0$.
\end{proof}
For the cost with inverse $\bs(u) = p_0e^{p_1u} + p_2$, we have the next corollary, observe
$$\begin{aligned}s_2^2 - s_3s_1 & = 0,\\
s_1^2 - ss_2 & =-p_0p_1^2p_2e^{p_1u}  = -p_1^2p_2(s-p_2),\\
-ss_1^2 +(s^2-1)s_2 & = p_0p_1^2(p_0p_2e^{p_1u} + p_2^2 - 1)e^{p_1u}
=p_1^2(p_2s - 1)(s-p_2).
\end{aligned}$$ 
\begin{corollary}If $\bs(u) = p_0e^{p_1u} + p_2, \bu(s) = \frac{1}{p_1}\log(\frac{s-p_2}{p_0})$, the MTW tensor is
  \begin{equation}\cross_{\rS^{n}_{\ft,\epsilon}}(\bdxi)=-\frac{p_2}{2p_1(x^{\ft}\barx - p_2)(p_2x^{\ft}\barx - 1)}(\xi^{\ft}\xi)(\bxi^{\ft}\bxi) + N_J\label{eq:antennaRS}\end{equation}
at $x, \barx\in\cM\times\cM$, where $N_J$ vanishes for null vectors.
  
For the hyperboloid model $\rH^n$, $x^{\ft}\barx \leq -1$ for $x, \barx\in \rH^n$, thus, if $p_1 < 0, p_0 < 0, p_2 =1$, the cross-curvature satifies A3s for all $x, \barx\in \rH^n$. When $p_2=0, p_0 < 0$, the Kim-McCann pairing of $\frac{1}{p_1}\log(\frac{1}{p_0}x^{\ft}\barx)$ on $\rH^n$ is nondegenerate and has \underline{zero cross curvature} on null vectors, thus, satisfying A3w.

For the sphere $\rS^n$, the model with $p_1 <0, p_2^2 =1, p_0p_2 < 0$, or $\bcc(x, \barx) = -\frac{1}{|p_1|}\log\frac{1-p_2 x^{\sfT}\barx}{|p_0|}$, a scaled reflector antenna, satisfies A3s for except for when $\barx=p_2 x$.\label{cor:antena}
\end{corollary}
Most of the above are clear, note $s_1=p_1(s-p_2)$. The case $p_2=0$ for $\rH^n$ is interesting. The fact that this cost satisfying A3w is already known from \cite{LeeLi}, we emphasize that the cross curvature is actually zero on null vectors. The Kim-McCann pairing with $\bs(u) = p_0e^{p_1u}$ is degenerate on $\R^{n+1}$. It is nondegenerate on $\rH^n$, and to compute the cross curvature in this case, as mentioned in \cref{sec:GaussCodazzi}, inspecting the proof of \cite[Theorem 4.5]{ONeil1983}, the Gauss-Codazzi equation is still applicable if we have a pairing-compatible projection and connection. For $\bs(u) = p_0e^{p_1u}+p_2, p_2\neq0$, $\Gamma^{\cE}(\bdxi, \bdeta)$ in \cref{theo:crossMain} has $s_2^2 - s_3s_1=0$, thus, it has no terms proportional to $x$ or $\barx$, the only terms that could cause singularity when $p_2=0$. Thus, the compatibility equation for the Levi-Civita connection can be taken through the limit, the limiting connection as $p_2\to 0$ below is still pairing-compatible
$$
\Gamma^{\cE}(\bdxi, \bdeta) = ( - \frac{1}{s}(\barx^{\ft}\eta)\xi - \frac{1}{s}(\barx^{\ft}\xi)\eta, - \frac{1}{s}(x^{\ft}\baeta)\bxi - \frac{1}{s}(x^{\ft}\bxi)\baeta).
$$
To apply the Gauss Codazzi equation for this connection, we can compute its cross curvature by taking the limit. In \cref{eq:crossRn}, the coefficient of $F^2$ is zero when $p_2\neq0$, so the cross curvature of the limiting connection on $\R^n$ is zero. The projection could also be taken through limit as $s_1^2-ss_2$ is not in the denominator, while $(I - \Pi(\tq))\Gamma(\xi_{\ast 0}, \xi_{\ast 0})=(0, 0)$ before and after the limit. Thus, the cross-curvature on $\rH^n$ for $p_2=0$ is consistent with \cref{eq:antennaRS}, and vanishes on null vectors. 

\begin{theorem}\label{theo:R1234} For $\rS^n$ and $\rH^n$, consider the cost $\bu(x^{\ft}\barx)$ on the manifold $\cM=\rS^n_{\ft, \epsilon}$ at $(x, \barx)$. If $\lvert x^{\ft}\barx\rvert \neq 1$, set $S_{\xi} = \barx^{\ft}\xi, S_{\bxi} = x^{\ft}\bxi$, then $S_{\xi\perp }^{2} := \xi^{\ft}\xi-\frac{\epsilon (\barx^{\ft}\xi)^2}{1-s^2} \geq 0$, $S_{\bxi\perp}^2:= (\bxi^{\ft}\bxi)-\frac{\epsilon(x^{\ft}\bxi)^2}{1-s^2}\geq 0$. The cross curvature at $\bdxi\in T\cM\times T\cM$ is given by
  \begin{equation}
\begin{aligned}
    \cross_{\cM}(\bdxi) & =
    \frac{1}{2}(\frac{R_1}{(1-s^2)^2} S_{\xi}^2S_{\bxi}^2
    + \frac{\epsilon R_{23}}{1-s^2}(S_{\bxi\perp}^2S_{\xi}^2 + S_{\bxi}^2S_{\xi\perp}^2)
    + R_4S_{\xi\perp }^{2}S_{\bxi\perp }^{2}
    ) + N_J \label{eq:sphereMTW},\\
 N_J & :=             - 2\frac{(s_1s_3 - s_2^2)}{s_1^4} (x^{\ft}\bxi)(\barx^{\ft}\xi) \langle\bdxi, \bdxi \rangle_{KM} \
 + \frac{s_2}{s_1}\langle\bdxi, \bdxi \rangle_{KM}^2.
\end{aligned} 
\end{equation}
The coefficients $R_1, R_{23}, R_4$ are defined below:
$$\begin{aligned}D & := (-ss_1^2-s_2(1-s^2))(s_1^2-s_2s)s_1^5,\\
R_1 & := \frac{1}{D}\big(((s_1^2 - ss_2)s_4 + ss_3^2 - 2s_1s_2s_3 + s_2^3)(- s_2-ss_1^2+s_2s^2)(1-s^2)^2 \\
 &  + ((s_3s_1-s_2^2)(1-s^2) + s_1^2(s_1^2-s_2s))^2 \
            \big),\\
R_{23} & := \frac{1}{D} ((s_3s_1-s_2^2)(1-s^2) + s_1^2(s_1^2-s_2s)) 
               s_1^2(s_1^2-s_2s),\\
               R_4 & := \frac{1}{D}s_1^4(s_1^2-s_2s)^2.
\end{aligned}                 
$$
If $n=1$ the cost function satisfies A3w if $R_1 \geq 0$ and A3s if $R_1 > 0$. If $n\geq 2$, the cost function satisfies A3w if and only if $R_1, R_{23}, R_4$ are nonnegative and A3s if they are all positive.
\end{theorem}
Using \cref{prop:semisphere}, we also see when $\lvert x^{\ft}\barx \rvert = 1$, the cross curvature is
\begin{equation}
\begin{gathered}\cross_{\cM}(\bdxi)_{\lvert x^{\ft}\barx\rvert = 1} = -\frac{1}{2}\frac{
    (s_1^2-ss_2)(\xi^{\ft}\xi) (\bxi^{\ft}\bxi)}{ss_1^3} + \frac{s_2}{s_1}\langle\bdxi, \bdxi \rangle_{KM}^2.\label{eq:crossS1}
\end{gathered}                 
\end{equation}
\begin{proof}
  When $|x^{\ft}\barx|\neq 1$, then $\barx - \epsilon sx \in T_x\rS^{n}_{\ft,\epsilon}$ and $x - \epsilon s\barx\in T_{\barx}\rS^{n}_{\ft,\epsilon}$. The projection of $\xi$ to $\barx-\epsilon sx$ is
  $\mathsf{proj}_{\barx-\epsilon sx}\xi =\frac{\xi^{\ft}(\barx-\epsilon sx)}{(\barx-\epsilon sx)^{\ft}(\barx-\epsilon sx)}(\barx-\epsilon sx)=\frac{\epsilon \barx^{\ft}\xi}{1-s^2}(\barx-\epsilon sx)$. Thus, in the Gram-Schmidt decomposition of $\xi$, the residual in the Riemannian metric is nonnegative
  $$S_{\xi\perp}^2 = |\xi|^2_{\ft} - |\mathsf{proj}_{\barx-\epsilon sx}\xi |_{\ft}^2 =
  \xi^{\ft}\xi - \frac{\epsilon S_{\xi}^2}{1-s^2}=
  |\xi -  \frac{\epsilon \barx^{\ft}\xi}{1-s^2}(\barx-\epsilon sx)|_{\ft}^2\geq 0$$
where we define $|\omega|_{\ft}^2 = \omega^{\ft}\omega, \omega\in T_x\rS_{\ft,\epsilon}$. Similarly $S_{\bxi\perp}^2\geq 0$. Substitute $\xi^{\ft}\xi = S_{\xi\perp}^2 + \frac{\epsilon S_{\xi}^2}{1-s^2}$, $\bxi^{\ft}\bxi = S_{\bxi\perp}^2 + \frac{\epsilon S_{\bxi}^2}{1-s^2}$, expand $\cross_{\cM}$ from \cref{prop:semisphere} using these, we get \cref{eq:sphereMTW}.

If $n=1$ then when $s^2 \neq 1$, we have $S_{\xi\perp}^2 = S_{\bxi\perp}^2=0$, the cross curvature has only the term with $R_1$. In this case, $(x, \barx)$ forms a basis of $\R^2$, so if $R_1 \neq 0$, the cross curvature is zero only if either $\xi$ or $\bxi$ is zero. When $s^2=1$, $R_1=\frac{\bs_1^4(\bs_1^2 - \bs\bs_2)^2}{D}\neq 0$. Note that $R_1$ and $D$ are of the same sign.

If $n\geq 2$, we can always find a vector $\xi\in T_x\cM$ orthogonal to $\barx$, and $\bxi\in T_{\barx}\cM$ orthogonal to $x$, thus, $S_{\xi}= S_{\bxi}=0$, while $S_{\xi\perp}, S_{\bxi\perp}$ are nonzero. Hence, if the cross curvature is nonnegative then $R_4 \geq 0$, and if $R_4=0$ then A3s is violated. Similarly, if $\xi$ is proportional to $\barx -\epsilon sx$ and $\bxi$ is proportional to $x-\epsilon s\barx$, we deduce $R_1 \geq 0$. The equation $R_{23}\geq0$ is deduced by using $\barx -\epsilon sx$ and $\bxi$ orthogonal to $x$, and A3s is violated if either coefficient is zero. If all three coefficients are positive, then the cross curvature is zero only if $S_{\xi}S_{\bxi}= S_{\xi}S_{\bxi\perp} = S_{\xi\perp}S_{\bxi}= S_{\xi\perp}S_{\bxi\perp}=0$, note that $S_{\xi}=0=S_{\xi\perp}$ implies $\xi=0$. But if either $S_{\xi}$ or $S_{\xi\perp}$ is not zero, these equalities force $S_{\bxi}=0=S_{\bxi\perp}$. At $s^2=1$, \cref{eq:crossS1} shows the conditions for A3w, A3s are controlled by $R_4$, which has the same sign with $D$.
\end{proof}
It is relatively easy to construct costs such that the MTW tensor satisfies A3w for some range of $x^{\ft}\barx$. Theorem \ref{theo:Hn1} in the introduction gives us costs that are regular for the full range $x^{\ft}\barx\leq -1$ of $\rH^n$. We will prove it now.

\begin{proof}(of \cref{theo:Hn1}) The terms corresponding to $\cross_{\R^{n+1}}$ vanishes since we use a cost in \cref{theo:sol}. Thus, the numerator of $R_1$ is a square of a positive number. By this observation, in the remaining cases, $R_1, R_{23}, R_4>0$ follows automatically from $D>0$ and $R_{23}>0$. In the generalized hyperbolic case
  $$\begin{gathered}
    s_1 = p_0p_1e^{p_1u} + p_2p_3e^{p_3u},\\
    s_1^2-s_2s =    -p_0p_2(p_1 - p_3)^2e^{(p_1+p_3)u},\\
    s_3s_1-s_2^2 = p_0p_1p_2p_3(p_1 - p_3)^2e^{(p_1+p_3)u}.
  \end{gathered}$$
If the coefficients are as specified in case 1 in the theorem then $s_1 > 0$, $s$ has range $(-\infty, \infty)$ covering the range $(-\infty, -1]$ of $x^{\ft}\barx$ and $s_1^2-s_2s > 0$. For $D >0$, we need to show $- s_2-ss_1^2+s_2s^2>0$. Let $Z = e^{p_3u} > 0$, then $p_0e^{p_1u} = s - p_2Z$, thus, $s_1 = p_1(s - p_2Z) + p_3p_2Z$, $s_2 = p_1^2(s - p_2Z) + p_3^2p_2Z$, and $- s_2-ss_1^2+s_2s^2$ is
    $$\begin{gathered}- p_1^2(s - p_2Z) - p_3^2p_2Z - s(p_1(s - p_2Z) + p_3p_2Z)^2 +
      ( p_1^2(s - p_2Z) + p_3^2p_2Z)s^2,\\
- s_2-ss_1^2+s_2s^2    = -sp_2^2(p_1- p_3)^2Z^2 + p_2(p_1 - p_3)(p_1s^2 + p_1 - p_3s^2 + p_3)Z - p_1^2s.
  \end{gathered}$$
 In the above, $p_1(s^2 + 1)-p_3(s^2 -1) <0$ for $-s \geq 1$ from our assumption, all three coefficients of the polynomial in $Z>0$ are positive, thus $D > 0$ and hence $R_4 >0$.
    
 For $R_{23}$, still set $Z = e^{p_3u}$, 
\begin{equation}\begin{gathered}
 \frac{DR_{23}}{s_1^2(s_1^2-ss_2)e^{(p_1+p_3)u}}=-p_0p_2(p_1 - p_3)^2
 \left(p_1p_3(s^2-1)+(p_1s +p_2(p_3-p_1)Z)^2\right)\\
=-p_0p_2(p_1 - p_3)^2\left(p_2^2(p_1 - p_3)^2Z^2 -2p_1p_2s(p_1 - p_3)Z+ p_1(p_1s^2 + p_3s^2 - p_3)\right).
  \end{gathered}\label{eq:DR23}
\end{equation}
By assumption, all coefficients of the  polynomial in $Z$ are positive. To check the last coefficient $p_1((p_1+p_3)s^2-p_3)$ is positive, note that the second factor is also negative, as $p_1+p_3\leq 0$. Thus, $R_{23}> 0$.
 
For case 2, from our assumption, $\bs$ is an increasing function for $u <u_c$, with range $(-\infty, \bs(u_c))$, for $\bs(u_c) = \frac{p_2(p_1-p_3)}{p_1}(-\frac{p_2p_3}{p_0p_1})^{\frac{p_3}{p_1-p_3}} > 0$ covering the range of $x^{\ft}\barx$. In this range $s_1 > 0$, $s_1^2-s_2s>0$, $R_4>0$ by the same calculation. For $R_{23}$, each coefficient of the polynomial in \cref{eq:DR23} is again positive, thus $R_{23} > 0$ and $R_1>0$ follows.
 
For the Lambert case, since $a_2 <0, a_1 >0$, we have that $\bs(-\infty)=-\infty$ and $\bs(u_c) = -\frac{a_1}{a_2}e^{a_2u_c} > 0$. Hence $(-\infty,-1]\subset (\bs(-\infty), \bs(u_c)]$ and $\bs$ maps tothe entire range $(-\infty, -1]$ of $x^{\ft}\barx$. In this range, $s_1 > 0$. Set $Z =e^{a_2u}$ then $a_1ue^{a_2u} = s- a_0Z$, and $Z'=a_2Z$, thus
      $$\begin{gathered}s_1 = a_1e^{a_2u} + (a_0+a_1u)a_2e^{a_2u} =a_1Z + a_2s, \\
      s_2 = a_1Z' + a_2s_1 = a_1a_2 Z + a_2(a_1Z + a_2s) = 2a_1a_2Z + a_2^2s,\\
      s_3 = 2a_1a_2Z' + a_2^2s_1 = 2a_1a_2^2Z + a_2^2(a_1Z + a_2s) = 3a_1a_2^2Z + a_2^3s.
      \end{gathered}
      $$
      Hence, since $s\leq -1$, we can show $R_1, R_{23}, R_4$ are positive from
$$\begin{gathered}      
      s_1^2-s_2s = a_1^2Z^2> 0,\\
      - s_2-ss_1^2+s_2s^2 = -a_1^2sZ^2 - 2a_1a_2Z - a_2^2s > 0,\\
      (s_3s_1-s_2^2)(1-s^2)=-a_1^2a_2^2Z^2(1-s^2) \geq 0.
\end{gathered}$$

For the affine case, $a_2=0, s_1=a_1, s_2 =s_3=0, D=-sa_1^9$, thus when $a_1>0$, it is clear $R_1, R_{23}, R_4$ are positive. Finally, the log case is known from \cref{cor:antena}.
\end{proof}  

We will show the proof of \cref{prop:power} in \cref{appx:proofpower}.

As a sanity check, and also to show the coefficients $R_1, R_{23}, R_4$ for a classical example, we make a connection with the antena cost. For $\rS^n$ with $\bu(s) = \alpha(1-(s+1)^{\frac{1}{\alpha}})$, the curvature in \cref{prop:power} is scaled by a factor of $\alpha$, thus when $\alpha$ goes to $\infty$, formally, $\bu(s)$ converges to $-\log(1+s)$, the {\it reflector antenna cost}, while the coefficients $R_1, R_{23}, R_3$, scaled by $\alpha$, all converge to $(s+1)^{-2}$, consistent with \cref{eq:antennaRS}.

For another connection to known cases, we now give the coefficients for the square Riemannian distance cost on the sphere, obtained with a symbolic calculation tool. The positivity of the coefficients could be confirmed by calculus. The original proof of positivity is in \cite{Loeper}.
\begin{proposition}On $\rS^n$, for the square Riemannian distance cost $\bcc(x, \barx) = \bu(s)=\frac{1}{2}\arccos^2(x^{\ft}\barx)$, $\bs(u) = \cos(\sqrt{2u})$, with $w=\sqrt{2u}$, $0\leq w < \pi$
\begin{gather}R_1 = \frac{4w^2 + w\sin 2w - 3 + 3\cos 2w}{w^2\sin^2w},\\
R_{23}=\frac{2(\sin w-w\cos w )}{\sin^3w},\\
R_4= \frac{w(2w - \sin(2w))}{2\sin^4w}.
  \end{gather}  
\end{proposition}
\section{Examples of the optimal map and $\bcc$-convexity}\label{sec:optimalexample}
Given a $C^2$ convex potential function, the optimal map for $\bu(x^{\ft}\barx)$ could be evaluated by solving a scalar equation. The solution may be written explicitly in terms of the potential for functions in \cref{theo:sol}. We will analyze the requirement for optimality to realize in the interior of the defining domain, as opposed to the case the supremum is unattainable inside an open domain, which depends on the situation we sometimes say it is attainable at a limit point on the boundary or at infinitive.

We recall some basic concepts in \cite{KimMcCann,Loeper,WongYang}. We work on $\R^n$ with the usual inner product, so $A_{\ft}$ is $I_n$. Recall the cost-exponential $\cexp:\R^n\times\R^n$ satisfies $\cD\bcc(x, \cexp(x, \nu)) = -\nu$ for $x, \nu\in \R^n$. The existence of the solution $\barx=\cexp(x, \nu)$ comes from condition A1 (the \emph{twisted condition}). In our case, this means $\barx=\cexp(x, \nu)$ satisfies $\bu_1(x^{\sfT}\barx)\barx = -\nu$, or $\barx$ is propotional to $\nu$, with the scaling parameter $t = -\bs_1(u_{\nu})$ where $u_{\nu}$ solves the scalar equation $\bs(u_{\nu}) = -\bs_1(u_{\nu}) x^{\sfT}\nu$ if $\nu\neq 0$ (since $x^{\sfT}(\bu_1(x^{\sfT}\barx)\barx) = -x^{\sfT}\nu$). For the costs in \cref{theo:sol}
\begin{equation}
\begin{gathered}u_{\nu} = \frac{1}{p_3-p_1}\log(\frac{p_0(1 + p_1x^{\sfT}\nu)}{-p_2(1+p_3x^{\sfT}\nu)})\text{ for }\bs(u) = p_0e^{p_1u} + p_2e^{p_3u},\\
  u_{\nu}= \frac{-(a_1+a_0a_2)x^{\sfT}\nu  - a_0}{a_1(1  + a_2x^{\sfT}\nu)} \text{ for }\bs(u) = (a_0 + a_1u)e^{a_2u},\\
u_{\nu} = \frac{1}{b_2}(\arctan(\frac{-b_2x^{\sfT}\nu}{1+b_1x^{\sfT}\nu}) - b_3 + k\pi) \text{ for }\bs(u) = b_0e^{b_1u}\sin(b_2u + b_3).\label{eq:solvup}
\end{gathered}
\end{equation}
In the last equation, $k$ is choosen so that $u_{\nu}$ and $x^{\sfT}\barx$ are in the correct range. Even when $u_{\nu}$ is computable from these equations, if $\bs$ has several branches of inverse, we also need to verify $u_{\nu}$ belongs to the correct branch.

For potentials $\phi$ on $\cM$, $\psi$ on $\bcM$, the $\bar{\bcc}$- and $\bcc$-transforms \cite{Loeper} of $\phi$ and $\psi$ are
\begin{equation}\begin{gathered}
    \phi^{\bar{\bcc}}(\barx) = \sup_x (-\bcc(x, \barx)-\phi(x)),\\
    \psi^{\bcc}(x) = \sup_{\barx} (-\bcc(x, \barx)-\psi(\barx)).
\end{gathered}
\end{equation}
The function $\phi$ is $\bcc$-convex if $\phi^{\bar{\bcc}\bcc} = \phi$, or $\phi = \psi^{\bcc}$ for some $\psi$. 
Assume $\phi$ is sufficiently smooth and let $\grad_{\phi}$ be the gradient of $\phi$. If $\sup_x (-\bcc(x, \barx)-\phi(x))$ is attained inside the defining domain, the first-order condition is $\cD\bcc(x, \barx) =- \grad_{\phi}(x)$. Define
\begin{equation}  
\bT: x\mapsto x_{\bT} = -\bs_1(u_{\phi})\grad_{\phi}(x).\label{eq:Tuphi}
\end{equation}
At a critical point $x$ of $-\bcc(x, \barx)-\phi(x)$, we have $\bT x = \barx$. From \cite[Theorem 2.9]{GuillenMcCann}, a solution to the transport problem in that theorem (and mentioned in the introduction) has the form $\bT(x)$ for a $\bcc$-convex function $\phi$, and $\bT$ is called the \emph{optimal map}.

The second-order condition is $\hess^{\bcc}_{\phi}(x) = \rD^{2}\bcc(x, \barx)_{\barx=x_{\bT}} + \hess_{\phi}(x)$ is positive semidefinite.
From $\rD^{2}\bcc(x, \barx)_{\barx=x_{\bT}} = \bu_2(x^{\sfT}x_{\bT})x_{\bT}x_{\bT}^{\sfT} =-\frac{\bs_2(u_{\phi})}{\bs_1(u_{\phi})}\grad_{\phi}(x)\grad_{\phi}(x)^{\sfT}$ and from \cref{eq:secondord},
$$-\frac{\bs_2(u_{\phi})}{\bs_1(u_{\phi})} = P\frac{\bs(u_{\phi})}{\bs_1(u_{\phi})} -S =
-Px^{\sfT}\grad_{\phi}(x) - S$$
with constants $P$ and $S$ as in \cref{theo:sol}. Thus, the second order condition is
\begin{equation}\hess^{\bcc}_{\phi}(x) = (-Px^{\sfT}\grad_{\phi}(x) - S)\grad_{\phi}(x)\grad_{\phi}(x)^{\sfT}+ \hess_{\phi}(x)\succeq 0.\label{eq:secondorderphi}\end{equation}
For the classical case $\bs(u) = -u$, this reduces to $\hess_{\phi}\succeq 0$, or convexity of $\phi$. For the log-type cost $P=p_1p_3=0$, this reduces to the convexity condition of $-(p_1+p_3)e^{-(p_1+p_3)\phi}$, following \cite[Theorem 1]{WongInfo}.

{\it We will use the notation $x_{\bT}= \bT x$ for the right-hand side of \cref{eq:Tuphi} even when $\phi$ is not known to be convex, if  $u_{\phi}$ is uniquely solved}. 

Simplest examples of $\bcc$-convex functions are studied in \cite{Loeper}. Let $I$ be a subset of $\cM$ and $x_0\in\cM$. Define $\hat{\psi}(\barx) = -\bcc(x_0, \barx)$ for $\barx\in I$, and $+\infty$ otherwise. Then
\begin{equation}\phi(x) := \hat{\psi}^{\bcc}(x) = \sup_{\barx\in I}\{-\bcc(x, \barx) +\bcc(x_0, \barx)\},
\end{equation}  
and $\phi$ is $\bcc$-convex. If $I$ is finite then $\phi(x) = \max_{\barx_i\in I}\{-\bcc(x, \barx_i) +\bcc(x_0, \barx_i)\}$, and in particular, when $\lvert I\rvert = 2$, then we get the function in \cite[Definition 2.8]{Loeper}.

\subsection{A sub-family of $\sinh$-type costs}
We now provide other examples of a $\bcc$-convex potential of $\sinh$-type, in particular, when
$\bs(u) = p_0e^{-ru}+ p_2e^{ru}$ (thus $p_3 =r=-p_1>0$), and $p_0> 0, p_2 < 0$. Then 
\begin{equation}\bu(s) = \frac{1}{r}\log(\frac{-s+\sqrt{s^2-4p_0p_2}}{2 |p_2|})
= -\frac{1}{r}\log(\frac{s+\sqrt{s^2-4p_0p_2}}{2p_0}).
\label{eq:sinhtypedef}
\end{equation}
Note that $\bs(u) = -2(-p_0p_2)^{\frac{1}{2}}\sinh(ru - \frac{1}{2}\log(\frac{p_0}{-p_2}))$. In this case, equation (\ref{eq:solvup}) becomes $u_{\nu} = \frac{1}{2r}\log(\frac{p_0(1-rx^{\sfT}\nu)}{-p_2(1+rx^{\sfT}\nu)})$. A special case is $\bs(u) = -\frac{\sinh(r u)}{r}$, or $p_0 = -p_2 =\frac{1}{2r}$, where $\bcc$ converges to the classical convex cost $-x^{\sfT}\barx$ when $r$ goes to $0$.

Let $\tphi$ be a $C^1$ $\bcc$-convex function on an open subset of $\R^n$. To attain optimality, we need $u_{\nu}$ to be defined for $\nu=\grad_{\phi(x)}$, or $|x^{\sfT}\grad_{\tphi}(x)| < \frac{1}{r}$. Simplify the factor $-\bs_1(u_{\grad_{\tphi}(x)})$, the map $\bT$ and the second order condition \cref{eq:secondorderphi} are given by
\begin{gather}\bT(x) = \cexp(x; \grad_{\tphi}(x)) = \frac{2r(-p_2p_0)^{\frac{1}{2}}}{(1-r^2(x^{\sfT}\grad_{\tphi}(x))^2)^{\frac{1}{2}}}\grad_{\tphi}(x),\\
\hess^{\bcc}_{\tphi}(x) =  \hess_{\tphi}(x)+ (r^2x^{\sfT}\grad_{\tphi}(x)) \grad_{\tphi}(x)\grad_{\tphi}(x)^{\sfT}\succeq 0.\label{eq:hesschyper}
\end{gather}
We now consider {\it absolutely-homogeneous convex functions}. The case of order $\alpha >1$ will be used to construct a sampling algorithm and a local divergence. The case of order $1$ is in \cref{sec:convexOrder1}.

\subsection{Absolutely-homogeneous function of order $\alpha > 1$}\label{ex:homogenousgeq1}
  Let $\tphi$ be a strictly convex, continuous, absolutely-homogeneous function of homogenous degree $\alpha > 1$ (thus, $\tphi(tx) = |t|^{\alpha}\tphi(x)$), and assume $\tphi$ is $C^1$ at $x\neq 0$ (in particular, $\tphi(x) > 0$ for $x\neq 0$). We have $x^{\sfT}\grad_{\tphi}(x) = \alpha \tphi(x)$. Further, assume $x\mapsto \grad_{\tphi}(x)$ is invertible for $x\in \R^n$, $x\neq 0$. This example covers powers of norms $|x|^{\alpha}_{\alpha}$ and the positive-definite quadratic forms $\frac{1}{2}x^{\sfT}C x$, and there is a rich family of nonnegative homogeneous convex polynomials.

Differentiate $\tphi(tx) = |t|^{\alpha}\tphi(x)$ in $x$, we get $\grad_{\tphi}(tx) = \frac{|t|^{\alpha}}{t}\grad_{\tphi}(x) = \sign(t)|t|^{\alpha-1}\grad_{\tphi}(x)$ if $t\neq 0$. This implies $\grad_{\tphi}^{-1}(t\barx) =\sign(t)|t|^{\frac{1}{\alpha-1}}\grad_{\tphi}^{-1}(\barx)$ if $t\neq 0$.  
\begin{proposition}\label{prop:homealphag1}
  Let $\tphi$ be a strictly convex, continuous, absolutely-homogeneous function $\tphi$ of order $\alpha>1$ in $\R^n$ with invertible $\grad_{\tphi}$ for $x\neq 0$. For all $\barx\in \R^n$, the function $L(., \barx): \R^n\mapsto \R$ defined by $x\mapsto \bu(x^{\sfT}\barx) + \tphi(x)$ has a unique global minimum  $x_{\opt,\barx}$. If $\barx\neq 0$, $x_{\opt,\barx}$ is the unique root of the equation:
\begin{equation}\frac{2r(-p_0p_2)^{\frac{1}{2}}}{(1-\alpha^2r^2\tphi(x)^2)^{\frac{1}{2}}}\grad_{\tphi}(x) = \barx\label{eq:homgeq1}
\end{equation}
and for $\barx=0, x_{\opt, \barx}=0$. Thus, $\psi(\barx) :=\tphi^{\bar{\bcc}}(\barx)$ is defined for all $\barx\in\R$ as $-L(x_{\opt, \barx},\barx)$ with $\barx^{\sfT}\grad_{\psi}(\barx)<\frac{1}{r}$ and the function $\phi(x)$ below is $\bcc$-convex
\begin{equation}\phi(x):=\tphi^{\bar{\bcc}\bcc}(x) =\begin{cases}\tphi(x)\text{ if } \tphi(x) <\frac{1}{r\alpha}, \\ \frac{1}{r\alpha}(1+\log(r\alpha\tphi(x))\text{ otherwise.}\end{cases}\end{equation}
\end{proposition}
\begin{proof}
For $\Delta > 0$, consider the hypersurface $H_{\Delta}$ defined by $\tphi(x) =\frac{\Delta}{\alpha}$ bounding the convex region $\tphi(x) \leq \frac{\Delta}{\alpha}$. A critical point of $L(., \barx)$ on $H_{\Delta}$ solves $u_1(x^{\sfT}\barx )\barx + \grad_{\tphi}(x) = \lambda \grad_{\tphi}(x)$, with $\bu_1(s) = -\frac{1}{r(s^2-4p_0p_2)^{\frac{1}{2}}}<0$. Thus, $\barx\neq 0 $ is proportional to $\grad_{\tphi}(x)$, $t\barx = \grad_{\tphi}(x)$ for $t\in\R$, and $x= \grad_{\tphi}^{-1}(t\barx)$ with $t$ solving the equation:
$$t\barx^{\sfT}\grad_{\tphi}^{-1}(t\barx) = x^{\sfT}\grad_{\tphi}(x) =\alpha\tphi(x)= \Delta.$$
Hence, $|t|^{1+\frac{1}{\alpha-1}}\barx^{\sfT}\grad_{\tphi}^{-1}(\barx) = \Delta$. Therefore, in general, there are two critical points corresponding to the maximum and minimum, with $t=\pm(\frac{\Delta}{\barx^{\sfT}\grad_{\tphi}^{-1}(\barx)})^{\frac{\alpha-1}{\alpha}}$, we get
$$\begin{gathered}
  x_{\text{critical}} = \pm(\frac{\Delta}{\barx^{\sfT}\grad_{\tphi}^{-1}(\barx)})^{\frac{1}{\alpha}}\grad_{\tphi}^{-1}(\barx),\\
  s = \pm\Delta^{1/\alpha}(\barx^{\sfT}\grad_{\tphi}^{-1}(\barx))^{\frac{\alpha-1}{\alpha}}.
\end{gathered}$$
Since $\bu$ is decreasing, the minimum $L_{\opt}(\Delta;\barx)$ corresponds to the positive sign  
$$L_{\opt}(\Delta) = L_{\opt}(\Delta; \barx)=L(x_{\min, \Delta}, \barx) = \bu(\Delta^{1/\alpha}(\barx^{\sfT}\grad_{\tphi}^{-1}(\barx))^{\frac{\alpha-1}{\alpha}}) + \frac{\Delta}{\alpha}.$$
With $K = K(\barx) = (\barx^{\sfT}\grad_{\tphi}^{-1}(\barx))^{\frac{\alpha-1}{\alpha}}$, the derivative $\frac{d}{d\Delta}L_{\opt}(\Delta)$ is
$$-\frac{1}{\alpha r}(\Delta^{2/\alpha}K^2-4p_0p_2)^{-\frac{1}{2}} \Delta^{1/\alpha-1}K + \frac{1}{\alpha}= -\frac{K}{\alpha r}(\Delta^{2}K^2-4p_0p_2\Delta^{2-\frac{2}{\alpha}})^{-\frac{1}{2}} +\frac{1}{\alpha},$$
which increases with $\Delta$ (note, $p_0p_2<0$). Therefore, $L_{\opt}$ is convex as a function of $\Delta$ with a global minimum, the unique positive root $\Delta_{\opt}$ of $\frac{d}{d\Delta}L_{\opt}(\Delta) = 0$, or
\begin{equation}K(\barx)^2\Delta_{\opt}^2-4p_0p_2\Delta_{\opt}^{\frac{2\alpha-2}{\alpha}} - \frac{K(\barx)^2}{r^2}=0.\label{eq:Deltamin}\end{equation}
Thus, $\tphi^{\bar{\bcc}}(\barx) = -L_{\opt}(\Delta_{\opt}(\barx))$ with $\Delta_{\opt}(\barx)$ solving \cref{eq:Deltamin}, and $\tphi^{\bar{\bcc}}(\barx)$ is $\bcc$-convex and defined for all $\barx\in\R^n$, including $\tphi^{\bar{\bcc}}(0) = 0$.

Since $L(x, \barx)\geq L_{\opt}(\Delta_{\opt}(\barx))=-\tphi^{\bar{\bcc}}(\barx)$, $\bu(x^{\sfT}\barx) +\tphi^{\bar{\bcc}}(\barx)+ \tphi(x)\geq 0$ for all $x$ and $\barx$, with equality if and only $x$ solves \cref{eq:homgeq1}. Hence, for fixed $x$, the function $R(x, .): \barx\mapsto -\bu(x^{\sfT}\barx) -\tphi^{\bar{\bcc}}(\barx)$ is bounded above by $\tphi(x)$. For \cref{eq:homgeq1} to be well-defined, we need $\tphi(x) < \frac{1}{\alpha r}$. Conversely, if we have this condition, then the upper bound of $R(x, .)$ is reached with $\barx$ in \cref{eq:homgeq1}, hence, $\tphi^{\bar{\bcc}\bcc}(x) = \tphi(x)$.

Assume $\tphi(x) \geq \frac{1}{\alpha r}$, then $R(x, \barx) = -\bu(x^{\sfT}\barx) + \bu(x_{\opt, \barx}^{\sfT}\barx)+ \tphi(x_{\opt, \barx})$, thus, with $z=x_{\opt, \barx}$
\begin{equation}\tphi^{\bar{\bcc}\bcc}(x) = \sup_{\tphi(z) < \frac{1}{r\alpha }}\{-\bu(\frac{2r(-p_0p_2)^{\frac{1}{2}}x^{\sfT}\grad_{\tphi}(z)}{(1-r^2\alpha^2 \tphi(z)^2)^{\frac{1}{2}}}) + \bu(\frac{2r(-p_0p_2)^{\frac{1}{2}}\alpha \tphi(z)}{(1-r^2\alpha^2\tphi(z)^2)^{\frac{1}{2}}})+ \tphi(z) \}.
  \label{eq:phicchom}
\end{equation}
We show the supremum is approached at a limit point on the boundary $\tphi(z) = \frac{1}{r\alpha}$. Consider again the hypersurface $\tphi(z) = \Delta <\frac{1}{r\alpha}$, since $-\bu$ is increasing, we need to maximize $x^{\sfT}\grad_{\tphi}(z)$ subjected to $\tphi(z) = \Delta$. Let $z_{\Delta} := (\frac{\Delta}{\tphi(x)})^{\frac{1}{\alpha}}x$ on this hypersurface. Since $\tphi$ is strictly convex, $\tphi(z_{\Delta}) -\tphi(z) - (z_{\Delta}-z)^{\sfT}\grad_{\tphi}(z)\geq 0$, hence $(z_{\Delta}-z)^{\sfT}\grad_{\tphi}(z)\leq 0$, thus
$$x^{\sfT}\grad_{\tphi}(z) = (\frac{\tphi(x)}{\Delta})^{\frac{1}{\alpha}} z_{\Delta}^{\sfT}\grad_{\tphi}(z)\leq (\frac{\tphi(x)}{\Delta})^{\frac{1}{\alpha}}\alpha\Delta
$$
with equality only at $z=z_{\Delta}$, and the quantity to maximize in \cref{eq:phicchom} is
$$-\bu(s_x) + \bu(s_z)+ \Delta =\frac{1}{r}\log(\frac{s_x +(s_x^2-4p_0p_2 )^{\frac{1}{2}}}{s_z+(s_z^2-4p_0p_2)^{\frac{1}{2}}}) + \Delta
$$
with $s_z = \frac{2r(-p_0p_2)^{\frac{1}{2}}\alpha \Delta}{(1-r^2\alpha^2\Delta^2)^{\frac{1}{2}}},
s_x=\frac{2r(-p_0p_2)^{\frac{1}{2}}(\frac{\tphi(x)}{\Delta})^{\frac{1}{\alpha}}\alpha\Delta}{(1-r^2\alpha^2 \Delta^2)^{\frac{1}{2}}}$. As $\Delta$ approaches $\frac{1}{r\alpha}$, both $s_x$ and $s_z$ are unbounded but the quantity approaches
$$\begin{gathered}\lim_{\Delta\to\frac{1}{r\alpha}}\frac{1}{r}\log(\frac{s_x}{s_z}\frac{1+ (1-\frac{4p_0p_2}{s_x^2})^{\frac{1}{2}}}{1+(1-\frac{4p_0p_2}{s_z^2})^{\frac{1}{2}}}) + \Delta
  =\frac{1}{r}\log(\frac{\tphi(x)}{(r\alpha)^{-1}})^{\frac{1}{\alpha}}+\frac{1}{r\alpha}
  =\frac{1}{r\alpha}\log(r\alpha\tphi(x)) +\frac{1}{r\alpha}.
\end{gathered}$$
The supremum is not at an interior point in this case, otherwise, as just argued, the maximum $\barx$ for $R(., x)$ is related to $x$ by \cref{eq:homgeq1}, contradicting $\tphi(x) \geq \frac{1}{\alpha r}$. Thus, the above is the supremum.
\end{proof}
\begin{remark}\label{rem:quadraticPotential}
  When $\tphi(x) = \frac{1}{2}x^{\sfT}Cx$ for a positive definite matrix $C$, then $g_{\tphi}(x) = Cx$. Let $W = \frac{1}{2}\barx^{\sfT}C^{-1}\barx$ then $x_{\opt} = \frac{C^{-1}\barx}{\left(-2p_0p_2r^2 + 2\sqrt{(p_0p_2r^2)^2 + r^2W^2}\right)^{\frac{1}{2}}}$, thus
$$\begin{gathered}\tphi^{\bar{\bcc}}(\barx) = -\frac{1}{2r}\log(\frac{\sqrt{(rp_0p_2)^2+ W^2} - W}{p_2^2r})
-   \frac{W}{2(-p_0p_2r^2 + \sqrt{(p_0p_2r^2)^2 + r^2W^2})}
\end{gathered}.$$
In general, $\tphi^{\bar{\bcc}}(\barx)$ could be expressed as function of $\barx^{\sfT}\grad_{\tphi}^{-1}(\barx)$ using \cref{eq:Deltamin}.
\end{remark}
For $n=1$, in \cref{fig:dual}, on the top, we graph examples from $\varphi(x)=|x|^{\alpha}$. The defining domain of the optimal map is $|x| < (r\alpha)^{-\frac{1}{\alpha}}$. The Legendre conjugate (the classical convex conjugate corresponding to $r\to 0$) is $\phi^*(\barx) =\alpha^{-\frac{\alpha}{\alpha-1}}(\alpha - 1)|\barx|^{\frac{\alpha}{\alpha-1}}$. We plot the $\bcc$-conjugates on the top part for $\alpha=1.8$, and observe for small $r$ the hyperbolic conjugates are close to the Legendre conjugate. In the bottom part, we plot another example together with the optimal maps. Note that the optimal maps go to infinitive as $x$ approaches the boundaries $\pm(r\alpha)^{-\frac{1}{\alpha}}$ at different speeds for different $r$'s.
\begin{figure}[ht!]
  \centering
  \includegraphics[width=.5\textwidth]{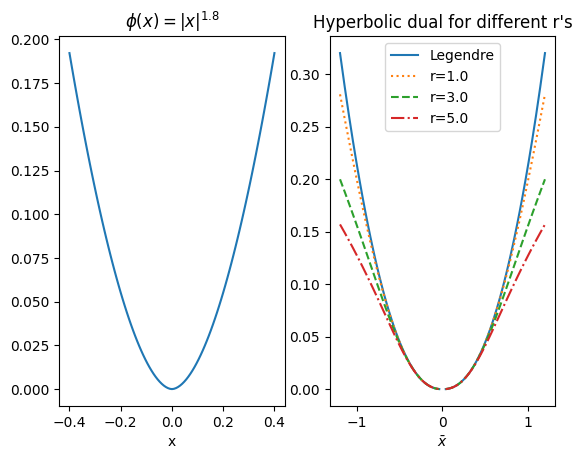}  
  \includegraphics[width=\textwidth]{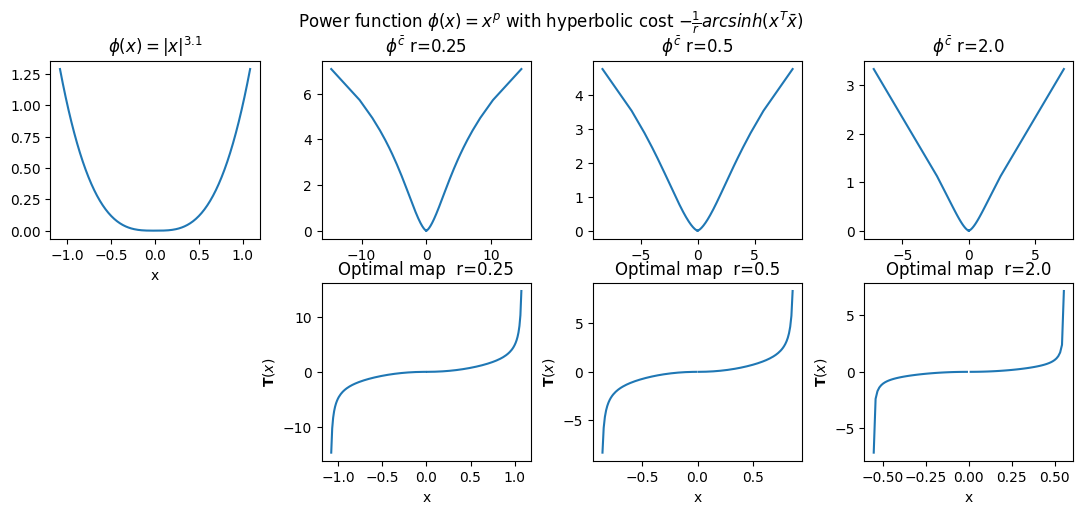}
  \caption{Optimal map and $\bcc$-conjugate functions under hyperbolic cost $-\frac{1}{r}\arcsinh(rx^{\sfT}\barx)$. Top: hyperbolic conjugates for different $r$ together with the Legendre conjugate. Bottom: conjugates and optimal map for different $r$. The ranges of $\barx$ correspond to the portion of the $x$-grid inside the domain of $\bT$.}
    \label{fig:dual}
\end{figure}
\section{Divergence}\label{sec:divergence}Following \cite{WongYang}, (the result there is in terms of concave functions, we restate it in terms of convex functions), given a $\bcc$-convex function $\phi$ then the $\bcc$-divergence
\begin{equation}\fD[x: x'] := \bcc(x, x'_{\bT}) + \phi(x) + \phi^{\bar{\bcc}}(x'_{\bT})
=\bcc(x, x'_{\bT}) + \phi(x)- \bcc(x', x'_{\bT}) -\phi(x')\label{eq:cdivergence}
  \geq 0\end{equation}
is non-negative for two points $x, x'\in\cM$. If $\bT$ is one-to-one and optimality is attainabled at a single point then equality is only at $x = x'$. Thus, if $\bcc$ is the cost in \cref{eq:sinhtypedef}, with $\bT$ given by the left-hand side of \cref{eq:homgeq1}, then in the region $|x^{\sfT}\grad_{\phi}(x)| < \frac{1}{r}$ set $v = x^{\sfT}\grad_{\phi}(x'), w = (x')^{\sfT}\grad_{\phi}(x')$, then
$$\begin{gathered}
(x')^{\sfT}x'_{\bT} + (((x')^{\sfT}x'_{\bT})^2-4p_0p_2)^{\frac{1}{2}}
  =\frac{2r(-p_0p_2)^{\frac{1}{2}}w}{(1-r^2w^2)^{\frac{1}{2}}}
  + \left(\frac{4r^2(-p_0p_2)w^2}{(1-r^2w^2)}-4p_0p_2 \right)^{\frac{1}{2}}\\
  = \frac{2(-p_0p_2)^{\frac{1}{2}}(r w+ 1 )}{(1-r^2w^2)^{\frac{1}{2}}},\\
  x^{\sfT}x'_{\bT} + ((x^{\sfT}x'_{\bT})^2-4p_0p_2)^{\frac{1}{2}} =
  \frac{2r(-p_0p_2)^{\frac{1}{2}}v}{(1-r^2w^2)^{\frac{1}{2}}} \
  + (\frac{4r^2(-p_0p_2)v^2}{(1-r^2w^2)}
  -4p_0p_2)^{\frac{1}{2}}\\
 = \frac{2(-p_0p_2)^{\frac{1}{2}}}{(1-r^2w^2)^{\frac{1}{2}}}(rv + (r^2(v^2 - w^2)+1)^{\frac{1}{2}}).
\end{gathered}$$

Hence, for $\bcc$-convex (not necessarily homogeneous) $\phi$, $\fD$ becomes
\begin{equation}\label{eq:divergence}\phi(x) - \phi(x')- \frac{1}{r}\log
  \frac{rx^{\sfT}\grad_{\phi}(x') + \left(r^2(x^{\sfT}\grad_{\phi}(x'))^2
    - r^2((x')^{\sfT}\grad_{\phi}(x'))^2+1\right)^{\frac{1}{2}}}{
    r(x')^{\sfT}\grad_{\phi}(x')+ 1 }.
\end{equation}
This is a hyperbolic version of the Bregman divergence $\phi(x) - \phi(x') - \grad_{\phi}(x')^{\sfT}(x - x') \geq 0$, the latter is defined for classically strictly convex functions. Assuming
$|rx^{\sfT}\grad_{\phi}(x)| < 1$ for $x$ in the defining domain, this is equivalent to
\begin{equation}\frac{\sinh(r(\phi(x)-\phi(x')) }{r} + \cosh(r(\phi(x)-\phi(x'))(x')^{\sfT}\grad_{\phi}(x') - x\grad_{\phi}(x') \geq 0. \label{eq:divergence2}
\end{equation}
To see this, if $|r(x')^{\sfT}\grad_{\phi}(x')| < 1$ then \cref{eq:divergence} is nonnegative and is equivalent to $0 < K \leq e^{r(\phi(x) - \phi(x'))}$, where $K$ is the positive root of
$$f(X) :=(r(x')^{\sfT}\grad_{\phi}(x')+ 1)X^2 - 2rx^{\sfT}\grad_{\phi}(x')X + (r(x')^{\sfT}\grad_{\phi}(x')- 1).
$$
Since the quadratic function $f$ has exactly one positive root, this is equivalent to $\frac{1}{2r}e^{-r(\phi(x) - \phi(x'))}f(e^{r(\phi(x) - \phi(x'))}) \geq 0$, which simplifies to \cref{eq:divergence2}.
\subsection{The dualistic geometry of $\bcc$-divergences}\label{sec:dualistic}

We start with some general results on $\bcc$-divergences. Let $\bcc$ be a regular cost. For simplicity, we will consider the case where $\cM$ and $\bcM$ are domains $\Omega$ and $\bOmg$ of $\R^n$. If a function $\phi$ is $\bcc$-convex, and for each $\barx\in \R^n$, the maximum $-\bcc(x, \barx) - \phi(x)$ exists at a unique inner point of the domain $\Omega$, then the correspondence between $\barx$ and that maximal point is a map, we assume its inverse is well-defined and is the optimal map $\bT=\bT_{\phi}$. Thus, $\bT_{\phi}x$ is defined as the point $\barx$ such that $\bcc(x, \barx) + \phi(x)$ is minimized at $x$. Recall the first-order and second-order conditions in \cref{sec:optimalexample}
\begin{gather} \cD\bcc(x, x_{\bT}) +\grad_{\phi}(x) = 0\label{eq:definebT},\\
 \hess^{\bcc}_{\phi}(x): \xi\mapsto  \rD_{\xi}\cD\bcc(x, x_{\bT}) +\hess_{\phi}(x)\xi  \succeq 0\text{ for }\xi \in \R^n.  \label{eq:hesscgeneral}
\end{gather}
We call a $\bcc$-convex potential $\phi$ is of $\bcc$-\emph{Legendre-type} if $\bT_{\phi}$ is a diffeomorphism (a smooth bijection with smooth inverse) between $\Omega$ and $\bOmg$, and $\hess^{\bcc}_{\phi}(x)$ is \emph{positive-definite}.
In this case, the $\bcc$-conjugate $\phi^{\barbcc}$ is given by $\phi^{\barbcc}(\barx) = - \bcc(\bT^{-1}_{\phi}(\barx), \barx) -\phi(\bT^{-1}_{\phi}(\barx))$. By $\bcc$-convexity, $\phi^{\barbcc\bcc}=\phi$, we have $\bT^{-1}_{\phi} = \bT^{\barbcc}_{\phi^{\barbcc}}$, where $\bT^{\barbcc}_{\phi^{\barbcc}}$ is defined as the transport map with cost $\barbcc(\barx, x) = \bcc(x, \barx)$ and $\phi^{\barbcc}$ in place of $\phi$. We now have a $\bcc$-convex version of the classical result \cite{Crouzeix}
\begin{proposition}\label{prop:cdivmetric}Let $\phi$ be a $\bcc$-convex potential of $\bcc$-Legendre type of the cost $\bcc(x, \barx)$ on $\cN\subset \Omega\times\bOmg\subset \cE\times \cE$. Assume $\hcN$ below is an open subset of $\cE\times\cE$
  \begin{equation}  \hcN= \{ (x, x') \in \Omega\times\Omega| (x, x'_{\bT})\in \cN\}.
  \end{equation}
Equip $\hcN$ with the Kim-McCann metric $KM_{\fD}$ of the $\bcc$-divergence cost $\fD$ in \cref{eq:cdivergence}. Then for a tangent vector $\hat{\xi} = (\xi, \xi')$ of $\hcN$ at $(x, x')$, we have
\begin{equation}\langle \hat{\xi}, \hat{\xi}\rangle_{KM_{\fD}} = - \brD_{d\bT(x')\xi'}\rD_{\xi}\bcc(x, x'_{\bT}).\label{eq:KMdiver}
\end{equation}
Here, $d\bT$ is the differential of $\bT$, given by the implicit function theorem
\begin{equation}\begin{gathered}
    \hess^{\bcc}_{\phi}(x') + \brD_{d\bT(x')\xi'}\cD\bcc(x', x'_{\bT})=0,\\
    d\bT(x')\xi'= - (\brD_{}\cD\bcc(x', x'_{\bT}))^{-1}\hess^{\bcc}_{\phi}(x').\\
    \label{eq:cdT}
\end{gathered}    
\end{equation}
The restriction of $KM_{\fD}$ to the diagonal $(x, x)\in\hcN, x\in\Omega$ is the divergence metric of $\fD$, given by $\hess^{\bcc}_{\phi}$ in \cref{eq:hesscgeneral}. We have the $\bcc$-Crouzeix identity
\begin{equation}\begin{gathered}
    d\bT_{\phi}(x) d\bT^{\barbcc}_{\phi^{\bar{\bcc}}}(x_{\bT}) = I_{\bar{\Omega}},\\
    \left((\brD\cD\bcc)^{-1}\hess^{\bcc}_{\phi}(\rD\bcD\bcc)^{-1}\hess^{\barbcc}_{\phi^{\bar{\bcc}}}\right)_{x, x_{\bT}} = I_{\bar{\Omega}}.\label{eq:crouzeix}
\end{gathered}    
\end{equation}
\end{proposition}
\begin{proof}Let $\rD'$ denote the differentiation in the second variable $x'$ in $\hcN$, \cref{eq:KMdiver} follows from the definition $\langle \hat{\xi}, \hat{\xi}\rangle_{KM_{\fD}} = - \rD'_{\xi'}\rD_{\xi}\fD$ and the chain rule. Equation (\ref{eq:cdT}) follows by differentiating \cref{eq:definebT} in $x$ at $x=x'$.
  
For $\xi\in\R^n$, the divergence metric is
  $$\begin{gathered}  -(\rD'_{\xi}\rD_{\xi}\fD)_{x'=x} = (-\rD'_{\xi}(x'\mapsto \rD_{\xi}\bcc(x, x'_{\bT}) + \xi^{\sfT}\grad_{\phi(x)})_{x'=x}
    =\\ (-\brD_{d\bT(x')\xi'}(\rD_{\xi}\bcc(x, x'_{\bT})))_{x'=x} = \hess^{\bcc}_{\phi}(x)
  \end{gathered}    $$
using the first line of \cref{eq:cdT}. The first line of \cref{eq:crouzeix} follows from $\bT^{-1}_{\phi} = \bT^{\barbcc}_{\phi^{\barbcc}}$, and the second is its expansion.
\end{proof}
For the cost $\bcc(x, \barx) = - x^{\sfT}\barx$, $\brD\cD\bcc=I$ and $\hess^{\bcc}_{\phi} =\hess_{\phi}$, we get the classical Crouzeix identity \cite{Crouzeix}.
\begin{proposition}
On $\hcN\subset \Omega\times\Omega$ with the metric $KM_{\fD}$, we denote the two components of the Levi-Civita connection $\nabla^{KM_{\fD}}$ as $\nabla^1$ and $\nabla^{-1}$ and their Christoffel functions as $\Gamma^1$ and $\Gamma^{-1}$. On the diagonal, they are the pair of dualistic connections of the divergence $\fD$, called the primal and dual connections. We have
\begin{equation}\begin{gathered}\Gamma^1(x, x;\xi_1, \xi_2) = ((\rD\bcD\bcc)^{-1}\rD_{\xi_1}\rD_{\xi_2}\bcD\bcc)(x, x_{\bT}).\label{eq:GammaD1}
  \end{gathered}\end{equation}
\end{proposition}
\begin{proof}
Comparing \cref{eq:GammaKM} and \cite{Eguchi} (and also \cite{WongYang,LFV}), we see $\Gamma^1$ and $\Gamma^{-1}$ are components of the Levi-Civita connection of the Kim-McCann metric on the diagonal. The proof of \cref{eq:GammaD1} is by direct substitution as follows. The gradient $\cD'\fD$ in $x'$ is $(d\bT(x'))^{\sfT}\bcD\bcc(x, x'_{\bT}) - \grad_{\phi}(x') - \grad_{x'\mapsto -\bcc(x', x'_{\bT})}(x'),$
  \begin{equation}\rD_{\xi_2}\cD'\fD(x, x') = (d\bT(x'))^{\sfT}\rD_{\xi_2}\bcD\bcc(x, x'_{\bT})
=    \hess_{\phi}^{\bcc}(x')\rD \bcD\bcc(x', x'_{\bT})^{-1}\rD_{\xi_2}\bcD\bcc(x, x'_{\bT}),
  \end{equation}
  $$\begin{gathered}  \rD_{\xi_1}\rD_{\xi_2}\cD'\fD(x, x) = (\hess_{\phi}^{\bcc}(x')\rD\bcD\bcc(x', x'_{\bT})^{-1}\rD_{\xi_1}\rD_{\xi_2}\bcD\bcc(x, x'_{\bT}))_{x'=x}\\
    =(\hess_{\phi}^{\bcc}(x)\rD\bcD\bcc(x, x_{\bT})^{-1}\rD_{\xi_1}\rD_{\xi_2}\bcD\bcc(x, x_{\bT})).
  \end{gathered}$$
Thus, $\Gamma^1=(\rD\cD'\fD)^{-1}\rD_{\xi_1}\rD_{\xi_2}\cD'\fD(x, x)$ simplifies to \cref{eq:GammaD1}.
\end{proof}
For $\bcc(x, \barx)= -x^{\sfT}\barx$, we recover $\Gamma^1 =0$ for the Bregman divergence. While in general, these primal geodesics are not of closed-form, they could be solved numerically in a two-dimensional subspace, as seen from the following corollary
\begin{corollary}If $\bcc$ is of the form $\bcc(x, \barx) = \bu(x^{\ft}x)$, a solution of the primal geodesic equation $\ddot{x} + \Gamma^1(x, \dot{x}, \dot{x})=0$ of the $\bcc$-divergence $\fD$ with initial condition $x(0) = x_0, \dot{x}(0) = \dot{x}_0$ stays in the subspace spanned by $x_0$ and $\dot{x}_0$ .
\end{corollary}
This follows from \cref{eq:GammaD1} and \cref{eq:sGammax}, as $\ddot{x}$ is a linear combination of $x$ and $\dot{x}$.

In \cref{eq:dualistichyper} in the introduction, we also compute the $\Gamma^{-1}$ component for the hyperbolic case.
For $r=0$, we recover the Hessian metric of $\phi$ and the dualistic pair of connections of the Bregman divergence $\Gamma^1(\xi_1, \xi_2)=0, \Gamma^{-1}(\xi_1, \xi_2) = \hess_{\phi}^{-1}\phi^{(3)}(\xi_1, \xi_2)$.
\begin{proof}(of \cref{eq:dualistichyper})The statement on the divergence metric follows from \cref{prop:cdivmetric}. Since \cref{eq:divergence} is only dependent on $r$, we can take $\bs(u) = -\frac{1}{r}\sinh(ru)$, then $x_{\bT} = \frac{1}{(1-r^2(x^{\sfT}\ttg)^2)^{\frac{1}{2}}}\ttg$, and $s_0 = x^{\sfT}x_{\bT}=\frac{x^{\sfT}\ttg}{(1-r^2(x^{\sfT}\ttg)^2)^{\frac{1}{2}}}$, while $\cosh^2ru -\sinh^2 ru =1$ implies
  $s_1 = -\cosh (ru) = - (r^2s_0^2 +1)^{\frac{1}{2}} = \frac{-1}{(1- r^2(x^{\sfT}\ttg)^2)^{\frac{1}{2}}}, s_2 = r^2s_0, s_3 = r^2s_1 $, $s_0s_2-s_1^2=-1$, 
  $s_1s_3- s_2^2=r^2$. Thus, $\Gamma^1(\xi_1, \xi_2)$ evaluates to $\Gamma_x$ in \cref{eq:sGammax} at $(x, x_{\bT})$ to
  $$   \Gamma^1(\xi_1, \xi_2) = -r^2(1-r^2(x^{\sfT}\ttg)^2)\frac{\ttg^{\sfT}\xi_1\ttg^{\sfT}\xi_2}{1-r^2(x^{\sfT}\ttg)^2}x    - r^2x^{\sfT}\ttg (1-r^2(x^{\sfT}\ttg)^2)^{\frac{1}{2}} ((\xi_2^{\sfT}x_{\bT})\xi_1   + (\xi_1^{\sfT}x_{\bT})\xi_2)
  $$
  which simplifies to the first line of \cref{eq:dualistic}.

We compute $\Gamma^{-1}$ from \cref{eq:divergence}. Let $\ttg' = \grad_{\phi}(x')$, $\tth'= \hess_{\phi}(x')$ and set $F = F(x' | x) = r^2(x^{\sfT}\ttg')^2 - r^2((x')^{\sfT}\ttg')^2+1$ as a function of $x'$ then  
  $$\cD\fD(x, x') = \grad_{x\mapsto \fD(x, x')} =
\ttg - \frac{1}{r}\frac{r\ttg'+ r^2(x^{\sfT}\ttg')F^{-\frac{1}{2}}\ttg'}{rx^{\sfT}\ttg' + F^{\frac{1}{2}}}
=  \ttg - F^{-\frac{1}{2}}\ttg'$$
by factoring out $rF^{-\frac{1}{2}}\ttg'$ in the numerator of the middle ratio. Note 
$$\begin{gathered}
\grad_F(x') = \grad_{x'\mapsto F(x'| x)}(x') = 2r^2((x^{\sfT}\ttg')\tth'x
                - ((x')^{\sfT}\ttg')\ttg' - ((x')^{\sfT}\ttg')(\tth'x')
            ),\\
  \rD'_{\xi}\cD\fD(x, x')  = - F^{-\frac{3}{2}}(F \tth'\xi - \frac{1}{2}\xi^{\sfT}\grad_F(x')\ttg'
  ),
\end{gathered}$$
Note $F(x| x) = 1$, $\grad_{x'\mapsto F(x'| x)}(x) = -2r^2(x^{\sfT}\ttg)\ttg$ and at $x'=x$
$$\begin{gathered}  
  \hess_{x'\mapsto F(x'| x)}(x)\xi_1  = 2r^2((x^{\sfT}\tth\xi_1)\tth x 
  + x^{\sfT}\ttg\phi^{(3)}(x, \xi_1)
  \\
  - (\xi_1^{\sfT}\ttg)\ttg   - (x^{\sfT}\tth\xi_1)\ttg   - (x^{\sfT}\ttg)\tth\xi_1\\ 
  - (\xi_1^{\sfT}\ttg)\tth x
  - (x^{\sfT}\tth\xi_1)\tth x
  - (x^{\sfT}\ttg)\phi^{(3)}(x, \xi_1)
  - (x^{\sfT}\ttg)\tth\xi_1   
  )\\
= - 2r^2((\xi_1^{\sfT}\ttg)\ttg   + (x^{\sfT}\tth\xi_1)\ttg + 2(x^{\sfT}\ttg)\tth\xi_1
+ (\xi_1^{\sfT}\ttg)\tth x).
\end{gathered}$$
Then $\rD'_{\xi}\cD\fD(x, x)  = - \tth\xi  -r^2(\xi^{\sfT}\ttg x^{\sfT}\ttg)  \ttg=-\hess^{\bcc}_{\phi}\xi$ as already known, and
$$\begin{gathered}\rD'_{\xi_1}  \rD'_{\xi_2}\cD\fD(x, x)  = \frac{3}{2}(-2r^2\xi_1^{\sfT}\ttg x^{\sfT}\ttg) ( \tth\xi_2 + r^2(x^{\sfT}\ttg\xi_2^{\sfT}\ttg ) \ttg) \\
  - (\xi_1^{\sfT}\grad_F(x) \tth\xi_2 + \phi^{(3)}(\xi_1, \xi_2)
  - \frac{1}{2}(\xi_2^{\sfT}\hess_F(x)\xi_1)\ttg
  - \frac{1}{2}(\xi_2^{\sfT}\grad_F(x))\tth\xi_1)\\
=   - (3r^2\xi_1^{\sfT}\ttg x^{\sfT}\ttg) \tth\xi_2 - 3r^4((x^{\sfT}\ttg)^2\xi_1^{\sfT}\ttg  \xi_2^{\sfT}\ttg ) \ttg 
+ 2r^2(x^{\sfT}\ttg)(\xi_1^{\sfT}\ttg) \tth\xi_2 - \phi^{(3)}(\xi_1, \xi_2)\\
- r^2(\xi_1^{\sfT}\ttg\xi_2^{\sfT}\ttg + (x^{\sfT}\tth\xi_1)\xi_2^{\sfT}\ttg + 2x^{\sfT}\ttg(\xi_2^{\sfT}\tth\xi_1)
+ \xi_1^{\sfT}\ttg(\xi_2^{\sfT}\tth x)  )\ttg
   -r^2x^{\sfT}\ttg\xi_2^{\sfT}\ttg\tth\xi_1.
\end{gathered}$$
From here, we get the equation for $\Gamma^{-1} = (\rD'_{\xi}\cD\fD)^{-1}(\rD'_{\xi_1}  \rD'_{\xi_2}\cD\fD)$ at $x$. The Amari-Chentsov tensor is a straightforward calculation.
\end{proof}
The curvature of the primal connection can be computed from \cref{eq:curveGeneral}. We carried out the computation but do not have an interesting result to report. See the workbook \href{https://github.com/dnguyend/regularMTW/blob/main/colab/HyperbolicDualGraphAndDivergence.ipynb}{colab{/}HyperbolicDualGraphAndDivergence{.}ipynb} in \cite{NguyenMTWGitHub} for numerical verifications of an implementation of both the curvature and the Amari-Chentsov tensor.
\section{Applications}
\subsection{Hyperbolic Mirror sampling and the multivariable $t$-distribution}
The results in this section are implemented in \href{https://github.com/dnguyend/regularMTW/blob/main/colab/Multivariate_t_Python.ipynb}{colab{/}Multivariate\_t\_Python{.}ipynb} in \cite{NguyenMTWGitHub}. As mentioned in the introduction, we use the optimal map and its inverse to transfer a sampling problem to its dual domain. Let us summarize the main ideas. We want to compute the integral $\int_{\Omega}e^{-V(x)} dvol_{\Omega}$, where $dvol_{\Omega}$ is the Lebesgues volume form on a domain $\Omega$ in $\R^n$. If there is a bijective smooth map $F$ from a domain $\Omega'$ with Lebesgues volume form $dvol_{\Omega'}$ onto $\Omega$, then
$$e^{-V(x)} dvol_{\Omega} =  e^{-V\circ F(y)} |\det dF(y)| dvol_{\Omega'} = e^{-V\circ F(y) + \log |\det dF(y)|} dvol_{\Omega'} =: e^{-W(y)}dvol_{\Omega'}.
$$
Here, $dF$ is the differential of $F$, represented by the Jacobian matrix, and the above is the change of variable formula for a multivariable integral. In high dimensions, we often need to use Monte Carlo simulation, and a change of variable may lead to a more effective sampling.

Mirror Monte Carlo sampling, as mentioned, (see \cite{NEURIPS2018_Mirror,pmlrv125_zhang20a,AhnChewi} among others) considers the case where $F$ is given by an optimal map $F= \bT^{\barbcc}_{\phi^{\barbcc}}$ of a transport problem with the classical cost $-x^{\sfT}\barx$ and $\phi$ is a convex function on $\Omega$. In this case, $\bT^{\barbcc}_{\phi^{\barbcc}}=\grad_{\phi^{\barbcc}}$ and $dF(y) = \hess_{\phi^{\barbcc}}(y) = \hess_{\phi}(x)^{-1}$ where $x_{\bT} = y$. We want to consider the more general case of $F = \bT^{\barbcc}_{\phi^{\barbcc}}$ for a different cost $\bcc$ with a $\bcc$-convex potential $\phi$ of Legendre-type, where \cref{eq:cdT} gives a formula for $d\bT_{\phi}$, hence for $d\bT^{\barbcc}_{\phi^{\barbcc}}$. In particular,
\begin{equation}
  \det(d\bT(x)) = \det(- (\brD_{}\cD\bcc(x, x_{\bT}))^{-1})\det(\hess^{\bcc}_{\phi}(x)).\label{eq:detdT}
\end{equation}
For a cost of type $\bu(x^{\sfT}\barx)$, using \cref{eq:KMCs} and the Weinstein-Aronszajn identity
\begin{equation}
  \begin{gathered}
\det(- (\brD_{}\cD\bcc(x, x_{\bT}))^{-1}) = \det (- \frac{1}{s_1}I_n+   \frac{s_2}{s_1^3} x_{\bT} x^{\sfT}  )^{-1}= \frac{(-s_1)^{n+2}}{s_1^2-s_2s_0}.
    \end{gathered}
\end{equation}
Here, $s_0 = x^{\sfT}x_{\bT}$. Let $\ttg = \grad_{\phi}(x)$. For the hyperbolic cost $\bs = -\frac{1}{r}\sinh(ru)$, note $s_1^2-s_2s_0=1$. From the proof of \cref{eq:dualistichyper}, $s_1 = \frac{-1}{(1-r^2(x^{\sfT}\ttg)^2)^{\frac{1}{2}}}$, hence
\begin{equation}\det(- (\brD_{}\cD\bcc)(x, x_{\bT}))^{-1}) = (1-r^2(x^{\sfT}\ttg)^2)^{-\frac{n+2}{2}}.
\end{equation}
With an appropriate choice of potential, we can transport a sampling problem to a more effective one. We consider an application to the problem of evaluating an expectation or a probability of the form $Prob(l \leq L X\leq u)$ where $X$ is a random variable with a \emph{fat-tailed} distribution with values in $\R^n$, and $L$ is a matrix. Applications include copula calculation in actuarial science and credit derivatives. In high dimensions, without an analytic formula, Monte Carlo simulation is probably the only way to evaluate the probability, and the \emph{fat-tailed} assumption poses a significant challenge. A naive simulation based on a uniform distribution on a hypercube of the form $[-R\times R]^n$ requires a large number of simulation points. In $n=10$ dimensions in the cited workbook, we show numerically that a simulation of $N=10^5$ only captures on average $65\%$ of probability with $290\%$ of standard deviation and is not useful. To scale up to higher dimensions, we need to use special features of the distribution to be effective.

There is an extensive literature on sampling of the multivariate $t$-distributions $\mathbf{Y}\sim t_{\nu} = t_{\nu}(0, I_n)$ of $n$ dimensions, our main application. The parameter $\nu$ is the degree of freedom. The general case will be converted to this case by an affine transformation. Existing effective algorithms often use the representation $\mathbf{Y} = (\frac{\mathbf{U}}{\nu})^{-\frac{1}{2}}\mathbf{Z}$ with $\mathbf{U}\sim \chi^2$ and $\mathbf{Z}\sim N(0, I_n)$, with $\mathbf{U}$ and $\mathbf{Z}$ being independent. The main reference is \cite[Chapter 4]{GenzBretz}, which uses this chi-square formulation of the $t$-distribution, including an exact sampling. The state-of-the-art method to calculate probabilities of the form $Prob(l \leq L X\leq u)$ in \cite{GenzBretz} uses an additional separation of variables (SOV) method, which gives the standard implementation in R with a quasi-Monte Carlo method. The work \cite{BotevEcuyer} improves on this, allowing sampling of very rare events.

Alternatively, we observe that the reason the naive Monte Carlo simulation in $\R^n$ fails is that the data spreads out in space. If we could transport the problem to a confined space, then we could sample more effectively. The absolutely-homogeneous convex potential transport between a bounded sphere to $\R^n$ is thus a good candidate. Informally, we expect the mirrored sampling to be effective if the transported pdf packs into a ball-shaped region of a simple form. 
\begin{figure}[ht!]
  \centering
  \includegraphics[width=\textwidth]{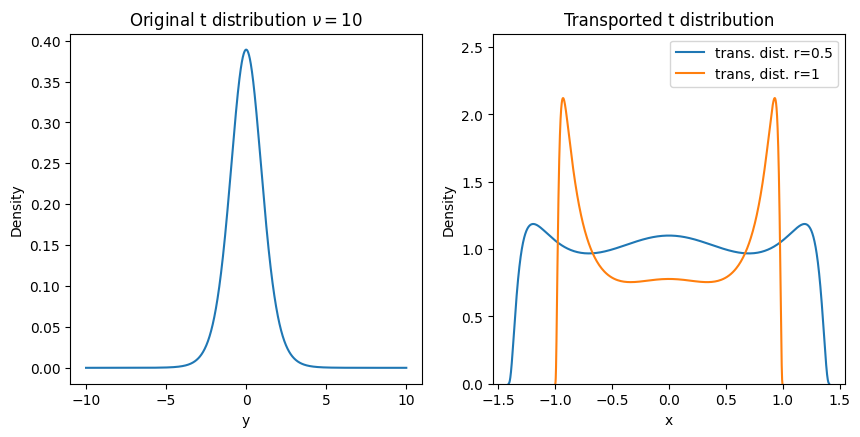}  
  \includegraphics[width=\textwidth]{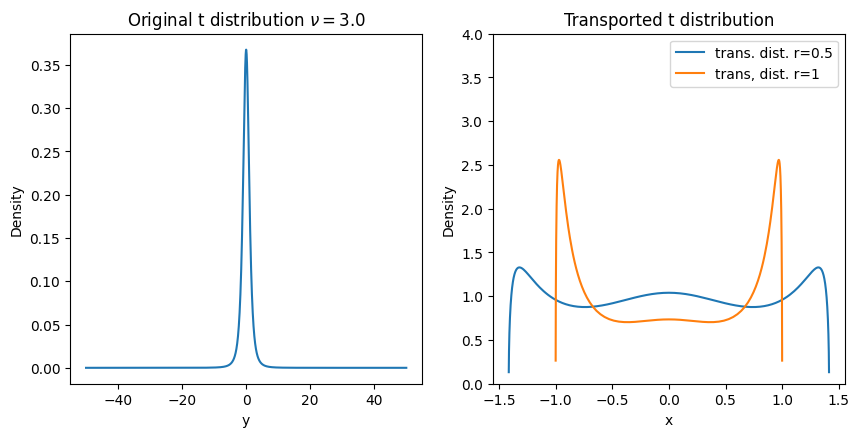}
\caption{Transporting $t$-distributions using hyperbolic costs and quadratic potentials $\phi(x)=x^2, |x| < r^{-\frac{1}{2}}$. The graphs of $e^{-W(y)}$ are on the left and of $e^{-V_{adj}(x)}$ are on the right. For the smaller $r=0.5$, the transported distribution is close to a uniform density, while the transported $r=1.$ has two peaks near the boundary points. For $\nu=10$, the densities at the two ends of $[-10,10]$ map close to the zero in the transported distribution on $[-r^{-\frac{1}{2}}, r^{-\frac{1}{2}}]$, while for the fatter-tailed $\nu=3$, even the points $\pm 50$ map further from zero in the transported distribution.}
    \label{fig:MVTtransport}
\end{figure}

Thus, consider the problem of sampling $e^{-W(y)}dvol_{\Omega'}$ for $\Omega'=\R^n$, where $W$ is the negative log density of the $t_{\nu}(0, I_n)$ distribution with $\nu$ degrees of freedom
\begin{equation}W(y) = \log\frac{\Gamma(\frac{\nu}{2})(\pi\nu)^{ \frac{n}{2}}}{\Gamma(\frac{\nu+n}{2})} + \frac{\nu+n}{2}\log(1 + \frac{1}{\nu}y^{\sfT}y).
\end{equation}

With a choice of $r$ and a positive-definite matrix $C\in \R^{n\times n}$, for the domain $\Omega$, we take the ellipsoid ball $B=\{x\in \R^n|\;(x^{\sfT}Cx)^2\leq \frac{1}{r}\}$ where $\phi(x) = \frac{1}{2}x^{\sfT}Cx$ in \cref{rem:quadraticPotential} is $\bcc$-convex. Then $\bT_{\phi}(x) = \frac{1}{(1-r^2(x^{\sfT}Cx)^2)^{\frac{1}{2}}}Cx$ maps $B$ to $\R^n$ bijectively. The $\bcc$-hessian is $C + r^2(x^{\sfT}Cx)Cxx^{\sfT}C$ with determinant $\det(C)(1+r^2 (x^{\sfT}Cx)^2)$. Using the change of variable formula, $y=\bT_{\phi}(x)$, by \cref{eq:detdT} and  the $\bcc$-Crouzeix's relation,
$$\begin{gathered}  
  V(x) =   W(\bT_{\phi}(x)) - \log\det d\bT_{\phi}(x) \\
  =\log\frac{\Gamma(\frac{\nu}{2})(\pi\nu)^{\frac{n}{2}}}{\Gamma(\frac{\nu+n}{2})}
            + \frac{\nu+n}{2}\log(1+\frac{\frac{1}{\nu}x^{\sfT}C^2x}{(1-r^2(x^{\sfT}Cx)^2)})\\
            + \frac{2+n}{2}\log(1-r^2(x^{\sfT}Cx)^2) \
           - \log(1+r^2(x^{\sfT}Cx)^2) - \log(\det(C),
\end{gathered}            $$
$$\begin{gathered}              
V(x)=  \log\frac{\Gamma(\frac{\nu}{2})(\pi\nu)^{\frac{n}{2}}}{\Gamma(\frac{\nu+n}{2})}
            + \frac{\nu+n}{2}\log(1- r^2(x^{\sfT}Cx)^2 + \frac{1}{\nu}x^{\sfT}C^2x) \\
            + (1-\frac{\nu}{2})\log(1-r^2(x^{\sfT}Cx)^2) \
            - \log(1+r^2(x^{\sfT}Cx)^2) - \log(\det(C)).
\end{gathered}            
$$
For $n=1$, the graphs of $W$ and $V_{adj}$ are plotted in \cref{fig:MVTtransport}, with $V_{adj}= V - \log(Vol(B))$. We note that the transported density $V_{adj}$ (which is $V$ shifted by a constant) is almost uniform for $r=0.5$, thus a uniform sampling is expected to be effective. Uniform sampling in the unit ball is well-known, and sampling on the ellipsoid is done by multiplying by $r^{-\frac{1}{2}}C^{-\frac{1}{2}}$. To evaluate an expectation $\E_{t_{\nu}}f(Y)$ for $Y\sim t_{\nu}$, we multiply the Monte Carlo average of $(f\circ \bT) e^{-V}$ from a uniform sampling in the ellipsoid ball $B$ with the volume of $B$, which is $\det(C)^{-\frac{1}{2}}r^{-\frac{n}{2}}\frac{\pi^{\frac{n}{2}}}{\Gamma(\frac{n}{2}+1)}$.
 Thus,
\begin{equation}\E_{t_{\nu}}f(Y) \approx \frac{1}{N} \sum_{i=1}^Nf(\bT_{\phi}(x_i))e^{-V_{adj}(x_i)}.
\end{equation}  
To calculate for $Z\sim t_{\nu}(\mu, \Sigma)$ with a location vector $\mu\in \R^n$ and scale matrix $\Sigma\in \R^{n\times n}$, we perform an affine change of variable $Y = L^{-1}(Z - \mu)$ where $\Sigma = LL^{\sfT}$ is a Cholesky decomposition. In applications, for example, in pricing risk tranches, we need to compute the probability where $Z$ belongs to a box $\mathtt{l}\leq LY \leq \mathtt{u}$ as mentioned. In this case, $f$ is an indicator function of that box.

In our first experiment with results in \cref{tbl:tprob1}, we take $\Sigma = LL^{\sfT}=2(I_n + \tone\tone^{\sfT})^{-1}$, and the box $[\mathtt{l}, \mathtt{u}] = [-1,\infty]^n$ as in \cite[Table 1]{BotevEcuyer}. Numerically, the results match well with the cited table, we note for large $n$ the number is closer to the  improved method of \cite{BotevEcuyer} than the method of the R package $\mathtt{mvtnorm}$ in \cite{GenzBretz}. Note that the sampling method $\mathtt{rvs}$ in $\mathtt{multivariate\_t}$ in scipy.stats returns $0$ for $n=100,150$, thus, we handle this situation better. For the second experiment $[\mathtt{l}, \mathtt{u}] = [0,\infty]^n$, with $C=\Sigma^{-1}$, the probabilities are very small, to the order of $10^{-17}$ for $n=20$, which is beyond the typical Python floating point accuracy of $10^{-15}$ in common matrix operations. We suspect this is the reason our method fails beyond $n=10$ for this example, while $\mathtt{mvtnorm}$ with the separation of variable method loop through scalar variables, thus, can avoid this problem.

Our next example in \cref{tbl:tprob2} also produces a match with \cite[Table 3]{BotevEcuyer}. Thus, mirror sampling is an effective method for the multivariate $t$-distribution, and we expect it to be effective for other fat-tailed distributions in high dimensions.

\begin{table}
  \resizebox{\textwidth}{!}{
    \begin{tabular}{lrrrrrrrr}
\toprule      
n & 5& 10& 20& 30& 40& 50& 100& 150\\
\midrule
Probability & 0.198& 0.0325& 0.00164& 0.000150& 2.08e-05& 3.72e-06& 6.73e-09& 8.95e-11\\
Std & 0.000960& 0.000441& 5.65e-05& 9.85e-06& 2.39e-06& 7.01e-07& 4.07e-09& 8.99e-11\\
\bottomrule
\end{tabular}
  }
\caption{Simulation for $\Sigma = LL^{\sfT}=2(I_n + \tone\tone^{\sfT})^{-1}$. Estimating $Prob(LY\in [-1,\infty]^n)$ for $v=10$ with $N=10^5$ samples repeated $100$ times. Compare with \cite[Table 1]{BotevEcuyer}.}
\label{tbl:tprob1}
\end{table}

\begin{table}
    \begin{tabular}{lrrrrrrrr}
      \toprule
      n & 5& 30& 50& 100& 150\\
      \midrule
Probability & 0.0993& 0.0599& 0.0520& 0.0423& 0.0374\\
Std & 0.00108& 0.00109& 0.000861& 0.000924& 0.000842\\
\bottomrule
\end{tabular}
\caption{Simulation for $\Sigma = LL^{\sfT}=(1-\rho)I_n + \rho\tone\tone^{\sfT}$ for $\rho=.95$. Estimating $Prob(LY\in [1,3]^n)$ for $v=10$ with $N=10^5$ samples repeated $100$ times. Compare with \cite[Table 3]{BotevEcuyer}.}
\label{tbl:tprob2}
\end{table}
\subsection{A local divergence and latent spaces}\label{sec:latent}
Divergences on probability simplices, in particular, the Kullback-Leibler (KL) and $\alpha$-divergences, are the workhorses of machine learning. Using the absolutely-homogeneous convex potential for the $\sinh$-type cost, we obtain a local version of the $\alpha$-divergence using $\phi(x) = \sum_{i=1}^n|x|_i^{\frac{1}{\alpha}}$ in \cref{ex:homogenousgeq1}. \emph{To be consistent with the $\alpha$-divergence, we will substitute $\frac{1}{\alpha}$  in place of $\alpha$ in \cref{ex:homogenousgeq1}}. 
\begin{proposition}\label{prop:divergence} Assume $0 < \alpha < 1$ and let $\mu, \mu'$ be two probability measures on a finite set $\Omega$. For all subsets $\Omega_1\subset \Omega$, we have
\begin{equation}\label{eq:AlphaLoc}\alpha(\mu(\Omega_1) - \mu'(\Omega_1)) -\log\frac{\int_{\Omega_1}(\frac{d\mu}{d\mu'})^{\alpha}d\mu'+\left(
\left(\int_{\Omega_1}(\frac{d\mu}{d\mu'})^{\alpha}d\mu'\right)^2
    - \mu'(\Omega_1)^2+1\right)^{\frac{1}{2}}}{1+\mu'(\Omega_1)}\geq 0
\end{equation}
 with equality if and only if $\mu = \mu'$ on $\Omega_1$.
\end{proposition}
In the above $\frac{d\mu}{d\mu'}$ is the Radon-Nikodym derivative, which is the ratio of the probability densities in this finite case. We will see that the limit when $\alpha$ goes to infinity is a local version of the Kullback-Leibler divergence. We also expect to have extensions to more general probability spaces. The new divergences could be useful in learning how two probability measures differ on different subsets.

Assume $\tphi$ is homogeneous of order $\frac{1}{\alpha} > 1, 0 < \alpha < 1$. Then $\phi(x) = \tphi^{\bar{\bcc}\bcc}(x)$ is $\bcc$-convex and equals to $\tphi(x)$  inside $\tphi(x) \leq \frac{\alpha }{r}$. We call this closed subset $\bar{B}$, its interior $B$ and its boundary $\partial B$ is defined by $\phi(x)=\frac{\alpha}{r}$.

We need a condition on when $\fD$ in \cref{eq:divergence} is zero on $\bar{B}$. If both $x, x'$ are inside $B$ then $\fD(x, x')=0$ only if $x=x'$. Since  $\tphi(x) = \tphi^{\bar{\bcc}\bcc}(x)$ also on the boundary $\partial B$, we still have $\tphi(x)  + \bu(x^{\sfT}\bar{x}) + \tphi^{\bar{\bcc}}(\barx) \geq 0$ for $x\in \partial B$ and all $\barx\in \R^n$ but we cannot have equality for finite $\barx$. However, as $x'$ approaches the boundary $\partial B$, $\barx = \bT(x')$ approaches infinity and $\fD(x, x')$ could be close to $0$. This shows for $x\in \partial B$ and $x'\in \bar{B}$ then $\fD(x, x')$ could only be zero if $x'\in \partial B$. But if $x'\in \partial B$ then $\fD(x, x') = -\frac{1}{r}\log\max\{rx^{\sfT}\grad_{\phi}(x'), 0\}$. Again, using the strictly convex condition, $\phi(x) - \phi(x') -\grad_{\tphi}(x')^{\sfT}(x-x')\geq 0$ for $x, x'\in \partial B$, we have $\fD(x, x')=0$ on $\bar{B}$ if and only if $x=x'$.

\begin{proof}(of \cref{prop:divergence}) For $0 < \alpha < 1$, consider $\tphi(x) =  \sum_{i=1}^n |x_i|^{\frac{1}{\alpha}}$, which is $\bcc$-convex inside the closed ball $\bar{B}$ defined by $\sum_{i=1}^n|x|_{i}^{\frac{1}{\alpha}} \leq \frac{\alpha}{r}$. Consider the probability simplex $\sum_{p\in \Omega} p = 1$ with $p \geq 0$ where $\Omega$ is a finite set containing $\Omega_1$ with $|\Omega_1| = n$. For two probability measures $\mu$ and $\mu'$ on $\Omega$ with corresponding distributions $\ptt=(p)_{p\in\Omega}$ and $\ptt'=(p')_{p'\in\Omega}$, set $x = (\frac{\alpha}{ r}p)^{\alpha}_{p\in\Omega_1}, x'=(\frac{\alpha}{ r}p')^{\alpha}_{p'\in\Omega_1}$ as elements of $\R^n$, then $x$ and $x'$ are in $\bar{B}$. We have $\phi(x) = \frac{\alpha }{r}\mu(\Omega_1), \phi(x') =\frac{\alpha}{r}\mu'(\Omega_1)$ and
  $$x^{\sfT}\grad_{\phi}(x') = \frac{1}{\alpha}\sum_i (\frac{\alpha}{r}p_i)^{\alpha} ((\frac{\alpha}{r}p'_i)^{\alpha} )^{\frac{1}{\alpha}-1} =
  \frac{1}{r}\int_{\Omega_1} (\frac{d\mu}{d\mu'})^{\alpha}d\mu'.$$
A substitution to \cref{eq:divergence} gives us \cref{eq:AlphaLoc}. The argument preceeding the proof shows we have equality if and only if $\mu = \mu'$ on $\Omega_1$.
\end{proof}
A corollary is if $\Omega = \Omega_1\cup\Omega_2\cdots \cup\Omega_k$ is a partition of $\Omega$. Then for $0 < \alpha < 1$
  \begin{equation} \prod_{i=1}^k\frac{\int_{\Omega_i}(\frac{d\mu}{d\mu'})^{\alpha}d\mu'+\left(
\left(\int_{\Omega_i}(\frac{d\mu}{d\mu'})^{\alpha}d\mu'\right)^2
    - \mu'(\Omega_i)^2+1\right)^{\frac{1}{2}}}{1+\mu'(\Omega_i)} \leq 1
  \end{equation}
with equality only if $\mu=\mu'$. This follows from \cref{prop:divergence} by summing \cref{eq:AlphaLoc} over all sets in the partition, noting $\mu(\Omega) = \mu'(\Omega) = 1$.

When $\Omega_1 = \Omega$, \cref{eq:AlphaLoc} becomes $-\log \int_{\Omega_1}(\frac{d\mu}{d\mu'})^{\alpha}d\mu' \geq 0$, equivalent to the $\alpha$-divergence \cite{WongInfo} $\frac{1}{\alpha-1}\log \int_{\Omega_1}(\frac{d\mu}{d\mu'})^{\alpha}d\mu'\geq0$. We can also divide \cref{eq:AlphaLoc} by $\alpha$ then take the limit when $\alpha$ goes to $0$ for a local form of the KL divergence.

Results in this section are in \href{https://github.com/dnguyend/regularMTW/blob/main/colab/HyperbolicDualGraphAndDivergence.ipynb}{colab{/}HyperbolicDualGraphAndDivergence{.}ipynb} in \cite{NguyenMTWGitHub}. Assuming we sample from a space $\Omega$, where samples from a subset $\Omega_0$ are unreliable, we know when a sample is in $\Omega_0$, but we do not know its value with certainty. However, if a sample is in the complement $\Omega_1$ then we can trust the sample value. Alternatively, we can think of a situation where we observe a system with a latent state $S\in\Omega$ which is only visible when $S\in\Omega_1$, but not visible when $S\in\Omega_0=\Omega-\Omega_1$. If we define a local divergence as a kind of distance between probability measures on $\Omega$, that is zero if they are identical on $\Omega_1$, then local divergence is a notion of distance on probability measures on a latent space, given observable data. This measure could be useful for state-space models in control theory and machine learning.

A simple construction is to consider $\Omega_* = \{S_*\}\cup \Omega_1$ for a dummy element $S_*$, and for any measure $\mu$ on $\Omega$, consider the measure $\mu_*$ on $\Omega^*$, with $\mu_*(S_*) = \mu(\Omega_0)$, while $\mu_*(A) = \mu(A)$ for any subset $A\in\Omega_1$. It is clear this defines a unique measure $\mu_*$ on $\Omega_*$, thus, given a divergence, for example, an $\alpha$-divergence $D_{\alpha}$, with $0 < \alpha < 1$, we can define a local divergence (which we call the blank local divergence)
\begin{equation}D_{\alpha,\flat}(\mu || \mu') := \alpha(1-\alpha)D_{\alpha}(\mu_* ||\mu'_*)
  =-\log \left((1-\mu(\Omega_1))^{\alpha}(1-\mu'(\Omega_1))^{1-\alpha} + \int_{\Omega_1} (\frac{d\mu}{d\mu'})^{\alpha}d\mu' \right).
\end{equation}  
For another example, take a convex function $f$, then $\frac{1}{\mu'(\Omega_1)}d\mu'$ is a probability measure on $\Omega_1$ if $\mu(\Omega_1)\neq 0$. By Jensen's inequality, $\int_{\Omega_1} f(\frac{d\mu}{d\mu'})\frac{d\mu'}{\mu'(\Omega_1)}\geq  f(\int_{\Omega_1} \frac{d\mu}{d\mu'}\frac{d\mu'}{\mu'(\Omega_1)})$, or
\begin{equation}\int_{\Omega_1} f(\frac{d\mu}{d\mu'})d\mu'\geq  \mu'(\Omega_1)f( \frac{\mu(\Omega_1)}{\mu'(\Omega_1)}).\end{equation}
For $f(x) = -x^{\alpha}$ for $0 <\alpha < 1$, this give us $\int_{\Omega_1} (\frac{d\mu}{d\mu'})^{\alpha}d\mu' \geq \mu(\Omega_1)^{\alpha}\mu'(\Omega_1)^{1-\alpha}$ and the local $(\alpha,f)$ divergence
\begin{equation}D_{\alpha, f}(\mu ||\mu') := \alpha\log\mu(\Omega_1) + (1-\alpha)\log\mu'(\Omega_1) - \log\int_{\Omega_1} (\frac{d\mu}{d\mu'})^{\alpha}d\mu' \geq 0.
\end{equation}
Finally \cref{eq:AlphaLoc} is another local divergence. All three local divergences reduce to $- \log\int_{\Omega} (\frac{d\mu}{d\mu'})^{\alpha}d\mu'$ when $\Omega_1=\Omega$. Numerically, we find the hyperbolic-$\alpha$ local divergence of \cref{eq:AlphaLoc} is the tightest. In our random test, it is always not larger than $D_{\alpha,\flat}$, which we conjecture to always hold. If we draw normally distributed weights for $\Omega$, for only $18$ percent of the time $D_{\alpha, f}$ is smaller than the divergence in \cref{eq:AlphaLoc}.

\section{Conclusion} In this article, we found a family of zero MTW tensors on an Euclidean space and new families of positive MTW tensors on the hyperbolic space and the sphere. The result links the Euclidean and log-type cost functions as part of a larger family. We also analyze the geometry of the $\bcc$-divergence using $\bcc$-convex functions with a focus on the $\sinh$-type hyperbolic family. The linkage between optimal transport and information geometry gives a new family of generalized hyperbolic analogs of the Bregman divergence, which we expect to be useful in statistics, learning, and optimization. In particular, we show that mirror sampling with non-classical costs is promising, and we expect the $\bcc$-mirror Langevin sampling method to complement the existing mirror Langevin sampling framework, while the hyperboloid costs will be useful in matching problems in hyperbolic embedding. We also obtain a new divergence-type inequality on probability spaces. In \cite{NguyenGeoglobal}, we show a family of $\log$-type costs constructed from the polar decomposition also satisfying A3w on the space of fixed-rank matrices. This suggests pairings between points in source and target domains could generate many families of interesting regular costs.

\section{Acknowledgements}
The author would like to thank the reviewer(s) and the handling editor for their careful reading and constructive suggestions that helped improve the article's quality significantly. In particular, he is grateful for the suggestions to include \cref{tbl:notations} and to express the cost function explicitly in \cref{theo:Hn1}. He wishes to thank his family for their loving support in this project.

\begin{appendices}

\section{Proof of \cref{prop:power}}\label{appx:proofpower}
\begin{proof}For $\rS^n$, we show the coefficients are given by
  \begin{equation}
    \begin{gathered}
    R_1 = \frac{(s + 1)^{\frac{1}{\alpha}}f_0}{
      \alpha^4(s + 1)^2(s + \alpha - 1)} \\
    \text{ with } f_0 = 
      2(1 - \alpha)s^3
        + (\alpha^3 - 4\alpha^2 + 9\alpha - 6)s^2
        + (\alpha^3 + 6\alpha^2 - 12\alpha + 6)s\\
        + (\alpha^4 - \alpha^3 - 2\alpha^2 + 5\alpha - 2),\\
    R_{23} = \frac{(s + 1)^{\frac{1}{\alpha}}((2\alpha - 1)s + \alpha^2 - \alpha + 1)}{\alpha^2(s + 1)^2(s + \alpha - 1)}, \\    
    R_4 = \frac{(s+\alpha)(s + 1)^{\frac{1}{\alpha}}}{\alpha(s + 1)^2(s+\alpha-1)}.
    \end{gathered}
\end{equation}    
  We have $\bs_i(u) = (-1)^i(-u)^{\alpha-i}\prod_{j=0}^{i-1}(\alpha-j)=\prod_{j=0}^{i-1}(\alpha-j)(s+1)u^{-i}$ for $i=1\cdots 4$ since $(-u)^{\alpha-i} = (s+1)(-u)^{-i}$ and $u\leq 0$. Thus,
    $$\begin{gathered}s_1^2-s_2s = \alpha^2(s+1)^2u^{-2} -\alpha(\alpha-1)(s+1)su^{-2} = \alpha u^{-2}(s+1)(s+\alpha),\\
      -ss_1^2 - s_2(1-s^2) = -s_2 - s(s_1^2-ss_2) = - \alpha(\alpha-1)(s+1)u^{-2} - \alpha u^{-2}(s+1)(s+\alpha)s \\
         = - \alpha(s+1)^2u^{-2}(s + \alpha-1),\\
s_1s_3 - s_2^2 =   \alpha (s+1)u^{-1}\alpha(\alpha-1)(\alpha-2)(s+1)u^{-3} - \alpha^2(\alpha-1)^2(s+1)^2u^{-4}\\
    = - \alpha^2(\alpha-1)(s+1)^2u^{-4},\\
(s_1^2 - ss_2)s_4 + ss_3^2 - 2s_1s_2s_3 + s_2^3 \\
      =\alpha u^{-2}(s+1)(s+\alpha)\alpha(\alpha-1)(\alpha-2)(\alpha-3)(s+1)u^{-4} \
        + s\alpha^2(\alpha-1)^2(\alpha-2)^2(s+1)^2u^{-6} \\
        - 2\alpha^3(\alpha-1)^2(\alpha-2)(s+1)^3u^{-6}
        + \alpha^3(\alpha-1)^3(s+1)^3u^{-6}\\
     = \alpha^2(\alpha-1)u^{-6}(s+1)^2(2s - \alpha^2 + 3\alpha)
    \end{gathered}
    $$
where between the second and last equal signs, after factorizing $\alpha^2(\alpha-1)u^{-6}(s+1)^2$, we get an expression of the form $a_1s+a_0$, and the coefficients $a_1$ and $a_0$ are
$$\begin{gathered}a_1  = (\alpha-2)(\alpha-3) \
            + (\alpha-1)(\alpha-2)^2 
            - 2\alpha(\alpha-1)(\alpha-2)
            + \alpha(\alpha-1)^2 =2,\\
        a_0 =\alpha (\alpha-2)(\alpha-3) -2\alpha(\alpha-1)(\alpha-2) +\alpha(\alpha-1)^2= 
        \alpha(-\alpha+3).\end{gathered}$$
The following gives us the numerator for $R_{23}$, then the numerator for $R_1$
$$\begin{gathered}(s_3s_1-s_2^2)(1-s^2) + s_1^2(s_1^2-s_2s) \\
  = - \alpha^2(\alpha-1)(s+1)^2u^{-4}(1-s^2) + \alpha^2(s+1)^2u^{-2}(\alpha u^{-2}(s+1)(s+\alpha))\\
  = \alpha^2(s+1)^3u^{-4}((2\alpha-1)s + \alpha^2 - \alpha + 1),
\end{gathered}$$
$$\begin{gathered}
  \text{Num. of } R_1 = \alpha^2(\alpha-1)u^{-6}(s+1)^2(2s - \alpha^2 + 3\alpha)(- \alpha(s+1)^2u^{-2}(s+ \alpha-1))(1-s^2)^2 \\
  + \alpha^4(s+1)^6u^{-8}((2\alpha-1)s + \alpha^2 - \alpha + 1)^2\\
=  \alpha^3(s+1)^6u^{-8}(
        -(\alpha-1)(2s - \alpha^2 + 3\alpha)(s+\alpha-1)(s-1)^2 \
        + \alpha((2\alpha-1)s + \alpha^2 - \alpha + 1)^2)  \\
        = \alpha^3(s+1)^6u^{-8}(s+\alpha) \\
        \times (2(1 - \alpha)s^3 + (\alpha^3 - 4\alpha^2 + 9\alpha - 6)s^2         + (\alpha^3 + 6\alpha^2 - 12\alpha + 6)s + \alpha^4 - \alpha^3 - 2\alpha^2 + 5\alpha - 2).
\end{gathered}$$
With $D = -\alpha^7(s+1)^8u^{-9}(s+\alpha)(s+\alpha-1)$, the formulas for $R_1, R_{23}, R_4$ follow.

It is clear that $D$, hence $R_4$ are positive if $s > -1$, $\alpha \geq 2$. For $R_1$, to show $f_0(s) > 0$, set $\alpha = b+2$ with $b\geq 0$, and rearange $f_0(s)$ as a polynomial in b:
$$ b^4 + (s^2 + s + 7)b^3 + (2s^2 + 12s + 16)b^2 + (s + 1)(-2s^2 + 7s + 17)b+ 2(-s + 4)(s + 1)^2.$$
All coefficients are positive for $s\in (-1, 1]$, thus, $f_0(s)$ is positive. For $R_{23}$, note
  $$f_1(s) := (2\alpha - 1)s + \alpha^2 - \alpha + 1 > f_1(-1) = (\alpha-1)(\alpha-2) \geq 0.$$
  For $\rH^n$, let $\alpha = \frac{1}{\beta}$, then $\bs(u) = - (-u)^{\alpha}$, we have $s_i = \prod_{j=0}^{i-1}(\alpha-j)su^{-i}$, thus,
  $$\begin{gathered}s_1^2-s_2s= \alpha^2s^2u^{-2} - \alpha(\alpha-1)s^2u^{-2}=\alpha s^2u^{-2} > 0,\\
    -s_2 - s(s_1^2-ss_2) =-\alpha(\alpha-1)su^{-2} -s\alpha s^2u^{-2}  = - \alpha s(s^2-1+\alpha)u^{-2}> 0,\\
    s_1s_3 - s_2^2 = \alpha^2(\alpha-1)(\alpha-2)s^2u^{-4} - \alpha^2(\alpha-1)^2s^2u^{-4}=\alpha^2(1-\alpha)s^2u^{-4}\geq0,
  \end{gathered}$$
$$\begin{gathered}  
    (s_1^2 - ss_2)s_4 + ss_3^2 - 2s_1s_2s_3 + s_2^3= \alpha^2(\alpha-1)(\alpha-2)(\alpha-3)s^3u^{-6} \\ +\alpha^2(\alpha-1)^2(\alpha-2)^2s^3u^{-6}-2\alpha^3(\alpha-1)^2(\alpha-2)s^3u^{-6} +\alpha^3(\alpha-1)^3s^3u^{-6}\\
    =-2\alpha^2(1-\alpha)s^3u^{-6}\geq 0,\\
    (s_3s_1-s_2^2)(1-s^2) + s_1^2(s_1^2-s_2s)= \alpha^2(1-\alpha)s^2u^{-4}(1-s^2)
    +\alpha^2s^2u^{-2}\alpha s^2u^{-2}  \\
    =\alpha^2u^{-4}s^2((2\alpha-1)s^2 +1 -\alpha) > 0.
\end{gathered}    
  $$
  Since $s_1 = \alpha su^{-1}>0$, the above shows $D>0, R_4>0, R_{23}>0, R_1>0$.
\end{proof}

\section{Monotonic ranges of functions in \cref{theo:sol}}\label{rem:Solrange}    
  We are interested in maximum ranges of $\bs$ and $\bu$ where they are monotonic. Since all the solutions of \cref{eq:zeroODE} are entire functions, a range for $u$ is bounded between values of maximum and minimum points $u_c$'s (roots of $\bs_1$), or $\pm\infty$. Hence, the domain $I_u$ of $\bu$ (range of $\bs$) is bounded between $\bs(-\infty), \bs(\infty)$ or
  $\bs(u_c)$, ordering by their relative values. Thus, there are one, two, or infinite such intervals depending on the number of critical points (zero, one, or infinite).

In this paper, we use the terms generalized hyperbolic and inverse generalized hyperbolic for functions of the form $p_0e^{p_1u} + p_2e^{p_3u}$ and their inverses, note the same terms may denote a different family of functions in the literature. The function $\bu$ needs to be solved numerically in the general case, however, when $p_1 = -p_3$ or $p_3=2p_1 > 0$ or $p_1 = 2p_3 < 0$, $\bu$ could be expressed in terms of $\log$ and square root.

Recall the Lambert functions $W_0$ and $W_{-1}$, available in many numerical software packages, are two branches of inverses of $w\mapsto we^w=x$, with $W_{-1}$ corresponding to the branch $ w \leq -1, -e^{-1} \leq x < 0$; $W_0$ corresponds to $ -1 \leq w, -e^{-1} \leq x$. Solutions of $(a_0 + a_1u)e^{a_2u}=s$ are $u = \frac{1}{a_2}W(\frac{a_2\exp{\frac{a_0a_2}{a_1}}}{a_1}s) - \frac{a_0}{a_1}$, where $W$ is either $W_0$ or $W_{-1}$.

Here are the branches of $\bu$ corresponding to the costs in \cref{theo:sol}
\begin{enumerate}
  \item{Generalized hyperbolic: \cref{eq:Sol1}}: $\bs(u) = p_0e^{p_1u} + p_2e^{p_3u}$, $p_3 > p_1$, $p_0p_2\neq 0$.
    \begin{enumerate}
    \item{Antenna-like} $p_0p_1p_2p_3 \geq 0, p_1p_3\geq 0$. There is {\it no critical point}, $\bs$ is monotonic.
      \begin{enumerate}
      \item $p_3 > p_1=0, \bs= p_0 + p_2e^{p_3u}$: $\bs(-\infty) = p_0$, $\bs(\infty) = p_2\infty$.
        \item $p_3 > p_1 > 0$, $\bs(-\infty) = 0, \bs(+\infty) = p_2\infty$.
        \item $0=p_3 > p_1, \bs= p_0e^{p_1u} + p_2$: $\bs(-\infty) = p_0\infty$, $\bs(\infty) = p_2$.
          \item $0 > p_3 > p_1$, $\bs(-\infty) = p_0\infty, \bs(\infty) = 0$.
        \end{enumerate}
    \item{$\sinh$-like:} $p_0p_1p_2p_3 > 0, p_3 > 0 > p_1$, $p_0p_2<0$. {\it No critical point}, $\bs(-\infty) = p_0\infty, \bs(\infty) = -p_0\infty=p_2\infty$. $I_u=\R$.
    \item{$p_0p_1p_2p_3 < 0$:} $u_c = \frac{1}{p_1-p_3}\log\frac{-p_2p_3}{p_0p_1}$, $\bs(u_c) = \frac{p_2(p_1-p_3)}{p_1}(\frac{p_0p_1}{-p_2p_3})^{\frac{p_3}{p3-p1}}$.\label{itm:crGH}
      \begin{enumerate}
      \item{One decaying arm:} $p_3  p_1 > 0, p_0p_2 < 0$, $\bs(-p_1\infty) = 0, \bs(p_1\infty) = p_2\infty$. Two branches with $s$-values separated by $\bs(u_c)$, one to $0$ and the other to $p_2\infty$.
      \item{$\cosh$-like:} $p_3 > 0 > p_1, p_0p_2 > 0$, $\bs(-\infty) = p_0\infty, \bs(+\infty) = p_2\infty$ of the same sign. Two branches with $s$-values enclosed between $\bs(u_c)$ and $p_0\infty$.
        \end{enumerate}
    \end{enumerate}
\item{Affine}: in \cref{eq:Sol3}, $a_2 =0$. {\it No critical point}, the classical affine case (\cite{Brenier}).    
  \item{Lambert:} in \cref{eq:Sol3}, $a_2\neq0$, $u_c = -\frac{a_0a_2+a_1}{a_1a_2}, \bs(u_c)=-\frac{a_1}{a_2}e^{-a_2u_c}$. Values at infinitive are $\bs(a_2\infty) = (\sign{a_1a_2})\infty$, $\bs(-a_2\infty)=0$.
  \item{Exponential-trigonometric: \cref{eq:Sol2}} $u_c$ solves $b_1\sin(b_2u + b_3) + b_2\cos(b_2u + b_3)=0$. There are infinitely many such solutions, the interval $I_u$ could be taken to be a monotonic segment between these solutions, or $I_u$ has length $\frac{\pi}{|b_2|}$ and starts at $\frac{1}{b_2}(\tan^{-1}(-\frac{b_2}{b_1})-b_3 + k\pi)$ for some $k\in \Z$.    
\end{enumerate}
\section{Absolutely homogeneous function of order 1 with hyperbolic cost}\label{sec:convexOrder1}
We consider the same hyperbolic cost as in \cref{ex:homogenousgeq1} with order $\alpha=1$, including the norm function $|x|_{p}$ for $p > 1$. Assume $\tphi$ is convex and absolutely-homogeneous of order $1$. The treatment for order $\alpha > 1$ needs some modifications, the main reason is now, $\grad_{\tphi}(x)$ may be constrained to a hypersurface. To see this, assume there is an exponent $p>1$ such that $F(x) = \tphi^p$ has invertible gradient with the smoothness assumption as in \cref{ex:homogenousgeq1}. We have $\grad_{\tphi}(x) = \frac{1}{p}F(x)^{\frac{1-p}{p}}\grad_F(x)$, thus, $z=\grad_{\tphi}(x)$ satisfies
  $\grad_F^{-1}(pz) = F(x)^{-\frac{1}{p}}x$, or $F(\grad_F^{-1}(p z)) = 1$. In particular, this holds for $F(x) = |x|_p^p$.

The index $p$ such that $\grad_{\tphi^p}$ is invertible is not unique. The following result seems dependent on $p$, but since $\tphi^{\bar{\bcc}\bcc}$ is not dependent on $p$, $\tphi^{\bar{\bcc}}$ is also not dependent on $p$.
  \begin{proposition}Let $\tphi(x)$ be an absolutely-homogeneous, convex function of order $1$ and assume $F: x\mapsto \tphi(x)^p$  ($F$ is convex and absolutely-homogeneous of order $p$) satisfies the smoothness and invertibility of $\grad_F$ as in \cref{ex:homogenousgeq1}. Then
    \begin{equation}\tphi^{\bar{\bcc}}(\barx) =\begin{cases}-\bu(x_{\opt, \barx}^{\sfT}\barx) - \tphi(x_{\opt, \barx}) \text{ if } F(\grad_F^{-1}(\barx)) \geq (\frac{2r(-p_0p_2)^{\frac{1}{2}}}{p})^{\frac{p}{p-1}},\\
 -\bu(0)\text{ otherwise.}
\end{cases}\label{eq:phicorder1}
    \end{equation}
where in the first case,  $x_{\opt, \barx}= \frac{(4r^2p_0p_2p^{-2}F(\grad_F^{-1}(\barx))^{2/p-2} + 1)^{\frac{1}{2}}}{rF(\grad_F^{-1}(\barx))^{\frac{1}{p}}}\grad_F^{-1}(\barx)$; and
\begin{equation}\tphi^{\bar{\bcc}\bcc}(x) = \sup_{\barx}\{-\bu(x^{\sfT}\barx)-\tphi^{\bar{\bcc}}(\barx)\}=\begin{cases}  \tphi(x) \text{ if }\tphi(x) <\frac{1}{r},\\
\frac{1}{r}(\log(r\tphi(x)) + 1) \text{ otherwise}.
\end{cases}
\end{equation}
  \end{proposition}
  \begin{proof}For $\Delta > 0$, consider the hypersurface $F(x) = \Delta^p$, where $L(x, \barx)=\bu(x^{\sfT}\barx) +\tphi(x)$ restricts to $\bu(x^{\sfT}\barx) + \Delta$. Using Lagrange multipliers, the minimum point $x_{\opt,\Delta}$ is proprotional to $\grad_F(\barx)$, hence $x_{\opt,\Delta} = \frac{\Delta}{F(\grad_F^{-1}(\barx))^{\frac{1}{p}}}\grad_F^{-1}(\barx)$, similar to \cref{ex:homogenousgeq1}. Consider $f(\Delta) := \bu(x_{\opt,\Delta}^{\sfT}\barx) + \tphi(x_{\opt,\Delta})=\bu(pF(\grad_F^{-1}(\barx))^{\frac{p-1}{p}}\Delta) + \Delta $, with
    $$f'(\Delta) = -\frac{K}{r(\Delta^2K^2-4p_0p_2)^{\frac{1}{2}}} + 1,\quad K=
pF(\grad_F^{-1}(\barx))^{\frac{p-1}{p}}.
$$
Thus, $f'$ is increasing with $\Delta$. When $K^2 + 4r^2p_0p_2 =
p^2F(\grad_F^{-1}(\barx))^{\frac{2p-2}{p}}+4r^2p_0p_2>0$, then $f'(\Delta) = 0$ has one root corresponding to the global minimum of $L(., \barx)$, with the global $x_{\opt}$ in the proposition. Otherwise, $f'$ is positive in the region $\Delta \geq 0$, and the infimum of $f$ is at $\Delta=0$, or infimum of $L(., \barx)$ is at $x=0$. We thus obtain the expression for $\tphi^{\bar{\bcc}}$.

As before, $\bu(x^{\sfT}\barx) + \tphi(x) +\tphi^{\bar{\bcc}}(\barx) \geq 0$ for all $x, \barx$, thus $\tphi(x)\geq \sup_{\barx}\{-\bu(x^{\sfT}\barx) - \tphi^{\bar{\bcc}}(\barx)\}$. When $\tphi(x) < \frac{1}{r}$, then $\barx_{\opt, x} :=\frac{2r(-p_0p_2)^{\frac{1}{2}}}{(1-r^2\tphi(x))^2)^{\frac{1}{2}}}\grad_{\tphi}(x)$ satisfies
$$\begin{gathered}F(\grad_F^{-1}(\barx_{\opt, x})) =
\left(\frac{2r(-p_0p_2)^{\frac{1}{2}}}{(1-r^2\tphi(x)^2)^{\frac{1}{2}}}\right)^{\frac{p}{p-1}}F(\grad_F^{-1}(\grad_{\tphi}(x)))\\
=\left(\frac{2r(-p_0p_2)^{\frac{1}{2}}}{p(1-r^2\tphi(x)^2)^{\frac{1}{2}}}\right)^{\frac{p}{p-1}}F(x)^{-1}F(x) \geq \left(\frac{2r(-p_0p_2)^{\frac{1}{2}}}{p}\right)^{\frac{p}{p-1}}.
\end{gathered}$$
Thus, $\tphi^{\bar{\bcc}}(\barx_{\opt, x})$ is given by the first case of \cref{eq:phicorder1}, and we verify $x\mapsto \barx_{\opt,x}$ and $\barx\mapsto x_{\opt, \barx}$ are inverse maps between the regions $0 < \tphi(x) < \frac{1}{r}$ and $F(\grad_F^{-1}(\barx)) > (\frac{2r(-p_0p_2)^{\frac{1}{2}}}{p})^{\frac{p}{p-1}}$. In this case, the supremum is attainable, and $\tphi^{\bar{\bcc}\bcc}(x) = \tphi(x)$. When $\tphi(x) = 0$ then $x=0$ and the supremum is also attainable using any $\barx$ with $F(\grad_F^{-1}(\barx)) \leq (\frac{2r(-p_0p_2)^{\frac{1}{2}}}{p})^{\frac{p}{p-1}}$. 

Now assume $x^{\sfT}\grad_{\tphi}(x) \geq \frac{1}{r}$. Set $z=x_{\opt, \barx}$ then $\tphi^{\bar{\bcc}\bcc}(x)  =\max\{S_1, S_2\}$ with
\begin{align*}
S_1 &=  \sup_{\barx}\{-\bu(x^{\sfT}\barx)+\bu(0)\;|\; F(\grad_F^{-1}(\barx)) \leq (\frac{2r(-p_0p_2)^{\frac{1}{2}}}{p})^{\frac{p}{p-1}}\},\\
S_2 &=  \sup_{z: 0< \tphi(z) < \frac{1}{r} }\{-\bu(\frac{2r(-p_0p_2)^{\frac{1}{2}}}{(1-r^2\tphi(z)^2)^{\frac{1}{2}}}x^{\sfT}\grad_{\tphi}(z))
+ \bu(\frac{2r(-p_0p_2)^{\frac{1}{2}}}{(1-r^2\tphi(z)^2)^{\frac{1}{2}}}\tphi(z)) + \tphi(z) \}.
\end{align*}
We will see $S_1$ corresponds to the expression $S(0)$ below when we examine $S_2$, so let us start with $S_2$. We can proceed as before, on the hypersurface $\tphi(z) = \Delta$, or $F(z) = \Delta^p$, we want to maximize $x^{\sfT}\grad_{\tphi}(z)$. Set $z_{\Delta} = \frac{\Delta}{\tphi(x)}x$ in the hypersurface, then
$\tphi(z_{\Delta}) - \tphi(z) -(z_{\Delta} - z)^{\sfT}\grad_{\tphi}(z)>0$ unless $z=z_{\Delta}$ by convexity. This means $z_{\Delta}^{\sfT}\grad_{\tphi}(z) \leq z^{\sfT}\grad_{\tphi}(z)=\Delta$. Thus, 
$$x^{\sfT}\grad_{\tphi}(z) = \frac{\tphi(x)}{\Delta}z_{\Delta}^{\sfT}\grad_{\tphi}(z) \leq \frac{\tphi(x)}{\Delta}\Delta = \tphi(x),$$
and $x^{\sfT}\grad_{\tphi}(z)$ is maximized at $z=z_{\Delta}$. Hence. we  need to find the supremum of
\begin{equation}S(\Delta) = -\bu(\frac{2r(-p_0p_2)^{\frac{1}{2}}}{(1-r^2\Delta^2)^{\frac{1}{2}}}\tphi(x))
+ \bu(\frac{2r(-p_0p_2)^{\frac{1}{2}}}{(1-r^2\Delta^2)^{\frac{1}{2}}}\Delta) + \Delta \label{eq:SDelta}
\end{equation}
for $0 < \Delta \leq \frac{1}{r}$. Let $s_x =\frac{2r(-p_0p_2)^{\frac{1}{2}}}{(1-r^2\Delta^2)^{\frac{1}{2}}}\tphi(x)$, $s_z = \frac{2r(-p_0p_2)^{\frac{1}{2}}}{(1-r^2\Delta^2)^{\frac{1}{2}}}\Delta$, then
$$S(\Delta) = \frac{1}{r}\log\frac{s_x+(s_x^2-4p_0p_2)^{\frac{1}{2}}}{s_z+(s_z^2-4p_0p_2)^{\frac{1}{2}}} + \Delta
$$
which again approaches $\lim_{\Delta\to\frac{1}{r}}\frac{1}{r}\log\frac{2s_x}{2s_z} + \Delta =
\frac{1}{r}(\log(r\tphi(x)) + 1)$. It remains to show the limit as $\Delta$ goes to $0$ is less than this value, since we cannot have an interior maximum in the case $|\tphi(x)| \geq \frac{1}{r}$.

The limiting value $S(0)$ when $\Delta$ go to $0$ of \cref{eq:SDelta} is well-defined. We show it is also equals to $S_1$. To see this, as $-\bu$ is increasing, the supremum $S_1$ is attained at a point on the hypersurface $V$ defined by $F(\grad_F^{-1}(\barx)) = (\frac{2r(-p_0p_2)^{\frac{1}{2}}}{p})^{\frac{p}{p-1}}$. We verify the optimal point is $\barx_* := 2r(-p_0p_2)^{\frac{1}{2}}\grad_{\tphi}(x)$. First, it is in $V$ thanks to the identity $\grad_F^{-1}(p\grad_{\tphi}(x)) = F(x)^{-\frac{1}{p}}x$ earlier. Set
$$\hat{x} = \grad_F^{-1}(\barx_*) = \grad_F^{-1}(\frac{2r(-p_0p_2)^{\frac{1}{2}}}{p}p\grad_{\tphi}(x)   ) = (\frac{2r(-p_0p_2)^{\frac{1}{2}}}{p})^{\frac{1}{p-1}}F(x)^{-\frac{1}{p}}x.$$
For $\bar{z}$ in $V$, set $\hat{z} = \grad_F^{-1}(\bar{z})$, thus, $\barx_* = \grad_F(\hat{x})$ and $\bar{z} = \grad_F(\hat{z})$. We have $F(\hat{x}) - F(\hat{z}) \geq \grad_F(\hat{z})^{\sfT}(\hat{x} - \hat{z})$, which implies $\grad_F(\hat{z})^{\sfT}\hat{x} \leq \grad_F(\hat{z})^{\sfT} \hat{z}$. Therefore
$$\bar{z}^{\sfT}\left((\frac{2r(-p_0p_2)^{\frac{1}{2}}}{p})^{\frac{1}{p-1}}F(x)^{-\frac{1}{p}}x\right)\leq
\barx_*^{\sfT}\left((\frac{2r(-p_0p_2)^{\frac{1}{2}}}{p})^{\frac{1}{p-1}}F(x)^{-\frac{1}{p}}x\right)
$$
since $\grad_F(\hat{z})^{\sfT} \hat{z} =\frac{1}{p}F(\hat{z})=\frac{1}{p}F(\hat{x}) = \grad_F(\hat{x})^{\sfT} \hat{x}$. After simplification, this confirms $x^{\sfT}\bar{z}$ is maximized on $V$ at $\bar{z} = \barx_*$. Thus $S_1=S(0)$. Finally,
$$rS(0) = \log(r\tphi(x) + (r^2\tphi(x)^2+1)^{\frac{1}{2}})< \log(r\tphi(x)) + 1.$$
This is because $1+(t^2+1)^{\frac{1}{2}}$ is increasing for $t=(r\tphi(x))^{-1}\leq 1$ and $\log(1+2^{\frac{1}{2}}) < 1$. This confirms $\max\{S_1, S_2\}$ is $\frac{1}{r}(\log(r\tphi(x)) + 1)$.
\end{proof}    
\end{appendices}


\bibliographystyle{amsplain}
\bibliography{zeroCrossCurv}

\end{document}